\documentclass{article}
\usepackage[a4paper,margin=1in]{geometry}
\usepackage{graphicx,amsmath,amsfonts,amsthm,amssymb,mathtools,soul} 
\usepackage{xcolor} 
\usepackage{dsfont}
\usepackage{bm} 
\usepackage{enumerate} 
\usepackage{comment}
\usepackage{todonotes}
\usepackage{bbm}
\usepackage{graphicx}
\usepackage{hyperref}
\usepackage{accents}
\hypersetup{
    colorlinks=true,
    linkcolor=blue,
	linktocpage=true
}

\DeclareMathOperator*{\essinf}{ess\,inf}
\DeclareMathOperator*{\esssup}{ess\,sup}

\newtheorem{theorem}{Theorem}[section]
\newtheorem{definition}[theorem]{Definition}

\newtheorem{proposition}[theorem]{Proposition}
\newtheorem{lemma}[theorem]{Lemma}
\newtheorem{assumption}{Assumption}
\newtheorem{remark}[theorem]{Remark}
\newtheorem{corollary}[theorem]{Corollary}

\counterwithin{equation}{section}

\allowdisplaybreaks

\title{A new probabilistic approach for\\ mean field games of optimal stopping}

\author{
  Andrea Cosso\thanks{Università degli Studi di Milano, Milan, Italy; andrea.cosso@unimi.it.}
  \and
  Laura D'Andolfi\thanks{ENSAE Paris, CREST, Institut Polytechnique de Paris, Palaiseau, France; laura.dandolfi@ensae.fr.}
  \and
  Roxana Dumitrescu\thanks{ENSAE Paris, CREST, Institut Polytechnique de Paris, Palaiseau, France; roxana.dumitrescu@ensae.fr.}
}

\date{\today\thanks{A.~Cosso acknowledges support from GNAMPA-INdAM. R.~Dumitrescu and L.~D'Andolfi gratefully acknowledge financial support from the FIME Research Initiative. L.~D'Andolfi receives financial support for her PhD research from the Région Île-de-France.}}

\begin{document}

\maketitle

\abstract{We propose a novel probabilistic formulation for optimal stopping mean field games (OS-MFGs) with randomized strategies. We characterize mean field equilibria through a new class of coupled forward-backward systems, termed coupled reflected forward-backward McKean--Vlasov stochastic differential equations (MKV-RFBSDEs). An equilibrium is represented by a quintuple $(X,Y,Z,A,L)$, where $L$ is an adapted, $[0,1]$-valued, non-increasing càdlàg process representing the randomized stopping strategy. The optimality of randomized stopping strategies is characterized through two novel Skorokhod-type conditions involving $L$. This characterization is new even for classical optimal stopping problems without mean field interactions. We rigorously prove an equivalence between solutions of the MKV-RFBSDE system and OS-MFG equilibria in randomized strategies. We establish the existence of equilibria by applying the Kakutani--Fan--Glicksberg fixed-point theorem to a set-valued best-response correspondence, relying on new stability, compactness, and continuity results for the coupled MKV-RFBSDE system. We also prove uniqueness under suitable conditions. Under alternative monotonicity assumptions, we develop a new order-theoretic approach based on Tarski's fixed-point theorem, yielding the existence of extremal equilibria and constructive schemes for the minimal and maximal solutions. We further show that a mean field equilibrium induces an approximate Nash equilibrium for the associated $N$-player stopping game. Finally, we connect our probabilistic formulation with the analytical approach characterized by a coupled system of constrained partial differential equations.}

\vspace{5mm}

\noindent {\bf Keywords:} mean field games of optimal stopping; reflected backward stochastic differential equations; McKean-Vlasov equations; randomized stopping strategies.

\vspace{5mm}

\noindent {\bf Mathematics Subject Classification (2020):} 91A16, 60G40, 60H10.

\section{Introduction}

The aim of this paper is to develop a new probabilistic approach to
optimal stopping mean field games (OS-MFG), based on a novel coupled
system of reflected forward-backward McKean--Vlasov stochastic differential equations
(MKV-RFBSDEs).

Let us first consider a game with a large, but finite population of
$N$ players. For a finite measure $\nu$ and an integrable function
$h$, we use the notation
\[
\langle h,\nu\rangle
:=
\int h(x)\,\nu(dx).
\]
For each $i\in\{1,\ldots,N\}$, the state process
$X^{i,N}=(X_t^{i,N})_{t\in[0,T]}$ evolves according to
\[
\begin{aligned}
dX_t^{i,N}
={}&
b\big(
t,X_t^{i,N},
\big\langle\bar b(t,\cdot),m_t^N\big\rangle
\big)dt
+
\sigma\big(
t,X_t^{i,N},
\big\langle\bar\sigma(t,\cdot),m_t^N\big\rangle
\big)dW_t^i;\,\, \qquad X^{i,N}_0=\xi^{i},
\end{aligned}
\]
where $(\xi^1,\xi^2,\ldots,\xi^N)$ are independent and identically distributed random variables, and $(W^1,\ldots,W^N)$ are independent Brownian motions. The flow of empirical measures $(m^N_t)_{t \in [0,T]}$ given by
\[
m_t^N(dx)
=
\frac1N
\sum_{k=1}^N
\delta_{X_t^{k,N}}(dx)\mathbf 1_{t<\tau^k}
\]
is the empirical occupation measure of the players who have not yet
stopped. In particular,
\[
\big\langle\bar b(t,\cdot),m_t^N\big\rangle
=
\frac1N
\sum_{k=1}^N
\bar b(t,X_t^{k,N})\mathbf 1_{t<\tau^k},
\]
and similarly for the interaction entering the diffusion coefficient.

Each player $i$ chooses an admissible stopping time $\tau^i$ with
values in $[0,T]$ in order to maximize
\[
\mathbb E\bigg[
\int_0^{\tau^i}
f\big(
t,X_t^{i,N},
\big\langle\bar f(t,\cdot),m_t^N\big\rangle
\big)dt
+
g\big(
\tau^i,
X_{\tau^i}^{i,N},
\big\langle\bar g(\tau^{i},\cdot),\mu^N\big\rangle
\big)
\bigg],
\]
where
\[
\mu^N(dt,dx)
=
\frac1N
\sum_{k=1}^N
\delta_{(\tau^k,X_{\tau^k}^{k,N})}(dt,dx)
\]
is the empirical joint distribution of the stopping times and exit
states, and
\[
\big\langle\bar g(t,\cdot),\mu^N\big\rangle
=
\frac1N
\sum_{k=1}^N
\bar g(t,\tau^k,X_{\tau^k}^{k,N}).
\]
Since the dynamics and objective
functionals are coupled through the empirical measures
$(m_t^N)_{t\in[0,T]}$ and $\mu^N$, it is natural to look for a Nash
equilibrium.
As the number of players tends to infinity, we expect, by a
\textit{propagation-of-chaos} argument, that the empirical occupation measures
converge to a deterministic flow of subprobability measures
$(m_t)_{t\in[0,T]}$, while the empirical joint distributions of the
stopping times and exit states converge to a deterministic probability
measure $\mu$.
The limiting MFG problem may first be formulated in pure strategies as follows.
Let $\mathcal T$ denote the set of admissible stopping times with
values in $[0,T]$.  Given a deterministic mean field environment
$\bigl((m_t)_{t\in[0,T]},\mu\bigr)$, the state process of the
representative player satisfies
\[
\begin{aligned}
dX_t^m
={}&
b\left(
t,X_t^m,
\left\langle\bar b(t,\cdot),m_t\right\rangle
\right)dt
+
\sigma\left(
t,X_t^m,
\left\langle\bar\sigma(t,\cdot),m_t\right\rangle
\right)dW_t;\,\,\qquad X_0^{m}=\xi,
\end{aligned}
\]
and the representative player solves
\begin{equation}
\label{eq:SDE limit stopping}
\sup_{\tau\in\mathcal T}
\mathbb E\left[
\int_0^\tau
f\left(
t,X_t^m,
\left\langle\bar f(t,\cdot),m_t\right\rangle
\right)dt
+
g\left(
\tau,X_\tau^m,
\left\langle\bar g(\tau,\cdot),\mu\right\rangle
\right)
\right].
\end{equation}
A mean field equilibrium in pure strategies is a stopping time
$\tau^\star\in\mathcal T$, together with an environment
$\bigl((m_t^\star)_{t\in[0,T]},\mu^\star\bigr)$, such that
$\tau^\star$ solves \eqref{eq:SDE limit stopping} under
$(m^\star,\mu^\star)$ and
\[
m_t^\star(B)
=
\mathbb P\big(
X_t^{m^\star}\in B,\,
t<\tau^\star
\big),
\qquad
B\in\mathcal B(\mathbb R),
\quad t\in[0,T],
\]
and
\[
\mu^\star
=
\mathcal L\big(
\tau^\star,
X_{\tau^\star}^{m^\star}
\big).
\]
Thus, at equilibrium, the mean field environment faced by the
representative player is generated by the optimal stopping time
itself. Since the equilibrium strategy is associated to an ordinary stopping time,
this formulation describes an equilibrium in pure strategies.
Such pure-strategy equilibria, however, need not exist under general
assumptions. It is therefore natural to enlarge the set of admissible
strategies by allowing randomized stopping strategies. In this paper,
a randomized stopping strategy is represented by an adapted,
right-continuous and non-increasing survival process
$L=\{L_t\}_{t\in[0,T]}$, taking values in $[0,1]$ and satisfying
\[
L_{0^-}=1,
\qquad
L_T=0.
\]
More precisely, on an extension of the probability space carrying an
independent uniform random variable $U$, the process $L$ induces the
randomized stopping time
\[
\tau^L
=
\inf\{t\in[0,T]:L_t\leq U\},
\]
and
\[
L_t
=
\mathbb P\big(\tau^L>t\,\big|\,\mathcal F_t\big),
\]
where $\{\mathcal{F}_t\}_{t \in [0,T]}$ is the filtration generated by $\xi$ and $W$.
Hence, $L_t$ represents the conditional survival probability of the
representative player at time $t$. At the population level, the aggregate surviving mass is
$\mathbb E[L_t]$, and occupation-type mean field quantities are
obtained by weighting the state of the representative player by
$L_t$. An ordinary stopping time $\tau$ is recovered as the pure
strategy
\[
L_t^\tau
=
\mathbf 1_{t<\tau}.
\]
Randomized stopping strategies have been extensively studied in the
optimal stopping literature (see, for instance,
\cite{meyer2006convergence,el1992probabilistic,
touzi2002continuous,bismut1979temps}).

The main novelty of our approach
consists in introducing the following new coupled system of reflected forward-backward McKean--Vlasov stochastic equations, in which the
equilibrium survival process is determined directly as part of the
solution:
\begin{align}
\label{sys}
\begin{cases}
\displaystyle
X_t
=
X_0
+
\int_0^t
b\!\left(
s,X_s,
\mathbb E\!\left[\bar b(s,X_s)L_s\right]
\right)ds
+
\int_0^t
\sigma\!\left(
s,X_s,
\mathbb E\!\left[\bar\sigma(s,X_s)L_s\right]
\right)dW_s,
\\[3mm]
\displaystyle
Y_t
=
\zeta
+
\int_t^T
f\!\left(
s,X_s,
\mathbb E\!\left[\bar f(s,X_s)L_s\right]
\right)ds
+
A_T-A_t
-
\int_t^T Z_s\,dW_s,
\qquad 0\leq t\leq T,
\\[3mm]
Y_t\geq\xi_t,
\qquad 0\leq t\leq T,
\\[2mm]
\displaystyle
\int_0^T (Y_t-\xi_t)\,dA_t=0,
\\[3mm]
\displaystyle
\int_{0^-}^T (Y_t-\xi_t)\,dL_t=0,
\\[3mm]
\displaystyle
\int_{0^-}^T A_t\,dL_t=0,
\end{cases}
\end{align}
with $\xi_t:=g(t,X_t, \mathbb{E}\left[\int_{0^-}^T\bar{g}(t,s,X_s)d(-L_s)\right] )$ and $\zeta:=\xi_T$. In this formulation, the equilibrium survival process $L$ is computed
jointly with the state process $X$ and the reflected backward
components $(Y,Z,A)$, rather than being recovered indirectly from a
fixed point for the associated flow of measures. The full list of
advantages of this formulation, together with our detailed results,
is presented and discussed below, after a brief literature review.
\paragraph{Literature review.} Mean field game (MFG) theory was independently introduced by Lasry and Lions in \cite{lasry2007mean} and Huang, Malhamé and Caines in \cite{huang2006large}, to provide a tractable framework for analyzing Nash equilibria in large-scale multi-agent systems. The theory addresses the limiting case where the number of players $N$ tends to infinity and agents interact symmetrically through their empirical distribution.\\
In the standard MFG model, given a flow of probability measures $\mu=(\mu_t)_{t \in [0,T]}$ representing the spatial distribution of the population, the representative player faces a classical stochastic control problem. The existence and uniqueness of the resulting equilibrium have been extensively investigated through two primary approaches. In the analytical theory, the system is characterized by a coupling of nonlinear partial differential equations (PDEs): a backward Hamilton--Jacobi--Bellman (HJB) equation, which determines the value function and the agent's optimal strategy, and a forward Kolmogorov-type (or Fokker--Planck, FP) equation, which describes the evolution of the population distribution. For a comprehensive overview of this approach, we refer the reader to the notes by Cardaliaguet \cite{cardaliaguet2010notes}, based on the lectures of P.-L.~Lions at the Collège de France \cite{lions2007theorie}. Alternatively, the probabilistic approach consists in describing the MFG equilibria through a coupled Forward-Backward Stochastic Differential Equation (FBSDE) system of McKean--Vlasov type (see \cite{carmona2018probabilisticI, carmona2018probabilisticII} for a detailed account).  Driven by the development of mean field game theory, considerable attention has also been devoted to McKean--Vlasov stochastic differential equations. These equations are characterized by coefficients that depend not only on the current state of the system, but also on its probability law. Originally introduced by McKean in \cite{mckean1966class}, they provide a natural probabilistic description of the limiting behavior of large systems of weakly interacting particles, in which the dynamics of each individual are influenced by the collective distribution of the population. We refer to \cite{bensoussan2013mean} for a comprehensive treatment of this class of equations.

A distinct, although closely related, line of research concerns the control of McKean--Vlasov dynamics, commonly referred to as Mean Field Control (MFC). In this framework, a central planner chooses a control in order to optimize a criterion that depends on the state process and its distribution. This differs from the mean field game setting, where each agent independently optimizes against a given population distribution (see \cite{carmona2013control} for a detailed comparison between MFGs and control problems for McKean--Vlasov SDEs).

More recently, the theory has been extended to models in which the coefficients and objective functionals depend not only on the distribution of the state, but also on the distribution of the controls. Such extended mean field control problems were first studied in the linear--quadratic setting in \cite{basei2017linear,graber2016linear}. A more general formulation was subsequently developed in \cite{acciaio2019extended}, where necessary and sufficient conditions for a Pontryagin maximum principle were established, together with a connection between the weak formulation of the control problem and optimal transport on path space.

While many results have been proved in the case of MFG with regular controls, OS-MFG represents a new trend in the literature. In this setting, the interaction is mediated
by the distribution of the residual population, namely the agents
who have not yet stopped, and possibly by the distribution of their
exit times and exit states. Despite the apparent simplicity of the
individual decision problem, the equilibrium analysis is delicate.
In particular, equilibria in pure stopping strategies may fail to
exist, and the flow of occupation measures may be discontinuous in
time. As documented in game theory, the existence of pure Nash equilibria (i.e., equilibria where all agents use non-randomized strategies) is generally not guaranteed. Consequently, players are typically forced to adopt randomized (or mixed) stopping strategies to ensure the existence of an equilibrium.\\
The literature on OS-MFG has primarily evolved along various approaches. Notably, \cite{bertucci2018optimal} makes several significant contributions to this field. By considering a state process with constant coefficients evolving in a bounded domain, the author characterizes the equilibrium through an analytical approach which consists in solving a coupled system of a variational inequality with a Fokker--Planck equation. In particular, he provides an example of the non-existence of Nash equilibria in pure strategies and he introduces the formal notion of mixed solutions within the framework. It is important to emphasize that, in the context of OS-MFG, the flow of measures may exhibit discontinuities, which makes the analytical treatment, and in particular, the proofs of existence more challenging.\\
In \cite{nutz2018mean} a tractable model for OS-MFG is formulated to directly investigate the properties of equilibria. In this framework, to prevent the possibility of the entire population stopping at the same time, a parameter is incorporated representing the heterogeneity in subjective risk perception, for which an explicit form of the equilibria is obtained. In a more general framework, \cite{carmona2017mean}
considers MFGs of timing, whose formulation is motivated by a dynamic model of bank runs in a continuous-time setting. In this paper, the authors adopt a purely probabilistic approach, but their techniques are more closely related to the weak convergence and compactness arguments \cite{carmona2016mean, lacker2016general, lacker2015mean} than to the probabilistic approach originated in \cite{carmona2013probabilistic}. 
Recent research has further advanced the field by developing linear programming approaches to mean field games involving regular control and optimal stopping (see, in particular, \cite{dumitrescu2021control,bouveret2020mean,dumitrescu2023linear,guo2024mf,guo2025continuous}). Relaxed controls are widely used to compactify stochastic control problems and thereby facilitate the proof of existence of solutions. Two main approaches to such relaxed formulations have been developed. The first, based on the controlled martingale problem, was introduced for mean field games with regular controls in \cite{lacker2015mean}. The second, initially developed in the context of optimal stopping mean field games in \cite{bouveret2020mean}, relies on a formulation in terms of the occupation measure of a process killed at the stopping time. Using this occupation-measure formulation, the authors of \cite{bouveret2020mean} establish a rigorous connection between relaxed Nash equilibria and the notion of mixed solutions introduced in \cite{bertucci2018optimal}.
Further important contributions to the study of OS-MFGs using alternative frameworks can be found in \cite{ferrari2025existence,he2025mean,dianetti2023unifying,dianetti2025entropy,possamai2025mean}.\\
Finally, although our work focuses on the non-cooperative game framework, we mention for completeness that the cooperative counterpart, known as Mean Field Control of Optimal Stopping (MFC-OS), has also recently started to attract attention (see, e.g., \cite{cardaliaguet2026mean, cosso2025mean,talbi2023dynamic, djehiche2019mean, djehiche2025propagation,djehiche2026zero}).
\paragraph{Our contributions.} 
In classical optimal-stopping theory, it is well known that the value process of an optimal stopping problem can be characterized as the solution of a reflected backward stochastic differential equation (RBSDE) (see, for instance, \cite{el1981aspects}). Despite the substantial progress made in the theory of mean field games with optimal stopping (OS-MFG), no analogous probabilistic characterization has so far been established for mean field equilibria in randomized stopping strategies. In particular, the connection between such equilibria and an appropriate class of reflected BSDEs remains unexplored. More generally, a characterization of randomized OS-MFG equilibria through a coupled forward-backward stochastic differential system is still missing.

Such a formulation would extend to mean field games with optimal stopping the probabilistic approach initiated in \cite{carmona2013probabilistic} for mean field games with regular controls. In the classical MFG framework, the correspondence between the analytical Hamilton--Jacobi--Bellman--Fokker--Planck system and its probabilistic FBSDE counterpart is well established and provides complementary perspectives for both theoretical analysis and numerical approximation. The first contribution of this paper fills this gap by introducing a probabilistic formulation of OS-MFG equilibria in randomized stopping strategies based on a new McKean--Vlasov reflected BSDE system.

Our contributions are the following:

\begin{itemize}

\item[(i)] We introduce a new probabilistic formulation of mean field games with optimal stopping based on a McKean--Vlasov reflected backward stochastic differential equation  coupled with the forward state dynamics (see \eqref{sys}). A solution to this system is defined as a quintuple
\[
(X,Y,Z,A,L),
\]
satisfying the coupled forward-backward system given by \eqref{MFG SDE} and \eqref{MFRBSDE}. Here, \(X\) denotes the state process, \(Y\) is the backward value process, \(Z\) is the martingale-representation component, \(A\) is the continuous non-decreasing reflection process that keeps \(Y\) above the obstacle \(\xi\), and \(L\) is a non-increasing process representing a randomized stopping strategy.

A distinctive feature of this formulation is that the randomized stopping strategy \(L\) is an endogenous component of the solution. Consequently, the fixed-point problem associated with the mean field interaction is formulated directly on the space of randomized stopping strategies. We establish the existence of solutions to the resulting MKV-RFBSDE system through two distinct approaches (via \textit{Kakutani} and \textit{Tarski's Theorems}), each applying under a different set of assumptions. The proofs of existence rely on new technical developments using BSDE theory and functional analysis tools.

Within the first set of assumptions allowing us to apply Kakutani's fixed-point theorem,  we also prove uniqueness under a Lasry--Lions monotonicity condition. Under the second set of monotonicity assumptions, which allows us to apply Tarski’s fixed-point theorem, we construct iterative learning schemes that converge to the minimal and maximal solutions of the system. We then provide sufficient conditions ensuring the uniqueness of the process $L$ satisfying the Skorokhod-type conditions. Under these conditions, the set of solutions to the system also enjoys a complete lattice structure.

\item[(ii)] We derive new necessary and sufficient optimality conditions for randomized stopping strategies. More precisely, optimality is characterized by the following two Skorokhod-type conditions:
\[
\int_{0^-}^T (Y_t-\xi_t)\,dL_t=0,
\qquad
\int_{0^-}^T A_t\,dL_t=0.
\]
Since \(Y\geq \xi\) and \(L\) is non-increasing, the first condition implies that, for almost every \(\omega\), the measure \(-dL(\omega)\) is supported on the contact set
\[
\left\{
t\in[0,T]:
Y_t(\omega)=\xi_t(\omega)
\right\}.
\]
Thus, stopping probability can be assigned only at times when the
value of continuing coincides with the stopping payoff.

Similarly, the second condition implies that \(-dL(\omega)\) is supported on
\[
\left\{
t\in[0,T]:
A_t(\omega)=0
\right\}.
\]
Hence, no stopping probability can be assigned once the reflection process has become active. Equivalently, the representative player can no longer stop strictly after the maximal stopping time
\[
\tau^{\max}
:=
\inf\left\{
t\geq 0:A_t>0
\right\}.
\]

These conditions also yield new results for classical optimal stopping problems without mean field interactions. While optimality conditions for stopping times, corresponding to pure strategies, are well established (see, for instance, \cite{el2006aspects}), analogous conditions for randomized stopping strategies appear to be missing from the literature. We further establish additional structural properties of optimal randomized stopping strategies.

\item[(iii)] We establish a complete equivalence between equilibria of the optimal-stopping mean field game in randomized strategies and solutions of the probabilistic MKV-RFBSDE system \eqref{sys}. In particular, every randomized mean field equilibrium induces a solution of the system, and conversely, every solution of the system yields a randomized mean field equilibrium. This result shows that the proposed MKV-RFBSDE system provides a full probabilistic characterization of the equilibrium problem.

\item[(iv)] We establish a rigorous connection between the probabilistic formulation developed in this paper and the analytic formulation based on mixed solutions introduced in \cite{bertucci2018optimal}. This result provides a bridge between the MKV-RFBSDE characterization and the corresponding analytic approach, thereby clarifying the relationship between the two notions of solution.

\item[(v)] Starting from a solution of the mean field system \eqref{sys}, we construct an approximate equilibrium for the associated \(N\)-player optimal-stopping game. This result provides a rigorous justification of the mean field model as an approximation of strategic interactions in large but finite populations.

\end{itemize}

This probabilistic approach offers several advantages over existing  formulations. First, in contrast to the analytical approach proposed in \cite{bertucci2018optimal} and the linear-programming formulations developed in \cite{bouveret2020mean}, \cite{dumitrescu2021control}, and \cite{dumitrescu2023linear}, our approach directly identifies an equilibrium optimal stopping strategy through the process $L$, which is itself part of the solution. Second, it naturally accommodates degenerate diffusion coefficients and extends to path-dependent, non-Markovian settings, thereby allowing the representative player’s strategy to depend on the entire history of the state process (see Remark~\ref{R:non-Markov}). Finally, when existence is established by means of Tarski’s fixed-point theorem, the properties of reflected BSDEs provide an alternative route for proving the monotonicity of the relevant correspondences, without relying on Topkis’s theorem as in the existing literature (see Lemma~\ref{gamma2}).
\paragraph{Organization of the paper.} The paper is organized as follows. In Section \ref{S:Kakutani}, we establish the existence of an equilibrium by applying the Kakutani--Fan--Glicksberg fixed-point theorem to the set of randomized strategies; we also address the uniqueness of an equilibrium. Section \ref{S:ProcessL} is devoted to deriving key properties of a stopping strategy $L$ satisfying the two novel Skorokhod-type conditions. In Section \ref{S:Tarski}, we introduce an alternative set of assumptions to prove existence via Tarski's fixed-point theorem, and provide sufficient conditions to ensure the uniqueness of the optimal best response. This approach is based on the lattice structure of the set of processes $L$. Furthermore, we construct learning algorithms for the \textit{minimal} and \textit{maximal} solution of the MKV-RFBSDE system \eqref{sys} for which rigorous convergence proofs are provided. Section \ref{S:Connection} is devoted to exploring the connection between the solutions of our new system and equilibria for OS-MFG in randomized strategies. In Section \ref{S:Approximate}, we prove the approximate Nash property for the $N$-player game. Finally, Section \ref{S:PDE} establishes the formal equivalence between this novel probabilistic formulation based on MKV-RFBSDEs and the analytical PDE framework introduced in \cite{bertucci2018optimal}.

\section{Existence via the Kakutani--Fan--Glicksberg fixed-point theorem}
\label{S:Kakutani}
In this section, we establish our first existence result for a solution to the MKV-RFBSDE system \eqref{MFG SDE}-\eqref{MFRBSDE}. To this end, we first introduce the probabilistic framework, notation, and assumptions. 
\paragraph{Probabilistic setting, notation and assumptions.} Let $W=\{W_t\}_{ t \in [0,T]}$ be a standard $m$-dimensional Brownian motion defined on a complete probability space $(\Omega, \mathcal {F}, \mathbb{P})$. Let also $X_0\colon\Omega\rightarrow\mathbb R^d$ be a random variable independent of $W$ and satisfying $\mathbb E[|X_0|^q]<\infty$, for some $q\geq4$. We denote by $\{\mathcal{F}_t\}_{ t \in [0,T]}$ the filtration generated by $X_0$ and by the Brownian motion $W$, augmented with all $\mathbb{P}$-null sets of $\mathcal {F}$. We then introduce the following sets.
\begin{itemize}
    \item $\mathcal{L}^p(\mathcal{F}_t)$, for all $t\in[0,T]$, $p\in[1,\infty)$, is the set of (equivalence classes of) $\mathcal{F}_t$-measurable real-valued random variables $\zeta$ such that $\mathbb{E}[|\zeta|^p]<\infty$;
    \item $\mathbb{H}^2(\mathbb R^k)$, or simply $\mathbb H^2$, is the set of (equivalence classes of) progressively measurable $\mathbb R^k$-valued stochastic processes $Z=\{Z_t\}_{t \in [0,T]}$ such that $\mathbb{E}\big[\int_0^T|Z_t|^2dt\big]<\infty$;
    \item $\mathbb{S}^2(\mathbb R^k)$, or simply $\mathbb S^2$, is the set of (equivalence classes of) adapted $\mathbb R^k$-valued continuous stochastic processes $Y=\{Y_t\}_{t \in [0,T]}$  such that $\mathbb{E}\big[\sup_{t \in [0,T]} |Y_t|^2\big]<\infty$;
    \item $\mathbb{A}^2\subset\mathbb S^2(\mathbb R)$ is the set of (equivalence classes of) adapted real-valued non-decreasing continuous stochastic processes $A=\{A_t\}_{t\in[0,T]}$  such that $A_0=0$ almost surely;
    \item $\mathcal{V} \subset \mathbb H^2(\mathbb R)$ is the set of equivalence classes of stochastic processes for which there exists a representative process $L = \{L_t\}_{t \in [0,T]}$ which is $[0,1]$-valued, adapted, non-increasing, c\`adl\`ag, and $L_T=0$; we also set $L_{0^-}:=1$. Throughout the paper, when we consider $L\in\mathcal V$ we always refer to such a representative process.
\end{itemize}
We consider measurable functions $b:[0,T]\times\mathbb{R}^d\times\mathbb R^k\to \mathbb{R}^d$, $\bar b\colon[0,T]\times\mathbb R^d\rightarrow\mathbb R^k$, $\sigma:[0,T]\times\mathbb{R}^d\times\mathbb R^k\to \mathbb{R}^{d\times m}$, $\bar\sigma\colon[0,T]\times\mathbb R^d\rightarrow\mathbb R^k$, $f\colon[0,T]\times\mathbb{R}^d\times\mathbb{R}^k \to \mathbb{R}$, $\bar f\colon[0,T]\times\mathbb R^d\rightarrow\mathbb R^k$, $g\colon\mathbb R^d\times\mathbb R\rightarrow\mathbb R$, $h\colon[0,T]\times\mathbb R^d\times\mathbb R\rightarrow\mathbb R$, $\phi\colon[-T,T]\rightarrow\mathbb R$, with $g(x,w)=h(T,x,w)$, for every $(x,w)\in\mathbb R^d\times\mathbb R$, on which we impose the following assumptions.
\begin{assumption}
\label{assumption:1A}\quad
\begin{enumerate}[1)]
    \item There exists a constant $K\geq0$ such that
    \begin{align*}
        |b(t,x,m)-b(t,x',m')| + |\sigma(t,x,m)-\sigma(t,x',m')|&\leq K\big(|x-x'|+|m-m'|\big), \\
        |\bar b(t,x)-\bar b(t,x')| + |\bar \sigma(t,x)-\bar \sigma(t,x')|&\leq K|x-x'|, \\
        |b(t,0,0)| + |\sigma(t,0,0)| + |\bar b(t,0)| + |\bar\sigma(t,0)| &\leq K,       
    \end{align*}
    for all $t\in[0,T]$, $(x,m),(x',m')\in\mathbb R^d\times\mathbb R^k$.
\end{enumerate}
\end{assumption}

\begin{assumption}
\label{assumption:1B}\quad
\begin{enumerate}[1)]
    \item There exist constants $K\geq0$ and $r \in [1,q/2)$ such that
    \begin{align*}
        |f(t, x, m')-f(t,x, m)|&\leq K|m'-m|, \\
        |f(t,x,0)| + |h(t,x,0)| + |\bar f(t,x)| &\leq K(1+|x|^r), \\
        |h(t,x,w) - h(t,x,w')| &\leq K |w-w'|, 
    \end{align*}
        for all $(t,x) \in [0,T]\times\mathbb R^d$, $(m,w),(m',w') \in \mathbb{R}^k\times\mathbb R$.
    \item $h$ is a jointly continuous function; $\phi$ is a function of class $\mathcal C^1([-T,T])$.
\end{enumerate}
\end{assumption}
\begin{assumption}\quad 
    \label{assumption: 1C}\begin{enumerate}
  \item $\bar b$, $\bar\sigma$, $\bar{f}$ are functions of class $\mathcal C^{1,2}([0,T] \times \mathbb R^d)$, whose derivatives satisfy the following polynomial growth condition: there exist constants $K\geq0$ and $p'\in[2,q/2]$ such that (here we write $\psi$ for $\bar b$, $\bar\sigma$, $\bar f$)
    \begin{equation}\label{PolGrowthCond}
    \begin{split}
        &|\partial_t \psi(t,x)|\leq K(1 + |x|^{p'}),  \\
        &|\partial_x \psi(t,x)|\leq K(1 + |x|^{p'-1}), \qquad |\partial_{xx}^2 \psi(t,x)| \leq K(1 + |x|^{p'-2}),
    \end{split}
    \end{equation}
    for all $(t,x)\in[0,T]\times\mathbb R^d$.
    \item $f=f(t,x,m)$ is continuous in $x$, uniformly with respect to $(t,m) \in [0,T] \times \mathbb R^k$.\end{enumerate}
\end{assumption}
\paragraph{Well-posedness of the forward McKean-Vlasov stochastic differential equation and the RBSDE.}
For a given $L\in\mathcal V$, we recall here the standard well-posedness results for the forward McKean--Vlasov SDE and the RBSDE involved in the system. To be precise,   consider the following McKean--Vlasov stochastic differential equation on $[0,T]$:
\begin{equation}
\label{MFG SDE}
X_t=X_0+\int_0^tb\big(s,X_s,\mathbb E\big[\bar b(s,X_s)L_s\big]\big)ds+\int_0^t\sigma\big(s,X_s,\mathbb E\big[\bar\sigma(s,X_s)L_s\big]\big)dW_s.
\end{equation}
We have the following proposition.
\begin{proposition}[\textit{Well-posedness of the forward McKean--Vlasov stochastic differential equation}]\label{P:SDE}
Suppose that Assumption \ref{assumption:1A} holds. Given $L\in\mathcal V$, there exists a unique $\mathbb{R}^d$-valued continuous stochastic process $X=(X_t)_{t \in [0,T]}$ solving equation \eqref{MFG SDE}. Moreover, $X$ satisfies, for every $2\leq p\leq q$,
\begin{equation}\label{EstimateSDE}
\mathbb E\Big[\sup_{0\leq t\leq T}|X_t|^p\Big] \leq C_p\big(1+\mathbb E[|X_0|^p]\big),
\end{equation}
for some constant $C_p\geq0$, depending on $p$, $T$, $K$, but independent of $L$. Finally, let $L'\in\mathcal V$ and denote by $X'$ the solution of equation \eqref{MFG SDE} with $L'$ in place of $L$. Then, it holds that, for every $2\leq p\leq q$,
\begin{align}\label{EstimateLL'}
\mathbb E\Big[\sup_{0\leq t\leq T}|X_t-X_t'|^p\Big] &\leq C_p\int_0^T\Big|\mathbb E\big[\bar b(t,X_t)L_t\big] - \mathbb E\big[\bar b(t,X_t)L_t'\big]\Big|^pdt \\
&\quad + C_p\int_0^T\Big|\mathbb E\big[\bar\sigma(t,X_t)L_t\big] - \mathbb E\big[\bar\sigma(t,X_t)L_t'\big]\Big|^pdt, \notag
\end{align}
for some constant $C_p\geq0$, depending on $p$, $T$, $K$, but independent of $L$ and $L'$.
\end{proposition}
\begin{proof}
Existence and uniqueness follow from \cite{carmona2018probabilisticI}, while estimates \eqref{EstimateSDE} and \eqref{EstimateLL'} can be proved proceeding along the same lines as for classical non-McKean--Vlasov stochastic differential equations, using Burkholder--Davis--Gundy and Gronwall inequalities.
\end{proof}

\noindent We are interested in \eqref{MFG SDE} and the following reflected forward-backward McKean--Vlasov stochastic differential equation on $[0,T]$ in the unknowns $(X,Y,Z,A,L)$:
\begin{equation}\label{MFRBSDE}
\begin{cases}
\vspace{2mm}Y_t=\zeta+\int_t^T f\big(s,X_s,\mathbb E\big[\bar f(s,X_s)L_s\big]\big)ds + A_T-A_t - \int_t^T(Z_s,dW_s), \quad 0\leq t\leq T, \\
\vspace{2mm}Y_t \geq \xi_t, \quad 0\leq t\leq T, \\
\vspace{2mm}\int_{0}^T (Y_t-\xi_t) dA_t=0, \\
\vspace{2mm}\int_{0^-}^T(Y_t-\xi_t) dL_t=0, \\
\int_{0^-}^T A_t dL_t=0,
\end{cases}
\end{equation}
where 
\[
\zeta:=g\bigg(X_T,\mathbb E\bigg[\int_{0^-}^T\phi(T-s)dL_s\bigg]\bigg), \qquad \xi_t := h\bigg(t,X_t,\mathbb E\bigg[\int_{0^-}^T\phi(t-s)dL_s\bigg]\bigg), \quad 0\leq t\leq T.
\]
For integrals from $0^-$ to $T$ with respect to the process $L$, we assume that each integrand process takes the same value at $0^-$ as it does at $0$.
\begin{remark}
    Since $L$ is a non-increasing survival process, the signed measure $dL$ is non-positive. For notational convenience, throughout the remainder of the paper, we write integrals with respect to $dL$ rather than with respect to the positive measure $d(-L)$. Accordingly, the corresponding sign convention is understood to be incorporated into the definition of $\phi$.
\end{remark}
Our aim is to study the existence of a quintuple $(X,Y,Z,A,L)\in\mathbb S^2\times\mathbb S^2\times\mathbb H^2\times\mathbb A^2\times\mathcal V$ solution to \eqref{MFG SDE} and to the above McKean--Vlasov reflected backward stochastic differential equation \eqref{MFRBSDE} (Theorem \ref{thm:main}). To this end, for every fixed $L\in\mathcal V$, consider the following classical reflected backward stochastic differential equation:
\begin{equation}\label{RBSDE}
\begin{cases}
\vspace{2mm}Y_t=\zeta+\int_t^Tf\big(s,X_s,\mathbb{E}\big[\bar f(s,X_s)L_s\big]\big)ds+A_T-A_t-\int_t^T(Z_s, dW_s), \quad 0 \leq t\le T,\\
\vspace{2mm}Y_t \ge \xi_t, \quad 0 \leq t\le T,\\
\int_0^T(Y_t-\xi_t)dA_t=0.
\end{cases}
\end{equation}

\begin{proposition}[\textit{Well-posedness of the RBSDE}]\label{P:ExistUniq}
Suppose that Assumptions \ref{assumption:1A} and \ref{assumption:1B} hold. For every $L\in\mathcal V$, there exists a unique triplet $(Y,Z,A)\in\mathbb S^2\times\mathbb H^2\times\mathbb A^2$ satisfying \eqref{RBSDE}. 
\end{proposition}
\begin{proof}
    See Theorem 5.2 in \cite{el1981aspects}.
\end{proof}
\subsection{Existence of a solution of the MKV-RFBSDE system \eqref{MFG SDE}-\eqref{MFRBSDE}}
We now introduce the set-valued map $\Gamma\colon\mathcal V\rightarrow 2^{\mathcal V}$, which is defined as follows: given $L\in\mathcal V$ then $\hat L\in\Gamma(L)$ if $\hat L\in\mathcal V$ and it holds that
\begin{equation}\label{hatL}
\int_{0^-}^T (Y^L_t-\xi^L_t)d\hat L_t=0 \qquad \text{ and } \qquad \int_{0^-}^T A^L_td\hat L_t=0,
\end{equation}
where $(Y^L,Z^L,A^L)\in\mathbb S^2\times\mathbb H^2\times\mathbb A^2$ is the unique solution of equation \eqref{RBSDE} corresponding to process $L$ (recall that the existence and uniqueness of $(Y^L,Z^L,A^L)$ follow from Proposition \ref{P:ExistUniq}). In the remainder of this section, for notational simplicity, we denote the solution of the reflected BSDE by $(Y,Z,A)$, with the understanding that it is associated with a given $L \in \mathcal{V}$. Our aim is to apply the Kakutani--Fan--Glicksberg fixed-point theorem (see e.g. Corollary 17.55 in \cite{AliprantisBorder2006}) to the set-valued map $\Gamma$. Such a theorem can be applied thanks to the following properties, whose proof is reported below:
\begin{itemize}
    \item $\mathbb H^2(\mathbb R)$, endowed with the weak topology, is a Hausdorff locally convex topological vector space;
    \item $\mathcal V$ is a non-empty, compact, and convex subset of $\mathbb H^2(\mathbb R)$;
    \item for every $L\in\mathcal V$ the set $\Gamma(L)$ is non-empty and convex;
    \item the graph of the set-valued map $\Gamma$ is closed.
\end{itemize}
If all these properties hold true, then there exists a fixed point $L\in\mathcal V$: $L\in\Gamma(L)$. This yields the existence of a quintuple $(X,Y,Z,A,L)\in\mathbb S^2\times\mathbb S^2\times\mathbb H^2\times\mathbb A^2\times\mathcal V$ solution to system \eqref{MFG SDE}-\eqref{MFRBSDE} (Proposition \ref{P:SDE} and Theorem \ref{thm:main}).

\paragraph{Properties of $\mathcal{V}$ and $\Gamma(L)$.}
In this part, we show some properties of the sets $\mathcal{V}$ and $\Gamma(L)$, for each $L \in \mathcal{V}$.

First note that the fact that $\mathbb H^2(\mathbb R)$, endowed with the weak topology, is a Hausdorff locally convex topological vector space follows from the fact that $\mathbb H^2(\mathbb R)$ is a Hilbert space. On the other hand, $\mathcal V$ is clearly non-empty and convex, while the compactness of $\mathcal V$ follows from the following lemma.
\begin{lemma}[\textit{Compactness of $\mathcal V$}]
\label{L:Compact}
The subset $\mathcal V\subset\mathbb H^2(\mathbb R)$ is compact when $\mathbb H^2(\mathbb R)$ is endowed with the weak topology.    
\end{lemma}
\begin{proof}
For every $L\in\mathcal V$, $\|L\|_{\mathbb H^2(\mathbb R)}^2
=
\mathbb E[\int_0^T |L_t|^2dt]
\le T$, hence $\mathcal V$ is bounded. Since $\mathbb H^2(\mathbb R)$ is a Hilbert space, it is reflexive. Therefore, it is enough to prove that $\mathcal V$ is weakly closed. As $\mathcal V$ is convex, it is enough to prove that $\mathcal V$ is strongly closed. Let us verify the strong closedness of $\mathcal V$. Let $\{L^n\}_n \subset \mathcal{V}$ be a sequence converging strongly to $L \in \mathbb H^2(\mathbb R)$. Strong convergence implies the existence of a subsequence converging pointwise almost everywhere to $L$. Since each $L^n$ is non-increasing, there exists $\tilde L$, in the same equivalence class of $L$, which has almost surely non-increasing paths. We now define a representative $\hat L$ by right-continuous regularization on $[0, T)$:
\[
\hat L_t(\omega) := \lim_{s \downarrow t, s \in \mathbb{Q}} \tilde{L}_s(\omega), \quad t \in [0, T).
\]
Set $\hat L_{0^-}=1$ and $\hat L_T=0$. This process is c\`adl\`ag on $[0, T]$ and adapted (see Theorem 3.13 in \cite{karatzas2014brownian}), so that $\hat L\in\mathcal V$, which shows that $\mathcal V$ is strongly closed.
\end{proof}
\begin{remark}
    \label{remark: V metrizable}
    By Lemma A 3.8 in \cite{kallenberg1997foundations}, $\mathbb H^2(\mathbb R)$ is separable. Consequently, the weak topology restricted to any weakly compact subset of $\mathbb H^2(\mathbb R)$ is metrizable (see, e.g., Theorem 3.28 in \cite{brezis2011functional}).
\end{remark}
Now, notice that for every $L\in\mathcal V$ the set $\Gamma(L)$ is convex, as it follows easily from \eqref{hatL}. Moreover, $\Gamma(L)$ is non-empty, as a consequence of the following result.

\begin{lemma}[\textit{Non-emptiness of $\Gamma(L)$}]\label{L:Non-Empty}
Suppose that Assumptions \ref{assumption:1A} and \ref{assumption:1B} hold. For every $L\in\mathcal V$ the set $\Gamma(L)$ is non-empty.
\end{lemma}
\begin{proof}
Given $L\in\mathcal V$, we know from Proposition \ref{P:ExistUniq} that there exists a unique triplet $(Y,Z,A)\in\mathbb S^2\times\mathbb H^2\times\mathbb A^2$ satisfying \eqref{RBSDE}. Define the stopping time $\tau$ as
\[
\tau := \inf \big\{ t \in [0, T] \colon Y_t = \xi_t \big\},
\]
with the convention $\inf \emptyset = T$. Since $\xi_T=\zeta$ and $\zeta=Y_T$, we see that $\tau\leq T$. We construct the candidate process $\hat{L}$ as follows:
\[
\hat{L}_t := \mathbf{1}_{t<\tau}.
\]
It is easy to see that $\hat L\in\mathcal V$. It remains to verify the validity of conditions \eqref{hatL}. It holds that
\[
\int_{0^-}^T (Y_t-\xi_t)d\hat L_t=-(Y_\tau-\xi_\tau)=0,
\]
where the last equality follows from the definition of $\tau$ and the continuity of the processes $Y$ and $\xi$. Finally, we have
\[
\int_{0^-}^T A_td\hat L_t=-A_\tau.
\]
By \eqref{RBSDE}, we know that the Skorokhod condition holds: $\int_0^T(Y_t-\xi_t)dA_t=0$. Since $A_0=0$ and $A$ is continuous, this condition implies that $A_t=0$ for every $t\leq\tau$, so that $A_\tau=0$.
\end{proof}
\paragraph{Technical results.}
 To prove the closure of the graph of the set-valued map $\Gamma$, we need to establish several technical results. We begin by establishing the continuous dependence of the solution $(Y,Z,A)$ to equation \eqref{RBSDE} with respect to the process $L\in\mathcal V$, as detailed in Proposition \ref{P:Stability} and Corollary \ref{C:Gamma^1}. To this end, we require the following three technical results.

\begin{lemma}\label{L:UnifConv_phi}
Suppose that Assumption \ref{assumption:1B}.2 holds. Let $\{L^n\}_n\subset\mathcal V$ be a sequence weakly converging in $\mathbb H^2(\mathbb R)$ to some $L\in\mathcal V$. Then, it holds that
\[
\lim_{n\rightarrow+\infty}\sup_{0\leq t\leq T}\bigg|\mathbb E\bigg[\int_{0^-}^T\phi(t-s)dL_s^n\bigg] - \mathbb E\bigg[\int_{0^-}^T\phi(t-s)dL_s\bigg]\bigg| = 0.
\]
\end{lemma}
\begin{proof}
By integration by parts, we obtain
\[
\int_{0^-}^T \phi(t-s)dL_s = \phi(t-T)L_T - \phi(t)L_{0^-} + \int_0^T L_s\phi'(t-s)ds = -\phi(t) + \int_0^T L_s\phi'(t-s)ds,
\]
where the last equality follows from $L_{0^-}=1$ and $L_T=0$. Since the same result holds with $L^n$ in place of $L$, we find
\begin{align*}
&\mathbb E\bigg[\int_{0^-}^T\phi(t-s)dL_s^n\bigg] - \mathbb E\bigg[\int_{0^-}^T\phi(t-s)dL_s\bigg] \\
&= \mathbb E\bigg[\int_0^T L_s^n\phi'(t-s)ds\bigg] - \mathbb E\bigg[\int_0^T L_s\phi'(t-s)ds\bigg].
\end{align*}
Recalling that the sequence $\{L^n\}_n$ weakly converges in $\mathbb H^2(\mathbb R)$ to $L$, we deduce that, for every $t\in[0,T]$,
\[
\lim_{n\rightarrow+\infty}\mathbb E\bigg[\int_0^T L_s^n\phi'(t-s)ds\bigg] = \mathbb E\bigg[\int_0^T L_s\phi'(t-s)ds\bigg].
\]
Let $v_n,v\colon[0,T]\rightarrow\mathbb R$ be defined respectively as $v_n(t) = \mathbb E[\int_0^T L_s^n\phi'(t-s)ds]$ and $v(t) = \mathbb E[\int_0^T L_s\phi'(t-s)ds]$. Notice that $\{v_n\}_n$ is a sequence of continuous and uniformly bounded functions on $[0,T]$, which converges pointwise to $v$. Let us prove that the convergence is uniform. Since the sequence $\{v_n\}_n$ is uniformly bounded, this follows if we prove that $\{v_n\}_n$ is uniformly equicontinuous. It holds that
\[
|v_n(t) - v_n(t')| \leq \mathbb E\bigg[\int_0^T L_s^n|\phi'(t-s) - \phi'(t'-s)|ds\bigg].
\]
Recalling that $|L^n|\leq1$, moreover $\phi'$ is continuous on the compact set $[-T,T]$, so that $\phi'$ is uniformly continuous with some modulus $\eta\colon[0,+\infty)\rightarrow[0,+\infty)$, we find
\[
|v_n(t) - v_n(t')| \leq T\eta(|t-t'|).
\]
This shows that $v_n$ is uniformly continuous on $[0,T]$, uniformly with respect to $n$.
Then, by the Arzel\`a--Ascoli theorem, given a subsequence $\{v_{n_k}\}_k$ there exists a sub-subsequence $\{v_{n_{k_h}}\}_h$ which converges uniformly to $v$. This implies that the entire sequence $\{v_n\}_n$ converges uniformly to $v$.
\end{proof}

\begin{lemma}\label{L:m_BV}
Suppose that Assumption \ref{assumption:1A} holds. Let $L'\in\mathcal V$ and let $X$ be the solution to equation \eqref{MFG SDE} with input process $L\in\mathcal V$. Define $m\colon[0,T]\rightarrow\mathbb R^k$ as follows:
    \[
    m(t)=\mathbb E[\psi(t,X_t)L_t'], \qquad 0\leq t\leq T,
    \]
    where $\psi\colon[0,T]\times\mathbb R^d\rightarrow\mathbb R^k$, with derivatives satisfying \eqref{PolGrowthCond}, and such that there exist constants $K\geq0$, $r\in[1,q/2)$,
    \[
        |\psi(t,x)| \leq K(1+|x|^r), \qquad (t,x)\in[0,T]\times\mathbb R^d.
    \]
    Then, every component of the function $m=(m_1,\ldots,m_k)$ has bounded variation.
\end{lemma}
\begin{proof}
Without loss of generality, we suppose that $k=1$, so that our aim is to prove that $m$ has bounded variation. Recalling that $|L_t'|\leq1$, we find
\begin{equation}\label{UniformBound}
\sup_{0\leq t\leq T}|m(t)|=\sup_{0\leq t\leq T}|\mathbb{E}[\psi(t,X_t)L_t']| \le K\Big(1+\mathbb{E}\Big[\sup_{0\leq t\leq T}|X_t|^{r}\Big]\Big) < \infty,
\end{equation}
where the last inequality follows from the standard estimate \eqref{EstimateSDE}. We proceed by applying the integration by parts formula to $\psi(t,X_t)L_t'$:
\begin{align*}
\psi(t,X_t)L_t'&=\psi(0,X_0)L_0'+\int_0^tL_{s^-}'d\psi(s,X_s)+\int_0^t\psi(s,X_s)dL_s'\\
&=\psi(0,X_0)L_0'+\int_0^tL_s'\partial_t \psi(s,X_s)ds + \int_0^t L_s'\partial_x\psi(s,X_s)b(s,X_s,\mathbb E[\bar b(s,X_s)L_s])ds \\
&\quad + \frac{1}{2}\int_0^t L_s'\text{tr}\big[(\sigma\sigma^{\scriptscriptstyle\intercal})(s,X_s,\mathbb E[\bar\sigma(s,X_s)L_s])\partial_{xx}^2\psi(s,X_s)\big]ds \\
&\quad + \int_0^tL_{s^-}'\partial_x\psi(s,X_s)\sigma(s,X_s,\mathbb E[\bar\sigma(s,X_s)L_s])dW_s + \int_0^t\psi(s,X_s)dL_s'.
\end{align*}
Since $L'$ is bounded, $\partial_x\psi$ satisfies a polynomial growth condition, $\sigma,\bar\sigma$ satisfy a linear growth condition, and estimate \eqref{EstimateSDE} holds, the stochastic integral is a martingale. As a consequence, taking the expectation, we find
\begin{align*}
m(t)&=m(0)+\int_0^t\mathbb E[L_s'\partial_t \psi(s,X_s)]ds + \int_0^t \mathbb E[L_s'\partial_x\psi(s,X_s)b(s,X_s,\mathbb E[\bar b(s,X_s)L_s])]ds \\
&\quad + \frac{1}{2}\int_0^t \mathbb E\big[L_s'\text{tr}\big[(\sigma\sigma^{\scriptscriptstyle\intercal})(s,X_s,\mathbb E[\bar\sigma(s,X_s)L_s])\partial_{xx}^2\psi(s,X_s)\big]\big]ds + \mathbb E\bigg[\int_0^t\psi(s,X_s)dL_s'\bigg].
\end{align*}
Notice that the first three integrals above correspond to bounded variation functions. Then, the claim follows if we prove that $\ell(t):=\mathbb E[\int_0^t\psi(s,X_s)dL_s']$, $t\in[0,T]$, has bounded variation. Notice that the total variation of $\ell$ on $[0,T]$ is upper bounded by
\[
\mathbb E\bigg[\int_0^T|\psi(s,X_s)|d|L'|_s\bigg].
\]
Recalling that the total variation of $L_t'$ from $t=0^-$ to $t=T$ is exactly $1$, and also that $\psi$ satisfies a polynomial growth condition, we obtain, for a suitable constant $C$,
\begin{equation}\label{UniformlyBoundedVariation}
\mathbb E\bigg[\int_0^T|\psi(s,X_s)|d|L'|_s\bigg] \le \mathbb E\Big[\sup_{0\leq s\leq T}|\psi(s,X_s)|\Big] \le C\Big(1+\mathbb{E}\Big[\sup_{t \in [0,T]}|X_t|^{r}\Big]\Big).
\end{equation}
By estimate \eqref{EstimateSDE}, we conclude that $\ell$ is a function with bounded variation.
\end{proof}

\begin{lemma}\label{L:Conv_E[XL]}
Suppose that Assumption \ref{assumption:1A} holds. Let $\{L^n\}_n\subset\mathcal V$ be a sequence weakly converging in $\mathbb H^2(\mathbb R)$ to some $L\in\mathcal V$. Let $\psi$ be as in Lemma \ref{L:m_BV}. Then, it holds that, up to a subsequence,
    \[
    \mathbb E[\psi(t,X_t)L_t^n] \rightarrow \mathbb E[\psi(t,X_t)L_t] \quad \text{and} \quad \mathbb E[\psi(t,X_t^n)L_t^n] \rightarrow \mathbb E[\psi(t,X_t)L_t], \qquad \text{for a.e. }t\in[0,T],
    \]
    where $X^n$ is the solution to equation \eqref{MFG SDE} with $L^n$ in place of $L$. In particular, for every $p\geq1$, it holds that
    \[
    \int_0^T\big|\mathbb E[\psi(t,X_t)L_t^n] - \mathbb E[\psi(t,X_t)L_t]\big|^pdt\rightarrow0, \qquad
    \int_0^T\big|\mathbb E[\psi(t,X_t^n)L_t^n] - \mathbb E[\psi(t,X_t)L_t]\big|^pdt\rightarrow0.
    \]
\end{lemma}
\begin{proof}
We prove the lemma in the case $k=1$, as the general case can be proved proceeding component-wise. We begin by proving the claim with $X$ fixed. Let $m(t)=\mathbb E[\psi(t,X_t)L_t]$, $t\in[0,T]$. Let us also define the sequence of functions $\{m^n\}_n$ where $m^n(t)=\mathbb{E}[\psi(t,X_t)L^n_t]$, $t\in[0,T]$. By Lemma \ref{L:m_BV} we know that each $m^n$, as well as $m$, has bounded variation. We also know by \eqref{EstimateSDE}, \eqref{UniformBound}, \eqref{UniformlyBoundedVariation} that they are uniformly bounded and they have a uniformly bounded variation. Then, by Helly's selection theorem there exists a subsequence $\{m^{n_h}\}_h$ and a function $\tilde{m}$ of bounded variation such that $m^{n_h}$ converges to $\tilde{m}$ pointwise almost everywhere.\\
Now, consider a function $F\in \text{L}^\infty([0,T])$. We have $F(t)m^{n_h}(t) \to F(t)\tilde{m}(t)$ a.e. and moreover $|F(t)m^{n_h}(t)|\le C||F||_{\infty} \in \text{L}^1([0,T])$. By Lebesgue's dominated convergence theorem, we deduce that
\[
\lim_{h \to \infty} \int_0^TF(t)\mathbb{E}[\psi(t,X_t)L^{n_h}_t]dt=\int_0^TF(t)\tilde{m}(t)dt.
\]
Since $F\in \text{L}^\infty([0,T])$, it holds that $F\psi(\cdot,X_\cdot)$ is in $\mathbb H^2(\mathbb R)$. Then, since $\{L^n\}_n$ weakly converges to $L$ in $\mathbb H^2(\mathbb R)$, we obtain
\begin{align*}
\lim_{n \to \infty}\int_0^T\mathbb{E}[F(t)\psi(t,X_t)L^n_t]dt&=\lim_{n \to \infty}\mathbb{E}\bigg[\int_0^TF(t)\psi(t,X_t)L^n_tdt\bigg]\\
&=\mathbb{E}\bigg[\int_0^TF(t)\psi(t,X_t)L_tdt\bigg]=\int_0^TF(t)\mathbb{E}[\psi(t,X_t)L_t]dt.
\end{align*}
By uniqueness of the limit, $\int_0^TF(t)(\mathbb{E}[\psi(t,X_t)L_t]-\tilde{m}_t)=0$. Since this is true for all $F \in \text{L}^\infty([0,T])$, we conclude that
\[
\tilde{m}(t)=\mathbb{E}[\psi(t,X_t)L_t], \qquad \text{for a.e. }t\in[0,T].
\]
Hence, $m^{n_h}(t)=\mathbb{E}[\psi(t,X_t)L^{n_h}_t] \to \mathbb{E}[\psi(t,X_t)L_t]$  for a.e. $t\in[0,T]$.\\
Finally, the last claim (the case with $X$ fixed) follows easily applying Lebesgue's dominated convergence theorem to show that every subsequence has a further subsequence converging to zero.\\
Let us now consider the case with $X^n$ and $X$. We define $m(t)=\mathbb E[\psi(t,X_t)L_t]$ and $m^n(t)=\mathbb{E}[\psi(t,X_t^n)L^n_t]$, $t\in[0,T]$. As before, we deduce that there exists a subsequence $\{m^{n_h}\}_h$ and a function $\tilde{m}$ of bounded variation such that $m^{n_h}$ converges to $\tilde{m}$ pointwise almost everywhere.\\
Now, consider a function $F\in \text{L}^\infty([0,T])$. We have $F(t)m^{n_h}(t) \to F(t)\tilde{m}(t)$ a.e. and moreover $|F(t)m^{n_h}(t)|\le C||F||_{\infty} \in \text{L}^1([0,T])$. By Lebesgue's dominated convergence theorem, we deduce that
\[
\lim_{h \to \infty} \int_0^TF(t)\mathbb{E}[\psi(t,X_t^{n_h})L^{n_h}_t]dt=\int_0^TF(t)\tilde{m}(t)dt.
\]
On the other hand, by \eqref{EstimateLL'} and the claim with $X$ fixed of this lemma, we have that $X^n$ converges to $X$ in $\mathbb S^q$, with $q>2$, whenever $L^n$ weakly converges to $L$ in $\mathbb H^2(\mathbb R)$. Since $\psi$ is continuous and satisfies a polynomial growth condition, moreover estimate \eqref{EstimateSDE} holds, we deduce by Vitali's convergence theorem that $\psi(\cdot,X_\cdot^n)$ strongly converges to $\psi(\cdot,X_\cdot)$ in $\mathbb H^2(\mathbb R)$. Since $F\in \text{L}^\infty([0,T])$, this immediately implies that $F\psi(\cdot,X_\cdot^n)$ strongly converges to $F\psi(\cdot,X_\cdot)$ in $\mathbb H^2(\mathbb R)$. Then, by the continuity of the inner product in the Hilbert space $\mathbb H^2(\mathbb R)$ when one component converges weakly and the other converges strongly, it holds that
\begin{align*}
\lim_{n \to \infty}\int_0^T\mathbb{E}[F(t)\psi(t,X_t^n)L^n_t]dt&=\lim_{n \to \infty}\mathbb{E}\bigg[\int_0^TF(t)\psi(t,X_t^n)L^n_tdt\bigg]\\
&=\mathbb{E}\bigg[\int_0^TF(t)\psi(t,X_t)L_tdt\bigg]=\int_0^TF(t)\mathbb{E}[\psi(t,X_t)L_t]dt.
\end{align*}
By uniqueness of the limit, $\int_0^TF(t)(\mathbb{E}[\psi(t,X_t)L_t]-\tilde{m}_t)=0$. Since this is true for all $F \in \text{L}^\infty([0,T])$, we conclude that
\[
\tilde{m}(t)=\mathbb{E}[\psi(t,X_t)L_t], \qquad \text{for a.e. }t\in[0,T].
\]
Hence, $m^{n_h}(t)=\mathbb{E}[\psi(t,X_t^{n_h})L^{n_h}_t] \to \mathbb{E}[\psi(t,X_t)L_t]$  for a.e. $t\in[0,T]$.\\
Finally, as before the last claim (the case with $X^n$ and $X$) follows easily applying Lebesgue's dominated convergence theorem to show that every subsequence has a further subsequence converging to zero.
\end{proof}

\begin{proposition}\label{P:Stability}
Suppose that Assumptions \ref{assumption:1A} and \ref{assumption:1B} hold. Let $L,L'\in\mathcal V$. Consider the corresponding solutions $X,X'\in\mathbb S^2$ of equation \eqref{MFG SDE} and $(Y,Z,A),$ $(Y',Z',A')\in\mathbb S^2\times\mathbb H^2\times\mathbb A^2$ of equation \eqref{RBSDE}, respectively. Then, there exists a constant $c$, depending only on $X_0$, $T$, $K$, $\phi$, such that
    \begin{align}\label{estimateDeltaYZA}
    &\mathbb E\bigg[\sup_{0\leq t\leq T}|\Delta Y_t|^2 + \int_0^T |\Delta Z_t|^2dt + \sup_{0\leq t\leq T}|\Delta A_t|^2\bigg] \notag \\
    &\leq c\,\mathbb E\bigg[\int_0^T\big|f(t,X_t,\mathbb E[\bar f(t,X_t)L_t]) - f(t,X_t',\mathbb E[\bar f(t,X_t)L_t])\big|^2dt\bigg] \\
    &\quad+ c\int_0^T\big|\mathbb E[\bar f(t,X_t)L_t] - \mathbb E[\bar f(t,X_t')L_t']\big|^2dt + c\,\mathbb E\big[|\xi_T-\tilde\xi_T'|^2\big] + c\sqrt{\mathbb E\Big[\sup_{0\leq t\leq T}|\xi_t-\tilde\xi_t'|^2\Big]} \notag \\
    &\quad+ c\bigg|\mathbb E\bigg[\int_{0^-}^T\phi(T-s)dL_s\bigg] - \mathbb E\bigg[\int_{0^-}^T\phi(T-s)dL_s'\bigg]\bigg|^2 \notag \\
    &\quad+ c\sup_{0\leq t\leq T}\bigg|\mathbb E\bigg[\int_{0^-}^T\phi(t-s)dL_s\bigg] - \mathbb E\bigg[\int_{0^-}^T\phi(t-s)dL_s'\bigg]\bigg|, \notag
    \end{align}
    where $\Delta Y_t=Y_t-Y_t'$, $\Delta Z_t=Z_t-Z_t'$, $\Delta A_t=A_t-A_t'$, $\xi_t=h(t,X_t,\mathbb E[\int_{0^-}^T\phi(t-s)dL_s])$, $\tilde\xi_t'=h(t,X_t',\mathbb E[\int_{0^-}^T\phi(t-s)dL_s])$, for every $t\in[0,T]$.
\end{proposition}
\begin{proof}
    By Proposition 3.6 in \cite{el1981aspects} it follows that there exists a constant $c\geq0$ such that (in the sequel, we denote by $c$ a non-negative constant, depending only on $X_0$, $T$, $K$, $\phi$, which may change from line to line)
\begin{align*}
    &\mathbb E\bigg[\sup_{0\leq t\leq T}|\Delta Y_t|^2 + \int_0^T |\Delta Z_t|^2dt + |\Delta A_T|^2\bigg] \leq c\,\mathbb E\big[|\xi_T-\xi_T'|^2\big] + c\sqrt{\mathbb E\Big[\sup_{0\leq t\leq T}\big|\xi_t - \xi_t'\big|^2\Big]}\sqrt{\psi_T} \\
    &+ c\,\mathbb E\bigg[\int_0^T\big|f(t,X_t,\mathbb E[\bar f(t,X_t)L_t]) - f(t,X_t',\mathbb E[\bar f(t,X_t')L_t'])\big|^2dt\bigg],
\end{align*}
where $\xi_t'=h(t,X_t',\mathbb E[\int_{0^-}^T\phi(t-s)dL_s'])$, for $t\in[0,T]$, and
\begin{align*}
    \psi_T &= \mathbb E\bigg[\xi_T^2 + (\xi_T')^2 + \int_0^T \big(\big|f(t,X_t,\mathbb E[\bar f(t,X_t)L_t])\big|^2 + \big|f(t,X_t',\mathbb E[\bar f(t,X_t')L_t'])\big|^2\big)dt \\
    &\quad + \sup_{0\leq t\leq T}\big(\big|\xi_t^+\big|^2 + \big|(\xi_t')^+\big|^2\big)\bigg].
\end{align*}
From the polynomial growth conditions on $f$, $\bar f$, $h$, the Lipschitz continuity of $f$ with respect to its last argument, estimate \eqref{EstimateSDE}, and the boundedness of $L$ and $L'$, it follows that
\[
\psi_T\leq c.
\]
By the Lipschitz property of $f$ and $h$ in their last argument, we obtain
\begin{align}\label{estimateDeltaYZA_proof}
    &\mathbb E\bigg[\sup_{0\leq t\leq T}|\Delta Y_t|^2 + \int_0^T |\Delta Z_t|^2dt + |\Delta A_T|^2\bigg] \notag \\
    &\leq c\,\mathbb E\bigg[\int_0^T\big|f(t,X_t,\mathbb E[\bar f(t,X_t)L_t]) - f(t,X_t',\mathbb E[\bar f(t,X_t)L_t])\big|^2dt\bigg] \\
    &\quad+ c\int_0^T\big|\mathbb E[\bar f(t,X_t)L_t] - \mathbb E[\bar f(t,X_t')L_t']\big|^2dt \notag \\
    &\quad+ c\mathbb E\bigg[\bigg|h\bigg(T,X_T,\mathbb E\bigg[\int_{0^-}^T\phi(T-s)dL_s\bigg]\bigg) - h\bigg(T,X_T',\mathbb E\bigg[\int_{0^-}^T\phi(T-s)dL_s\bigg]\bigg)\bigg|^2\bigg] \notag \\
    &\quad+ c\sqrt{\mathbb E\bigg[\sup_{0\leq t\leq T}\bigg|h\bigg(t,X_t,\mathbb E\bigg[\int_{0^-}^T\phi(t-s)dL_s\bigg]\bigg) - h\bigg(t,X_t',\mathbb E\bigg[\int_{0^-}^T\phi(t-s)dL_s\bigg]\bigg)\bigg|^2\bigg]} \notag \\
    &\quad+ c\bigg|\mathbb E\bigg[\int_{0^-}^T\phi(T-s)dL_s\bigg] - \mathbb E\bigg[\int_{0^-}^T\phi(T-s)dL_s'\bigg]\bigg|^2 \notag \\
    &\quad+ c\sup_{0\leq t\leq T}\bigg|\mathbb E\bigg[\int_{0^-}^T\phi(t-s)dL_s\bigg] - \mathbb E\bigg[\int_{0^-}^T\phi(t-s)dL_s'\bigg]\bigg|. \notag
\end{align}
Now, by equation \eqref{RBSDE} we have
\[
    A_t = Y_0-Y_t- \int_0^t f(s,X_s,\mathbb E[\bar f(s,X_s)L_s])ds + \int_0^t(Z_s,dW_s)
\]
and similarly for $A'$. Hence
\begin{align*}
    &\mathbb E\big[\sup_{0\leq t\leq T}|\Delta A_t|^2\big] \leq c\,\mathbb E\bigg[\sup_{0\leq s\leq T}|\Delta Y_s|^2 \\
    &+ \bigg(\int_0^T \big|f(s,X_s,\mathbb E[\bar f(s,X_s)L_s]) - f(s,X_s',\mathbb E[\bar f(s,X_s')L_s'])\big|ds\bigg)^2 + \sup_{0\leq t\leq T}\bigg|\int_0^t(\Delta Z_s,dW_s)\bigg|^2\bigg].
\end{align*}
Using the Lipschitz property of $f$ in its last argument, by \eqref{estimateDeltaYZA_proof} we conclude that \eqref{estimateDeltaYZA} holds true.
\end{proof}

\begin{corollary}\label{C:Gamma^1}
Suppose that Assumptions \ref{assumption:1A}, \ref{assumption:1B} and \ref{assumption: 1C} hold. Let $L\in\mathcal V$ and consider the corresponding solution $(Y,Z,A)\in\mathbb S^2\times\mathbb H^2\times\mathbb A^2$ of system \eqref{RBSDE}. Let also $\{L^n\}_n\subset\mathcal V$ be a sequence weakly converging in $\mathbb H^2(\mathbb R)$ to $L$. For each $n$, let $(Y^n,Z^n,A^n)\in\mathbb S^2\times\mathbb H^2\times\mathbb A^2$ be the solution of system \eqref{RBSDE} with $L^n$ in place of $L$. Then, it holds that $Y^n \to Y$ in $\mathbb{S}^2$, $Z^n \to Z$ in $\mathbb{H}^2$, $A^n \to A$ in $\mathbb{S}^2$.
\end{corollary}
\begin{proof}
Denote $\Delta Y_t=Y^n_t-Y_t$, $\Delta Z_t=Z^n_t-Z_t$, $\Delta A_t=A^n_t-A_t$, for every $t\in[0,T]$. Then, by Proposition \ref{P:Stability} it holds that
\begin{align}\label{Gamma^1_proof}
    &\mathbb E\bigg[\sup_{0\leq t\leq T}|\Delta Y_t|^2 + \int_0^T |\Delta Z_t|^2dt + \sup_{0\leq t\leq T}|\Delta A_t|^2\bigg] \notag \\
    &\leq c\,\mathbb E\bigg[\int_0^T\big|f(t,X_t,\mathbb E[\bar f(t,X_t)L_t]) - f(t,X_t^n,\mathbb E[\bar f(t,X_t)L_t])\big|^2dt\bigg] \\
    &\quad+ c\int_0^T\big|\mathbb E[\bar f(t,X_t)L_t] - \mathbb E[\bar f(t,X_t^n)L_t^n]\big|^2dt \notag \\
    &\quad+ c\,\mathbb E\bigg[\bigg|h\bigg(T,X_T,\mathbb E\bigg[\int_{0^-}^T\phi(T-s)dL_s\bigg]\bigg) - h\bigg(T,X_T^n,\mathbb E\bigg[\int_{0^-}^T\phi(T-s)dL_s^n\bigg]\bigg)\bigg|^2\bigg] \notag \\
    &\quad+ c\sqrt{\mathbb E\bigg[\sup_{0\leq t\leq T}\bigg|h\bigg(t,X_t,\mathbb E\bigg[\int_{0^-}^T\phi(t-s)dL_s\bigg]\bigg) - h\bigg(t,X_t^n,\mathbb E\bigg[\int_{0^-}^T\phi(t-s)dL_s\bigg]\bigg)\bigg|^2\bigg]} \notag \\
    &\quad+c\bigg|\mathbb E\bigg[\int_{0^-}^T\phi(T-s)dL_s\bigg] - \mathbb E\bigg[\int_{0^-}^T\phi(T-s)dL_s^n\bigg]\bigg|^2 \notag \\
    &\quad+ c\sup_{0\leq t\leq T}\bigg|\mathbb E\bigg[\int_{0^-}^T\phi(t-s)dL_s\bigg] - \mathbb E\bigg[\int_{0^-}^T\phi(t-s)dL_s^n\bigg]\bigg|. \notag
\end{align}
Now, by Lemma \ref{L:UnifConv_phi} and Lemma \ref{L:Conv_E[XL]}, we see that the second and the last two terms on the right-hand side of \eqref{Gamma^1_proof} vanish. Regarding the other three terms, recall that by \eqref{EstimateLL'} and Lemma \ref{L:Conv_E[XL]} we have that $X^n$ converges to $X$ in $\mathbb S^q$, with $q>2$, whenever $L^n$ weakly converges to $L$ in $\mathbb H^2(\mathbb R)$. Since $x\mapsto f(t,x,m)$ is continuous, for every fixed $(t,m)$, and satisfies a polynomial growth condition of order $r<q/2$, moreover estimate \eqref{EstimateSDE} holds, we deduce by Vitali's convergence theorem that the first term in the right-hand side of \eqref{Gamma^1_proof} converges to zero as $n\rightarrow\infty$. Similarly, since $(t,x)\mapsto h(t,x,w)$ is continuous, for every fixed $w$, we again deduce that the third and fourth terms in \eqref{Gamma^1_proof} converge to zero as $n\rightarrow\infty$.
\end{proof}

\noindent We now address the closure of the graph of $\Gamma$, for which we require the following technical result.

\begin{lemma}\label{L:WeakConv}
Let $\{L^n\}_n\subset\mathcal V$ be a sequence weakly converging in $\mathbb H^2(\mathbb R)$ to some $L\in\mathcal V$. Let $I=\{I_t\}_{0\leq t\leq T}$ be a real-valued It\^o process:
    \[
    I_t = I_0 + \int_0^t F_s ds + \int_0^t G_s dW_s, \qquad 0\leq t\leq T,
    \]
    with $F\colon[0,T]\times\Omega\rightarrow\mathbb R$, $G\colon[0,T]\times\Omega\rightarrow\mathbb R^m$ being progressively measurable and bounded, and with $I_0$ being a deterministic constant. We set $I_{0^-}:=I_0$. Then, it holds that $\int_{0^-}^T I_tdL_t^n$ converges to $\int_{0^-}^T I_tdL_t$ weakly in $\mathcal{L}^2(\mathcal{F}_T)$, that is
    \[
    \mathbb E\bigg[\eta\int_{0^-}^T I_tdL_t^n\bigg] \rightarrow \mathbb E\bigg[\eta\int_{0^-}^T I_tdL_t\bigg],
    \]
    for every real-valued random variable $\eta$ in $\mathcal{L}^2(\mathcal{F}_T)$.
\end{lemma}
\begin{proof}
    By the integration by parts formula, we find
    \begin{align*}
    I_TL_T^n &= I_0L_{0^-}^n + \int_0^T L_{t^-}^ndI_t + \int_{0^-}^T I_tdL_t^n \\
    &= I_0L_{0^-}^n + \int_0^T L_{t}^n F_tdt + \int_0^T L_{t^-}^nG_tdW_t + \int_{0^-}^T I_tdL_t^n.
    \end{align*}
    Recalling that $L_{0^-}^n=1$ and $L_T^n=0$, we obtain
    \[
    \int_{0^-}^T I_tdL_t^n = - I_0 - \int_0^T L_{t}^n F_tdt - \int_0^T L_{t^-}^nG_tdW_t.
    \]
    Since $\eta\in\mathcal{L}^2(\mathcal{F}_T)$ and the filtration $\{\mathcal F_t\}_{t\in[0,T]}$ is the augmentation of the filtration generated by $X_0$ and $W$, there exists a process $H\in\mathbb H^2$ such that
    \[
    \eta = \mathbb E[\eta|X_0] + \int_0^T H_tdW_t.
    \]
    Using again the integration by parts formula, we find
    \[
    \eta\bigg(\int_0^T L_t^nF_tdt\bigg) = \int_0^T\bigg(\mathbb E[\eta|X_0] + \int_0^t H_sdW_s\bigg)L_t^nF_tdt + \int_0^T\bigg(\int_0^t L_s^nF_sds\bigg)H_tdW_t
    \]
    and
    \begin{align*}
    \eta\bigg(\int_0^T L_{t^-}^nG_tdW_t\bigg) &= \int_0^T\bigg(\mathbb E[\eta|X_0] + \int_0^t H_sdW_s\bigg)L_{t^-}^nG_tdW_t + \int_0^T\bigg(\int_0^t L_{s^-}^nG_sdW_s\bigg)H_tdW_t \\
    &\quad + \int_0^TL_{t^-}^n(G_t,H_t)dt.
    \end{align*}
    Hence
    \begin{align}\label{eta_int_IdL^n}
    \mathbb E\bigg[\eta\int_{0^-}^T I_tdL_t^n\bigg] &= - \mathbb E[\eta I_0] - \mathbb E\bigg[\int_0^T\bigg(\mathbb E[\eta|X_0] + \int_0^t H_sdW_s\bigg)L_{t^-}^nF_tdt\bigg] \\
    &\quad - \mathbb E\bigg[\int_0^TL_{t^-}^n(G_t,H_t)dt\bigg]. \notag
    \end{align}
    Similarly, considering the process $L$ in place of $L^n$, and proceeding along the same lines as above, we obtain
    \begin{equation}\label{eta_int_IdL}
    \begin{split}
    \mathbb E\bigg[\eta\int_{0^-}^T I_tdL_t\bigg] &= - \mathbb E[\eta I_0] - \mathbb E\bigg[\int_0^T\bigg(\mathbb E[\eta|X_0] + \int_0^t H_sdW_s\bigg)L_{t^-}F_tdt\bigg] \\
    &\quad - \mathbb E\bigg[\int_0^TL_{t^-}(G_t,H_t)dt\bigg].
    \end{split}
    \end{equation}
    Since $F$ and $G$ are bounded, and $H\in\mathbb H^2$, the stochastic processes
    \[
    t\mapsto \bigg(\mathbb E[\eta|X_0] + \int_0^t H_sdW_s\bigg)F_t, \qquad\qquad t\mapsto (G_t,H_t)
    \]
    belong to $\mathbb H^2(\mathbb R)$. Then, by the weak convergence of $\{L^n\}_n$ to $L$, we can pass to the limit in \eqref{eta_int_IdL^n} obtaining
    \[
    \begin{split}
    \lim_{n\rightarrow\infty}\mathbb E\bigg[\eta\int_{0^-}^T I_tdL_t^n\bigg] &= - \mathbb E[\eta I_0] - \mathbb E\bigg[\int_0^T\bigg(\mathbb E[\eta|X_0] + \int_0^t H_sdW_s\bigg)L_{t^-}F_tdt\bigg] \\
    &\quad - \mathbb E\bigg[\int_0^TL_{t^-}(G_t,H_t)dt\bigg].
    \end{split}
    \]
    Then, the claim follows from \eqref{eta_int_IdL}.
\end{proof}
\paragraph{Closedness of the graph and main existence result.}
Using the technical results from the previous paragraph, we are now able to show the closedness of the graph of the set-valued map $\Gamma$.
\begin{proposition}[\textit{Closedness of the graph of the set-valued map $\Gamma$}]\label{P:Gamma}
Suppose that Assumptions \ref{assumption:1A}, \ref{assumption:1B} and \ref{assumption: 1C} hold. The  graph of the set-valued map $\Gamma\colon\mathcal V\rightarrow 2^{\mathcal V}$ is closed.
\end{proposition}
\begin{proof}
By Remark \ref{remark: V metrizable}, it is enough to prove sequential closedness. Let $L\in\mathcal V$ and consider the corresponding solution $(X,Y,Z,A)\in\mathbb{S}^2\times\mathbb S^2\times\mathbb H^2\times\mathbb A^2$ of \eqref{MFG SDE}-\eqref{RBSDE}. Let also $\{L^n\}_n\subset\mathcal V$ and, for each $n$, let  $(X^n,Y^n,Z^n,A^n)\in\mathbb{S}^2\times\mathbb S^2\times\mathbb H^2\times\mathbb A^2$ be the solution of \eqref{MFG SDE}-\eqref{RBSDE}, with $L^n$ in place of $L$. Now, for each $n$, let $\hat L^n\in\Gamma(L^n)$. Suppose that $L^n$ weakly converges to $L$ in $\mathbb H^2(\mathbb R)$, and similarly that $\hat{L}^n$ weakly converges to $\hat{L}$. We want to show that $\hat L\in\Gamma(L)$, that is
\[
\int_{0^-}^T(Y_t-\xi_t)d\hat{L}_t=0 \qquad \text{and} \qquad \int_{0^-}^TA_td\hat{L}_t=0.
\]
We report the proof of the first equality, as the other can be proved along the same lines.\\
By Corollary \ref{C:Gamma^1}, it holds that $(Y^n,Z^n,A^n) \rightarrow (Y,Z,A)$ in $\mathbb S^2\times\mathbb H^2\times\mathbb S^2$. Moreover, since $\hat{L}^n \in \Gamma(L^n)$, we know that $\int_{0^-}^T(Y^n_t-\xi_t^n)d\hat{L}^n_t=0$ and $\int_{0^-}^TA^n_td\hat{L}^n_t=0$. Now, we have
\begin{align*}
&\bigg|\int_{0^-}^T(Y_t-\xi_t)d\hat{L}^n_t\bigg| \\
&=\bigg|\int_{0^-}^T(Y_t-Y^n_t)d\hat{L}^n_t+\int_{0^-}^T(\xi_t^n-\xi_t)d\hat{L}^n_t +\int_{0^-}^T(Y_t^n-\xi_t^n)d\hat{L}^n_t\bigg| \\
&=\bigg|\int_{0^-}^T(Y_t - Y^n_t)d\hat{L}^n_t+\int_{0^-}^T(\xi_t^n-\xi_t)d\hat{L}^n_t\bigg| \\
&\le \int_{0^-}^T|Y^n_t-Y_t|d|\hat{L}^n|_t + K\int_{0^-}^T \bigg|\mathbb{E}\bigg[\int_{0^-}^T\phi(t-s)dL_s\bigg] - \mathbb{E}\bigg[\int_{0^-}^T\phi(t-s)dL_s^n\bigg]\bigg|d|\hat{L}^n|_t \\
&\quad +\int_{0^-}^T\bigg|h\bigg(t,X_t^n,\mathbb{E}\bigg[\int_{0^-}^T\phi(t-s)dL_s\bigg]\bigg)-h\bigg(t,X_t,\mathbb{E}\bigg[\int_{0^-}^T\phi(t-s)dL_s\bigg]\bigg)\bigg|d|\hat{L}^n|_t,
\end{align*}
where $d|\hat{L}^n|_t$ is the measure of the total variation of $\hat{L}^n$. Since the total variation of $\hat{L}^n_t$ is bounded by 1, we obtain
\begin{align*}
\bigg|\int_{0^-}^T(Y_t-\xi_t)d\hat{L}^n_t\bigg| \le \sup_{t \in [0,T]}|Y^n_t-Y_t| + K\sup_{0\leq t\leq T}\bigg|\mathbb{E}\bigg[\int_{0^-}^T\phi(t-s)dL_s^n\bigg] - \mathbb{E}\bigg[\int_{0^-}^T\phi(t-s)dL_s\bigg]\bigg| \\
+ \sup_{0\leq t\leq T}\bigg|h\bigg(t,X_t^n,\mathbb{E}\bigg[\int_{0^-}^T\phi(t-s)dL_s\bigg]\bigg)-h\bigg(t,X_t,\mathbb{E}\bigg[\int_{0^-}^T\phi(t-s)dL_s\bigg]\bigg)\bigg|.
\end{align*}
Recalling that $Y^n\rightarrow Y$ in $\mathbb S^2$, there exists a subsequence such that $\sup_t|Y_t^n-Y_t|$ converges almost surely to zero. Moreover, the second term above vanishes thanks to Lemma \ref{L:UnifConv_phi}. Furthermore, regarding the last term above, recall that by \eqref{EstimateLL'} and Lemma \ref{L:Conv_E[XL]} we have that $X^n$ converges to $X$ in $\mathbb S^q$, with $q>2$, whenever $L^n$ weakly converges to $L$ in $\mathbb H^2(\mathbb R)$. Since $(t,x)\mapsto h(t,x,w)$ is continuous, for every fixed $w$, and satisfies a polynomial growth condition, moreover estimate \eqref{EstimateSDE} holds, we deduce by Vitali's convergence theorem that the last term above converges to zero as $n\rightarrow\infty$ almost surely, up to a subsequence. In conclusion, up to a subsequence, the following almost sure convergence holds true:
\begin{equation}\label{WeakConvCesaro-1}
\lim_{n\rightarrow+\infty}\int_{0^-}^T(Y_t-\xi_t)d\hat{L}^n_t = 0.
\end{equation}
Now, let $\hat\nu^n(dt):=d|\hat{L}^n|_t$ be the measure of the total variation of $\hat{L}^n$, that is $\hat\nu^n(\{0\}):=1-\hat{L}_0^n$ and $\hat\nu^n((t,T]):=\hat L_t^n$, $t\in[0,T]$. By Lemma 3.5 of \cite{kabanov1999hedging}, there exist a random probability measure $\bar{\hat\nu}$ and a subsequence of $\{\hat{\nu}^n\}_n$, denoted by $\{\hat{\nu}^{n_k}\}_k$, such that the following Cesàro convergence holds: a.s. 
\begin{equation}\label{WeakConvCesaro0}
\int_{[0,T]}\phi(t)\bar{\hat{\nu}}^{n_{k}}(dt) \rightarrow \int_{[0,T]}\phi(t)\bar{\hat{\nu}}(dt), \qquad \forall \phi \in \mathcal{C}([0,T]),
\end{equation}
where $\bar{\hat{\nu}}^{n_k}(dt):=\frac{1}{k}\sum_{j=1}^k\hat{\nu}^{n_j}(dt)$, so that $\int_{[0,T]}\phi(t)\bar{\hat{\nu}}^{n_k}(dt)=\frac{1}{k}\sum_{j=1}^k\int_{[0,T]}\phi(t)\hat{\nu}^{n_j}(dt)$, and $\mathcal C([0,T])$ is the set of continuous functions from $[0,T]$ to $\mathbb R$. Convergence \eqref{WeakConvCesaro0} corresponds to the almost sure weak convergence of the random probability measures $\bar{\hat{\nu}}^{n_k}$ towards the random probability measure $\bar{\hat{\nu}}$. Notice that such a convergence holds almost surely and that the null set where it does not hold does not depend on $\phi$. Now, define $\bar{\hat L}\in\mathcal V$ as follows: $\bar{\hat L}_{0^-}:=1$ and $\bar{\hat L}_t:=\bar{\hat\nu}((t,T])$, for $t\in[0,T]$. Then, \eqref{WeakConvCesaro0} can be rewritten as
\begin{equation}\label{WeakConvCesaro1}
\int_{0^-}^T\phi(t)d\bar{\hat L}^{n_{k}}_t \rightarrow \int_{0^-}^T\phi(t)d\bar{\hat L}_t, \qquad \forall \phi \in \mathcal{C}([0,T]),
\end{equation}
where $\bar{\hat{L}}^{n_k}_t:=\frac{1}{k}\sum_{j=1}^k\hat{L}^{n_j}_t$, so that $\int_{0^-}^T\phi(t)d\bar{\hat{L}}^{n_k}_t=\frac{1}{k}\sum_{j=1}^k\int_{0^-}^T\phi(t)d\hat{L}^{n_j}_t$. Let us prove that $\bar{\hat L}$ is equal to $\hat L$ in $\mathbb H^2(\mathbb R)$. Recall that the weak convergence of probability measures on $[0,T]$ in \eqref{WeakConvCesaro1} can be written replacing $\mathcal C([0,T])$ by a countable convergence determining class $\{\phi^m\}_m\subset\mathcal C^\infty([0,T])$, see e.g. Chapter 3, Theorem 4.5, in \cite{ethier1986}. Then, by \eqref{WeakConvCesaro1} we obtain: a.s.
\[
\int_{0^-}^T\phi^m(t)d\bar{\hat{L}}^{n_k}_t \rightarrow \int_{0^-}^T\phi^m(t)d\bar{\hat{L}}_t, \qquad \text{for every }m.
\]
By the boundedness of $\phi^m$, and the fact that $|\bar{\hat{L}}^{n_k}|\leq1$, we deduce by Lebesgue's dominated convergence theorem that
\begin{equation}\label{WeakConvCesaro_m}
\lim_{k\rightarrow+\infty}\mathbb E\bigg[\eta\int_{0^-}^T\phi^m(t)d\bar{\hat{L}}^{n_k}_t\bigg] = \mathbb E\bigg[\eta\int_{0^-}^T\phi^m(t)d\bar{\hat{L}}_t\bigg],
\end{equation}
for every real-valued random variable $\eta$ in $\mathcal{L}^2(\mathcal{F}_T)$. Now, since $\phi^m\in\mathcal C^\infty([0,T])$, $I=\{\phi^m(t)\}_{t\in[0,T]}$ is an It\^o process satisfying the assumptions of Lemma \ref{L:WeakConv}. Therefore, $\int_{0^-}^T \phi^m(t)d\hat L_t^{n_k}$ converges to $\int_{0^-}^T \phi^m(t)d\hat L_t$ weakly in $\mathcal{L}^2(\mathcal{F}_T)$, that is
\[
\lim_{k\rightarrow+\infty}\mathbb E\bigg[\eta\int_{0^-}^T  \phi^m(t)d\hat L_t^{n_k}\bigg] = \mathbb E\bigg[\eta\int_{0^-}^T  \phi^m(t)d\hat L_t\bigg],
\]
for every real-valued random variable $\eta$ in $\mathcal{L}^2(\mathcal{F}_T)$. Then, by Cauchy's limit theorem on Cesàro means, it follows that
\[
\lim_{k\rightarrow+\infty}\mathbb E\bigg[\eta\int_{0^-}^T  \phi^m(t)d\bar{\hat L}_t^{n_k}\bigg] = \mathbb E\bigg[\eta\int_{0^-}^T  \phi^m(t)d\hat L_t\bigg].
\]
By \eqref{WeakConvCesaro_m} we deduce that a.s. $\int_{0^-}^T\phi^m(t) d\hat L_t=\int_{0^-}^T\phi^m(t) d\bar{\hat L}_t$, for every $m$ (since $\{\phi^m\}_m$ is countable the null set where the equality does not hold can be taken independent of $m$). Recalling that $\{\phi^m\}_m$ is convergence determining, we conclude that $\hat L$ and $\bar{\hat L}$ are equal.\\
To conclude the proof, fix an $\omega\in\Omega$ such that \eqref{WeakConvCesaro1} holds and consider the continuous function $\phi_t:=Y_t(\omega)-\xi_t(\omega)$. Then, for such a $\phi$ the convergence \eqref{WeakConvCesaro1} holds. Since \eqref{WeakConvCesaro1} holds for almost every $\omega\in\Omega$, we obtain, a.s.
\[
\int_{0^-}^T(Y_t-\xi_t)d\bar{\hat{L}}^{n_{k}}_t \rightarrow \int_{0^-}^T(Y_t-\xi_t)d\bar{\hat{L}}_t.
\]
By \eqref{WeakConvCesaro-1} and Cauchy's limit theorem on Cesàro means, we obtain, a.s.,
\[
\int_{0^-}^T(Y_t-\xi_t)d\bar{\hat{L}}^{n_k}_t \rightarrow 0.
\]
In conclusion, we find
\[
\int_{0^-}^T(Y_t-\xi_t)d\hat{L}_t=0.
\]
\end{proof}
We now give the main existence result of a solution of the MKV-RFBSDE system \eqref{MFG SDE}-\eqref{MFRBSDE}.
\begin{theorem}[\textit{Existence of a solution of the MKV-RFBSDE system}]\label{thm:main}
Suppose that Assumptions \ref{assumption:1A}, \ref{assumption:1B} and \ref{assumption: 1C} hold. There exists a quintuple $(X,Y,Z,A,L) \in \mathbb{S}^2\times\mathbb{S}^2\times\mathbb{H}^2\times\mathbb{A}^2\times\mathcal{V}$ solution to system \eqref{MFG SDE}-\eqref{MFRBSDE}.
\end{theorem}
\begin{proof}
Let $\Gamma\colon\mathcal V\rightarrow 2^{\mathcal V}$ be the set-valued map defined as follows: given $L\in\mathcal V$ then $\hat L\in\Gamma(L)$ if $\hat L\in\mathcal V$ and it holds that
\[
\int_{0^-}^T (Y_t-\xi_t)d\hat L_t=0 \qquad \text{ and } \qquad \int_{0^-}^T A_td\hat L_t=0,
\]
where $(Y,Z,A)\in\mathbb S^2\times\mathbb H^2\times\mathbb A^2$ is the unique solution of equation \eqref{RBSDE} corresponding to process $L$ (recall that the existence and uniqueness of $(Y,Z,A)$ follow from Proposition \ref{P:ExistUniq}). Then, by Lemma \ref{L:Compact}, Lemma \ref{L:Non-Empty}, Proposition \ref{P:Gamma} we see that we can apply the Kakutani--Fan--Glicksberg fixed-point theorem (see e.g. Corollary 17.55 in \cite{AliprantisBorder2006}) to the set-valued map $\Gamma$. It follows that there exists a fixed point $L\in\mathcal V$: $L\in\Gamma(L)$. This yields the existence of a quintuple $(X,Y,Z,A,L) \in\mathbb{S}^2\times\mathbb{S}^2\times\mathbb{H}^2\times\mathbb{A}^2\times\mathcal{V}$ solution to \eqref{MFG SDE}-\eqref{MFRBSDE}.
\end{proof}

\begin{remark}[\textit{Extension to the non-Markovian case}]\label{R:non-Markov}
The results reported in this section can be extended to the non-Markovian (path-dependent) setting, where the coefficients $b, \bar b, \sigma, \bar\sigma, f, \bar f, h$ do not depend on $x\in\mathbb R^d$ but on a path $x \in C([0,T];\mathbb R^d)$ (endowed with the supremum norm $\|\cdot\|_\infty$). We refer to $x_{\cdot \wedge t}$ as the path stopped at time $t$. Such an extension requires modifying Assumptions \ref{assumption:1A}, \ref{assumption:1B} and \ref{assumption: 1C} as follows.
\begin{itemize} 
\item The Lipschitz continuity of $b$, $\sigma$, $\bar b$, $\bar\sigma$ is formulated with respect to the supremum norm (here we report the Lipschitz property for $b$):
    \[
    |b(t,x,m)-b(t,x',m')|\leq K\big(\|x_{\cdot\wedge t}-x'_{\cdot\wedge t}\|_\infty + |m-m'|\big).
    \]
\item The polynomial growth condition of $f$ and $h$ becomes
    \[
    |f(t,x,0)|+|h(t,x,0)|+|\bar f(t,x)|\leq K\left(1+\|x_{\cdot\wedge t}\|_\infty^r\right).
    \]
\item The functions $\bar b,\bar\sigma,\bar f\colon[0,T]\times C([0,T];\mathbb R^d)\rightarrow\mathbb R^k$ are assumed of class $\mathcal{C}^{1,2}$ in the sense of functional It\^o calculus (see, e.g., \cite{dupire2019}). This entails that these functionals admit a continuous horizontal derivative and twice-continuous vertical derivatives. This regularity guarantees that the stochastic processes $(\bar b(t,X_{\cdot\wedge t}))_t$, $(\bar\sigma(t,X_{\cdot\wedge t}))_t$, $(\bar f(t,X_{\cdot\wedge t}))_t$ remain semimartingales, as a consequence of the functional It\^o formula. We also require that these functions, together with their derivatives, satisfy a polynomial growth condition as in \eqref{PolGrowthCond}, with $|x|$ replaced by $\|x_{\cdot\wedge t}\|_\infty$.
\end{itemize}
In such a non-Markovian setting, equations \eqref{MFG SDE} and \eqref{MFRBSDE} read as follows:
\[
X_t=X_0+\int_0^tb\big(s,X_{\cdot\wedge s},\mathbb E\big[\bar b(s,X_{\cdot\wedge s})L_s\big]\big)ds+\int_0^t\sigma\big(s,X_{\cdot\wedge s},\mathbb E\big[\bar\sigma(s,X_{\cdot\wedge s})L_s\big]\big)dW_s, \quad 0\leq t\leq T
\]
and
\[
\begin{cases}
\vspace{2mm}Y_t=\zeta+\int_t^T f\big(s,X_{\cdot\wedge s},\mathbb E\big[\bar f(s,X_{\cdot\wedge s})L_s\big]\big)ds + A_T-A_t - \int_t^T(Z_s,dW_s), \qquad 0\leq t\leq T, \\
\vspace{2mm}Y_t \geq \xi_t, \qquad 0\leq t\leq T, \\
\vspace{2mm}\int_{0}^T (Y_t-\xi_t) dA_t=0, \\
\vspace{2mm}\int_{0^-}^T(Y_t-\xi_t) dL_t=0, \\
\int_{0^-}^T A_t dL_t=0,
\end{cases}
\]
\[
\zeta:=g\bigg(X_{\cdot\wedge T},\mathbb E\bigg[\int_{0^-}^T\phi(T-s)dL_s\bigg]\bigg), \quad \xi_t:= h\bigg(t,X_{\cdot\wedge t},\mathbb E\bigg[\int_{0^-}^T\phi(t-s)dL_s\bigg]\bigg), \quad 0\leq t\leq T.
\]
\end{remark}

\subsection{Uniqueness}
\label{SubS:Uniqueness}

Suppose that Assumptions \ref{assumption:1A}, \ref{assumption:1B} and \ref{assumption: 1C} hold. By Theorem \ref{thm:main} we know that there exists a 5-tuple $(X,Y,Z,A,L)\in\mathbb{S}^2\times\mathbb S^2\times\mathbb H^2\times\mathbb A^2\times\mathcal V$ solution to system \eqref{MFG SDE}-\eqref{MFRBSDE}. In the present section we investigate the uniqueness of the 5-tuple $(X,Y,Z,A,L)$, for which we need the following additional assumption.

\begin{assumption}
\label{assum:uniquness}\quad
   \begin{enumerate}[i)]
    \item Suppose that $b=b(t,x)$ and $\sigma=\sigma(t,x)$ do not depend on their last argument, so that $X$ satisfies the following stochastic differential equation on $[0,T]$:
    \begin{equation}
    \label{eq:SDE without L}
    X_t=X_0+\int_0^tb(s,X_s)ds+\int_0^t\sigma(s,X_s)dW_s.
    \end{equation}
    \item The following monotonicity condition holds: for all $L, L' \in \mathcal{V}$,
    \[
    \begin{split}
    &\mathbb{E}\biggl[\int_0^T\big(f(t,X_t,\mathbb{E}[\bar{f}(t,X_t)L_t])-f(t,X_t,\mathbb{E}[\bar{f}(t,X_t)L'_t])\big)(L_t-L'_t)dt\\
    &-\int_{0^-}^T\bigg(h\bigg(t,X_t,\mathbb{E}\bigg[\int_{0^-}^T\phi(t-s)dL_s\bigg]\bigg)-h\bigg(t,X_t,\mathbb{E}\bigg[\int_{0^-}^T\phi(t-s)dL'_s\bigg]\bigg)\bigg)d(L-L')_t\bigg] \le 0,
    \end{split}
    \]
    with $X$ satisfying \eqref{eq:SDE without L}. Moreover, equality holds if and only if $L_t=L_t'$ for all $t \in [0,T]$ almost surely.
\end{enumerate} 
\end{assumption}

\begin{theorem}\label{T:Uniqueness}
   Suppose that Assumptions \ref{assumption:1A}, \ref{assumption:1B}, \ref{assumption: 1C} and \ref{assum:uniquness} hold. Let $(X,Y,Z,A,L)\in\mathbb S^2\times\mathbb{S}^2\times\mathbb{H}^2\times\mathbb{A}^2\times\mathcal V$ and $(X,Y',Z',A',L')\in\mathbb S^2\times\mathbb{S}^2\times\mathbb{H}^2\times\mathbb{A}^2\times\mathcal V$ be two solutions of system \eqref{eq:SDE without L}-\eqref{MFRBSDE}. Then, $(Y,Z,A,L)$ and $(Y',Z',A',L')$ are equal in $\mathbb{S}^2\times\mathbb{H}^2\times\mathbb{A}^2\times\mathcal V$.
\end{theorem}
\begin{proof}
    By Definition \ref{defn:OSequilibrium} and Theorem \ref{T:BSDE_MFG}, we have
    \[
    \begin{split}
    &\mathbb{E}\bigg[\int_0^Tf(t,X_t,\mathbb{E}[\bar{f}(t,X_t)L_t])(L_t-L'_t)dt-\int_{0^-}^Th\bigg(t,X_t,\mathbb{E}\bigg[\int_{0^-}^T\phi(t-s)dL_s\bigg]\bigg)d(L-L')_t\bigg] \ge 0, \\ &\mathbb{E}\bigg[\int_0^Tf(t,X_t,\mathbb{E}[\bar{f}(t,X_t)L'_t])(L'_t-L_t)dt-\int_{0^-}^Th\bigg(t,X_t,\mathbb{E}\bigg[\int_{0^-}^T\phi(t-s)dL'_s\bigg]\bigg)d(L'-L)_t\bigg] \ge 0.
    \end{split}
    \]
    Adding these two inequalities, we get
    \[
    \begin{split}&\mathbb{E}\bigg[\int_0^T\big(f(t,X_t,\mathbb{E}[\bar{f}(t,X_t)L_t])-f(t,X_t,\mathbb{E}[\bar{f}(t,X_t)L'_t])\big)(L_t-L'_t)dt\\
    &-\int_{0^-}^T\bigg(h\bigg(t,X_t,\mathbb{E}\bigg[\int_{0^-}^T\phi(t-s)dL_s\bigg]\bigg)-h\bigg(t,X_t,\mathbb{E}\bigg[\int_{0^-}^T\phi(t-s)dL'_s\bigg]\bigg)\bigg)d(L-L')_t\bigg]\ge 0.
    \end{split}
    \]
    Hence, by the monotonicity condition in item ii) of Assumption \ref{assum:uniquness}, it follows that the previous expectation is identically equal to zero. Using the strict part of item ii) in Assumption \ref{assum:uniquness}, we deduce that $L=L'$. Thus, $(Y,Z,A)$ and $(Y',Z',A')$ solve the same reflected backward stochastic differential equation \eqref{RBSDE}, so that from Proposition \ref{P:ExistUniq} we conclude that they coincide.
\end{proof}

\section{Properties of the process $L$ satisfying the Skorokhod conditions and link between $Y$ and randomized stopping}
\label{S:ProcessL}

In this section, we establish several properties of the process $L$ satisfying the two new Skorokhod conditions. In particular, these conditions lead to new results for classical optimal stopping problems without mean field interactions. Although optimality conditions for stopping times, corresponding to pure strategies, are well established (see, for instance, \cite{el2006aspects}), analogous conditions for randomized stopping strategies appear to be absent from the literature.
We prove that the two Skorokhod-type conditions arising in system~\eqref{MFG SDE}-\eqref{MFRBSDE} characterize the optimality of randomized stopping strategies (see Theorem~\ref{equiv}). We further show that these conditions yield bounds on any process $L$ satisfying them, expressed in terms of processes associated with pure strategies (see Theorem ~\ref{Prop:L}). Finally, for a given $L$, we establish an equality between the $Y$-component of the reflected BSDE and the value of an optimization problem over $\hat{L} \in \mathcal{V}$.

Suppose that Assumptions \ref{assumption:1A}, \ref{assumption:1B} and \ref{assumption: 1C} hold. By Theorem \ref{thm:main} we know that there exists a 5-tuple $(X,Y,Z,A,L)\in\mathbb{S}^2\times\mathbb S^2\times\mathbb H^2\times\mathbb A^2\times\mathcal V$ solution to system \eqref{MFG SDE}-\eqref{MFRBSDE}. Let
\begin{align}
\tau_{\min}&=\inf\big\{t\in[0,T]\colon Y_t=\xi_t\big\}, \label{tau_min} \\
\tau_{\max}&=\inf\big\{t\in[0,T]\colon A_t>0\big\}, \label{tau_max}
\end{align}
with $\inf\emptyset=T$, where we recall that $\xi_t= h\big(t,X_t,\mathbb E\big[\int_{0^-}^T\phi(t-s)dL_s\big]\big)$, for all $0\leq t\leq T$. Set also
\[
L_t^{\min} = \mathbf 1_{t<\tau_{\min}}, \qquad\qquad L_t^{\max} = \mathbf 1_{t<\tau_{\max}},
\]
for $0\leq t\leq T$.

\begin{theorem}[\textit{Bounds on the process $ L$}]\label{Prop:L}
Suppose Assumption \ref{assumption:1A} and \ref{assumption:1B} hold. Let $\hat L \in \mathcal V$, let $\hat X$ be the corresponding solution of \eqref{MFG SDE}, and let $(\hat Y, \hat Z, \hat A)$ be the corresponding solution to the reflected BSDE \eqref{RBSDE}. Let $L \in \mathcal V$ satisfying the two new Skorokhod conditions.
Then, it holds that, $\mathbb P$-a.s.,
\[
L_t^{\min} \leq L_t \leq L_t^{\max}, \qquad 0\leq t\leq T,
\]
or, equivalently, almost surely,
\[
L_t = 1, \quad t<\tau_{\min}, \qquad\qquad L_t=0, \quad t\geq\tau_{\max},
\]
where $\tau_{\min}$ and $\tau_{\max}$ are given by \eqref{tau_min} and \eqref{tau_max} respectively with $(\hat X, \hat Y, \hat Z, \hat A)$ instead of $(X,Y,Z,A)$.
\end{theorem}
\begin{proof}
For $r \in [0,1]$,
define 
$$G(r,\cdot):=\inf\{s\in[0,T]\colon L_s \leq r\}.$$
Let $\tau_r:=G(r, \cdot)$, which is a stopping time with respect to $\{\mathcal F_t\}_{t\in[0,T]}$ because of the right-continuity and adaptedness of $L$.
For every $u \in [0,T]$ we have the following equality:
\begin{align}\label{chst}
\mathbf{1}_{G(r,\cdot) \leq u}=\mathbf{1}_{L_u \leq r}.
\end{align}
Set $\hat \xi_t:=h(t,\hat X_t, \mathbb E[\int_{0^-}^T\phi(t-s)d\hat L_s])$, from the two constraints
\[
\int_{0^-}^T(\hat Y_t-\hat \xi_t)dL_t=0 \qquad \text{ and } \qquad  \int_{0^-}^T\hat A_tdL_t=0,
\]
we deduce by \eqref{chst} that
\[
\int_0^1 (\hat Y_{\tau_r}-\hat \xi_{\tau_r})dr=0 \qquad \text{ and } \qquad  \int_0^1 \hat A_{\tau_r}dr=0.
\]
Since $\hat Y_t\geq\hat \xi_t$ and $\hat A_t \geq 0$, for $0\leq t\leq T$, we derive that, $dr\otimes d\mathbb P$-a.e.,
$$\hat Y_{\tau_r}-\hat \xi_{\tau_r}=0, \qquad \hat A_{\tau_r}=0.$$
By right-continuity of $r \to \tau_r$, together with continuity of $t \to \hat Y_t-\hat \xi_t$ and $t \to \hat A_t$, we obtain, $\mathbb P$-a.s., for all $r\in[0,1)$,
$$\hat Y_{\tau_r}-\hat \xi_{\tau_r}=0,\qquad \hat A_{\tau_r}=0.$$
Therefore, by classical results, we deduce that 
$$ \tau_{\min} \leq \tau_r \leq \tau_{\max}, \qquad \text{for all } r \in [0,1).$$ This implies that $\mathbf{1}_{t<\tau_{\min}} \leq \mathbf{1}_{t<\tau_{r}} = \mathbf{1}_{r<L_t} \leq \mathbf{1}_{t<\tau_{\max}}$, $\mathbb P$-a.s., $t \in [0,T]$, $r \in [0,1)$. Therefore 
\[
\mathbf{1}_{t<\tau_{\min}} \leq \int_0^1\mathbf{1}_{r<L_t}dr \leq \mathbf{1}_{t<\tau_{\max}},\qquad \mathbb P\text{-a.s.},\,\,  \forall t \in [0,T].
\]
Hence $\mathbf{1}_{t<\tau_{\min}} \leq L_t \leq \mathbf{1}_{t<\tau_{\max}}$, $\mathbb P$-a.s., $t \in [0,T]$.
\end{proof}

\noindent For every $t\in[0,T]$, let $\mathcal{T}([t,T])$ be the set of $[t,T]$-valued stopping times with respect to the filtration $\{\mathcal F_s\}_{s\in[0,T]}$. 

We recall that by Theorem \ref{thm:main} we know that, under Assumptions \ref{assumption:1A}, \ref{assumption:1B}, \ref{assumption: 1C}, there exists a 5-tuple $(X,Y,Z,A,L)\in\mathbb{S}^2\times\mathbb S^2\times\mathbb H^2\times\mathbb A^2\times\mathcal V$ solution to system \eqref{MFG SDE}-\eqref{MFRBSDE}. Now, given $L \in \mathcal{V}$, from Proposition 2.3 in \cite{el1981aspects} the following probabilistic representation for $Y$ holds:
\begin{equation}\label{Y_t=esssup_tau}
Y_t=\underset{\tau \in \mathcal{T}([t,T])}{\text{ess\,sup}}\,\mathbb{E}\biggl[\int_t^\tau f(s,X_s, \mathbb{E}[\bar{f}(s,X_s)L_s])ds+h\bigg(\tau, X_\tau,\mathbb{E}\bigg[\int_{0^-}^T\phi(\tau-s)dL_s\bigg]\bigg)\bigg|\mathcal{F}_t\biggr],
\end{equation}
for all $0\leq t\leq T$. Then, we have the following result.
\begin{theorem}[\emph{Link between $Y$ and optimal stopping in randomized strategies}]\label{equalvalues}
Suppose that Assumptions \ref{assumption:1A} and \ref{assumption:1B} hold. Then, for every $L \in \mathcal V$ and the associated solution $(X,Y,Z,A) \in \mathbb S^2 \times \mathbb S^2 \times \mathbb H^2 \times \mathbb A^2$ to system \eqref{MFG SDE}-\eqref{MFRBSDE}, it holds that
\begin{align}\label{rand}
Y_t=\underset{\hat L \in \mathcal{V}_t}{\textup{ess\,sup}}\,\mathbb{E}\bigg[\int_t^T f(s,X_s, \mathbb{E}[\bar{f}(s,X_s)L_s])\hat L_sds-\int_{t^-}^Th\bigg(s, X_s,\mathbb{E}\bigg[\int_{0^-}^T\phi(s-r)dL_r\bigg]\bigg)d\hat L_s\bigg|\mathcal{F}_t\bigg],
\end{align}
for all $0\leq t\leq T$, where $\mathcal{V}_t=\{ \hat L \in \mathcal{V}\colon\hat L_{t^-}=1\}$. 
\end{theorem}

\begin{proof}
Let us first show that 
$$Y_t \leq \underset{\hat L \in \mathcal{V}_{t}}{\text{ess\,sup}}\,\mathbb{E}\biggl[\int_t^T f(s,X_s, \mathbb{E}[\bar{f}(s,X_s)L_s])\hat L_sds-\int_{t^-}^Th\bigg(s, X_s,\mathbb{E}\bigg[\int_{0^-}^T\phi(s-r)dL_r\bigg]\bigg)d\hat L_s\bigg|\mathcal{F}_t\biggr].$$
Fix $\tau \in \mathcal{T}([t,T])$. Define the process $L^\tau_s:=\mathbf{1}_{s <\tau}$, for $s\in[0,T]$. It is clear that $L^{\tau} \in \mathcal{V}_{t}$ and that using $L^\tau$ we recover the conditional expectation on the right-hand side of \eqref{Y_t=esssup_tau}. This shows the validity of the above inequality.\\
Let us now show the reverse inequality.
Fix $t\in[0,T]$ and $\hat L \in \mathcal{V}_{t}$. For $r \in [0,1]$ and $\omega\in\Omega$,
define 
\begin{equation}\label{tau_r}
G(r,\omega):=\inf\{s \geq t:\,\, \hat L_s(\omega) \leq r\}.
\end{equation}
Let $\tau_r(\omega)=G(r,\omega)$, $\omega\in\Omega$. Notice that $\tau_r$ is a stopping time because of the right continuity of $\hat L$ (see Lemma 4.8 in \cite{revuz2013continuous}). Moreover, for every $r\in[0,1]$, $u\in[t,T]$, we have the following equality
\begin{align}
\mathbf{1}_{G(r,\cdot) \leq u}=\mathbf{1}_{\hat L_u \leq r}.
\end{align}
In particular, we get
\[
\begin{split}
\mathbb{E}\bigg[\int_t^Tf(s,X_s,\mathbb{E}[\bar{f}(s,X_s)L_s])\hat L_sds\bigg| \mathcal{F}_t\bigg]&=\mathbb{E}\bigg[\int_t^Tf(s,X_s,\mathbb{E}[\bar{f}(s,X_s)L_s])\bigg(\int_0^1\mathbf{1}_{\tau_r >s}dr\bigg)ds\bigg|\mathcal{F}_t\bigg]\\
&=\int_0^1\mathbb{E}\bigg[\int_t^{\tau_r}f(s,X_s,\mathbb{E}[\bar{f}(s,X_s)L_s])ds\bigg|\mathcal{F}_t\bigg]dr
\end{split}
\]
and, by Proposition 4.9 in \cite{revuz2013continuous}
\[ \mathbb{E}\bigg[\int_{t^-}^Th\bigg(s,X_s,\mathbb E\bigg[\int_{0^-}^T\phi(s-u)dL_u\bigg]\bigg)(-d\hat{L}_s)\bigg|\mathcal{F}_t\bigg]=\int_0^1\mathbb E \bigg[h\bigg(\tau_r,X_{\tau_r}, \mathbb E\bigg[\int_{0^-}^T\phi(\tau_r-u)dL_u\bigg]\bigg)\bigg|\mathcal{F}_t\bigg]dr.\]
Since $\tau_r$ is a stopping time for every $r \in [0,1]$, by \eqref{Y_t=esssup_tau} we get
\[
\mathbb{E}\bigg[\int_t^{\tau_r}f(s,X_s,\mathbb{E}[\bar{f}(s,X_s)L_s])ds+h\bigg(\tau_r,X_{\tau_r},\mathbb{E}\bigg[\int_{0^-}^T\phi(\tau_r-u)dL_u\bigg]\bigg)\bigg|\mathcal F_t\bigg]\le Y_t.
\]
Thus,
\[
\begin{split}
\mathbb{E}\bigg[\int_t^Tf(s,X_s, \mathbb{E}[\bar{f}(s,X_s)L_s])\hat L_sds-\int_{t^-}^Th\bigg(s,X_s,\mathbb{E}\bigg[\int_{0^-}^T\phi(s-u)dL_u\bigg]\bigg)&d\hat L_s\bigg|\mathcal{F}_t\bigg] \\
&\le \int_0^1Y_tdr=Y_t.
\end{split}
\]
From the arbitrariness of $\hat L$, we can pass to the essential supremum and get the reverse inequality.
\end{proof}

\noindent We now establish the following \textit{necessary} and \textit{sufficient} conditions of optimality of a process $\hat{L} \in \mathcal{V}_t$ for $Y_t$ given by \eqref{rand}.
\begin{theorem}[\textit{Equivalence between the Skorokhod conditions and optimality of $\hat{L}$ for \eqref{rand}}]\label{equiv} Suppose that Assumptions \ref{assumption:1A} and \ref{assumption:1B} hold. Let $L\in\mathcal V$, let $X$ be the solution of \eqref{MFG SDE}, and let $(Y,Z,A)$ be the corresponding solution to the reflected BSDE \eqref{RBSDE}. Let $Y$ also satisfy the representation \eqref{rand}. Set $\xi_t:=h(t,X_t,\mathbb E[\int_{0^-}^T\phi(t-s)dL_s])$. A process $\hat{L} \in \mathcal{V}_t$ is optimal  for \eqref{rand} if and only if \begin{align}\label{Skorot}
\int_{t^-}^T (Y_s-\xi_s)d\hat{L}_s=0;\qquad\qquad\,\,\int_{t^-}^T A^t_sd\hat{L}_s=0, \quad \text{a.s.}
\end{align}
with $A^t_s:=A_s-A_t$ for $s \geq t$.
\end{theorem}
\begin{proof}
Fix $0 \leq t \leq T$. Let $\Lambda_t\colon\mathcal V\times\mathcal V_t\rightarrow \mathcal L^2(\mathcal F_t)$ be the map defined as follows:
\begin{align*}
\Lambda_t(L,\hat L)&:=\mathbb{E}\biggl[\int_t^T f(s,X_s, \mathbb{E}[\bar{f}(s,X_s)L_s])\hat L_s ds \\
&\quad-\int_{t^-}^T h\bigg(s, X_s,\mathbb{E}\bigg[\int_{0^-}^T\phi(s-r)dL_r\bigg]\bigg)d\hat L_s\biggr|\mathcal{F}_t\biggr]. \notag
\end{align*}
We show that
\begin{equation}\label{E[Y_t]}
\mathbb E[Y_t] = \mathbb{E}\left[\Lambda_t(L,\hat{ L})\right] + \mathbb E\bigg[\int_t^T \hat{ L}_{s^-} dA_s-\int_{t^-}^T (Y_s-\xi_s)d\hat{ L}_s\bigg]. 
\end{equation}
Indeed, by the integration by parts formula, we have
\[
Y_T \hat{ L}_T - Y_t \hat{ L}_{t^-} = \int_{t^-}^T Y_s d\hat{ L}_s + \int_t^T \hat{ L}_{s^-} dY_s.
\]
Recalling that $\hat{ L}_T = 0$ and $\hat{ L}_{t^-} = 1$, and using the backward stochastic differential equation \eqref{MFRBSDE} satisfied by $Y$, we obtain
\begin{equation}\label{eq:Y_0}
-Y_t = \int_{t^-}^T Y_s d\hat{ L}_s - \int_t^T f(s, X_s, \mathbb{E}[\bar{f}(s, X_s)L_s])\hat{ L}_s ds + \int_t^T \hat{ L}_{u^-} Z_u dW_u - \int_t^T \hat{ L}_{u^-} dA_u.
\end{equation}
Since $\hat L$ is bounded and $Z\in\mathbb H^2$, the stochastic integral is a martingale, then \eqref{E[Y_t]} follows taking the expectation in the above equality. Assume that $\hat{ L} \in \mathcal{V}_t$ is an optimal control. Then, from \eqref{E[Y_t]}, we deduce that 
\begin{align}
\mathbb E\bigg[\int_t^T \hat{ L}_{s^-} dA_s-\int_{t^-}^T (Y_s-\xi_s)d\hat{ L}_s\bigg]=0.
\end{align}
The two random variables $\int_t^T \hat{L}_s\,d A_s$ and $-\int_{t^-}^T (Y_s-\xi_s)\,d\hat{L}_s$ are both non-negative, so that both are equal to zero almost surely. Finally, applying the integration by parts formula to \(A_s^t\hat{L}_s\), yields
\[
A_T^t\hat{L}_T - A_{t^-}^t\hat{L}_{t^-}
=
\int_{t^-}^T A_s^t\,d\hat{L}_s
+
\int_t^T \hat{L}_{s^-}\,dA_s^t.
\]
Since \(\hat{L}_T=0\), \(A_{t^-}^t:=A_t^t=0\), and \(\int_t^T \hat{L}_{u^-}\,dA_u^t=0\) (noting that $dA_s^t=dA_s$), it follows that $\int_{t^-}^T A_u^t\,d\hat{L}_u = 0$. Thus, the conditions \eqref{Skorot} are satisfied.
Assume now that the converse holds. From \eqref{E[Y_t]}, and using similar arguments as above, we deduce that 
\begin{equation}
\mathbb E[Y_t] = \mathbb{E}\left[\Lambda_t(L,\hat{ L})\right]. 
\end{equation}
Since from equation \eqref{rand} we have $Y_t \geq \Lambda_t(L, \hat{ L})$ a.s. we conclude, using the above relation, that 
$$Y_t=\Lambda_t(L, \hat{ L})\,\ \text{ a.s.},$$ i.e. $\hat{L} \in \mathcal{V}_t$ is an optimal control.
\end{proof}

\section{Existence and extremal equilibria via Tarski's fixed-point theorem}
\label{S:Tarski}

In this section, we study the existence of solutions to thereflected forward-backward McKean--Vlasov stochastic differential equation~\eqref{MFG SDE}-\eqref{MFRBSDE} using a second approach based on Tarski's fixed-point theorem~\cite{Tarski1955}. We also develop learning algorithms for the \textit{minimal} and \textit{maximal} solutions of system~\eqref{MFG SDE}-\eqref{MFRBSDE} and establish their convergence. As shown below, our approach based on reflected BSDEs provides a simpler alternative for proving the monotonicity of the relevant correspondences, without relying on Topkis' theorem, as is commonly done in the existing literature.

\paragraph{Preliminaries.} We begin by introducing a lattice structure on $\mathcal V$.\\
Let $L, L' \in \mathcal{V}$ and set (recall that we always work with the representative processes $L,L'\in\mathcal V$ which are $[0,1]$-valued, adapted, non-increasing, c\`adl\`ag, and $L_T=L_T'=0$)
\[
L \le_{\mathcal{V}} L' \qquad \text{if and only if} \qquad L_t \le L'_t, \quad \text{for all }t\in[0,T],\;\text{almost surely}.
\]
We also define
\begin{equation}
\label{complete lattice cond}
    (L \wedge L')_t:=L_t \wedge L'_t \,\, \text{ and } \,\, (L \vee L')_t:=L_t \vee L'_t, \qquad \text{for all }t\in[0,T].
\end{equation}
Then, the set $(\mathcal{V},\le_{\mathcal{V}})$ is a partially ordered set and the operations $\wedge,\vee$ provide a lattice structure on $\mathcal{V}$, which is compatible with the order relation $\leq_{\mathcal{V}}$. Moreover, the essential infimum/supremum of an arbitrary family $\{L^i\}_i$ of càdlàg processes in $\mathcal V$ admits a càdlàg modification $L$ that belongs to $\mathcal V$. This implies that $(\mathcal{V},\le_\mathcal{V})$ is a complete lattice.\\
In the present section, we impose the following assumptions.
\begin{assumption}
\label{assum:Tarski_1}
    Let $K$ be a non-negative constant  and $r \in [1, q/2]$.
    \begin{itemize}
        \item [i)] The functions $f,h$ are measurable. Moreover, we suppose that for every $R >0$ there exists a constant $C_R>0$ such that
        \begin{align*}
        |f(t,x,m)|&\leq K(1+|x|+|m|^r)\\
        |h(t,x,w)| & \le C_R(1+|x|),\\
        |\bar f(t,0)| &\le K,
        \end{align*}
        for any $(t,x,m,w) \in [0,T] \times \mathbb{R}^d\times \mathbb{R}^k\times \mathbb{R}$, with $|w|\leq R$.
        \item [ii)] $h$ is a continuous function and $\phi \in \mathcal C^1([-T,T])$.
        \item [iii)] There exists a constant $K \ge 0$ such that
        \begin{align*}
            |f(t,x,m)-f(t,x',m)| &\leq K |x-x'|,\\
            |\bar{f}(t,x)-\bar{f}(t,x')| &\leq K|x-x'|,
        \end{align*}
        for all $(t,m) \in [0,T] \times \mathbb{R}^k$, $x,x' \in \mathbb{R}^d$.
    \end{itemize}
\end{assumption}
\noindent Notice that under Assumption \ref{assumption:1A} and Assumption \ref{assum:Tarski_1}.i)-ii), for every $L\in\mathcal V$ there exists a unique solution $X\in\mathbb S^2$ to equation \eqref{MFG SDE} (with input process $L$) and a unique solution $(Y,Z,A)\in\mathbb S^2\times\mathbb H^2\times\mathbb A^2$ to equation \eqref{RBSDE} (with input process $L$).
\begin{assumption}
\label{assum:Tarski_2}
\quad
    \begin{enumerate}[i)]
        \item The functions $\bar{b}:[0,T] \times \mathbb{R}^d\to \mathbb{R}^k$ and $\bar{f}\colon[0,T] \times \mathbb{R}^d\to \mathbb{R}^k$ are component-wise non-negative. Moreover, we suppose that $\phi'\colon[-T,T]\rightarrow\mathbb R$ is non-negative. Finally, we assume that $\sigma=\sigma(t)$, i.e. $\sigma$ only depends on the time variable.
        \item Let $b=(b_1,\ldots,b_d)$, $\bar b=(\bar b_1,\ldots,\bar b_k)$ and $\bar f=(\bar f_1,\ldots,\bar f_k)$. The following monotonicity properties hold:
        \begin{align*}
            b_i(t,x,m)&\le b_i(t,x',m'), \\
            \bar b_j(t,x)&\le\bar b_j(t,x'), \\
            f(t,x,m)&\le f(t,x',m'), \\
            \bar f_j(t,x)&\le \bar f_j(t,x'), \\
            h(t,x,w)&\le h(t,x',w'),
        \end{align*}
        for all $i=1,\ldots,d$, $j=1,\ldots,k$, $t\in[0,T]$, $x,x'\in\mathbb R^d$, $m,m'\in\mathbb R^k$, $w,w'\in\mathbb R$, with $x\leq x'$ (component-wise), $m\leq m'$ (component-wise), $w\leq w'$.
        \item  Let $L, L' \in \mathcal{V}$. Suppose that there exist the corresponding solutions $X,X'$ to equation \eqref{MFG SDE} and $(Y,Z,A),(Y',Z',A')$ to equation \eqref{RBSDE}. We assume that whenever $L \leq_{\mathcal{V}} L'$ then it holds that, almost surely, \[
        Y_t-h\bigg(t,X_t,\mathbb{E}\bigg[\int_{0^-}^T\phi(t-s)dL_s\bigg]\bigg) \leq Y'_t-h\bigg(t,X'_t,\mathbb{E}\bigg[\int_{0^-}^T\phi(t-s)dL'_s\bigg]\bigg),  \quad 0\leq t\leq T.
        \]
    \end{enumerate}
\end{assumption}
\begin{remark}
    Notice that Assumption \ref{assum:Tarski_2}-iii) is not empty in the case of the obstacle dependent on $L$. Let $d=k=1$, suppose that the state process $X$ is independent of $L$ and the functions involved have the following forms:
    \[
        h(t,x,w)=x+w,\qquad
        \phi(t)=\textup{e}^t\qquad
        f(t,x,m)=\lambda m, \qquad
        \bar{f}(t,x)=1,
    \]
    for some $\lambda>0$. Then, by It\^o's formula,
    \[
    \xi_t:=h\bigg(t,X_t, \mathbb{E}\bigg[\int_{0^-}^T\phi(t-s)dL_s\bigg]\bigg)=X_t+\textup{e}^t\mathbb{E}\bigg[-1+\int_0^T\textup{e}^{-s}L_sds\bigg].
    \]
    To verify the monotonicity of the gap $Y_t-\xi_t$, we recall, from Proposition 2.3 of \cite{el1981aspects}, that
    \[
    Y_t-\xi_t=\underset{\tau \in \mathcal{T}([t,T])}{\textup{ess\,sup}}\,\mathbb{E}\bigg[\lambda\int_{t}^\tau\mathbb{E}[L_s]ds+(\textup{e}^\tau-\textup{e}^t)\mathbb{E}\bigg[-1+\int_0^T\textup{e}^{-s}L_sds\bigg]+X_\tau-X_t\Big|\mathcal{F}_t\bigg]:=\underset{\tau \in \mathcal{T}([t,T])}{\textup{ess\,sup}}\,R(\tau, L).
    \]
    As a consequence, since $\lambda>0$ and $\tau \geq t$, if $L \leq_{\mathcal{V}} L'$ we have that $R(\tau, L) \leq R(\tau, L')$ almost surely, for any stopping time $\tau \in\mathcal{T}([t,T])$. Taking the essential supremum over $\tau\in\mathcal{T}([t,T])$ on both sides, yields
    \[
    Y_t-\xi_t=  \underset{\tau \in \mathcal{T}([t,T])}{\textup{ess\,sup}}\,R(\tau, L)\leq \underset{\tau \in \mathcal{T}([t,T])}{\textup{ess\,sup}}\,R(\tau, L')=Y'_t-\xi'_t.
    \]
\end{remark}
\noindent We now give the following assumption, which will be shown to imply the uniqueness of the best response.
\begin{assumption}
\label{assum:Tarski_3}
 Let $L \in \mathcal{V}$. Suppose that there exist a solution $X$ to equation \eqref{MFG SDE} (with input process $L$) and a solution $(Y,Z,A)$ to equation \eqref{RBSDE} (with input process $L$). Set $\xi_t:=h(t, X_t, \mathbb E[\int_{0^-}^{T} \phi(t-s)dL_s])$, for $t\in[0,T]$. We assume that
        \[
        \tau_{\min}=\tau_{\max},
        \]
        where $\tau_{\min}=\inf\{t\in[0,T]\colon Y_t=\xi_t\}$ and $\tau_{\max}=\inf\{t\in[0,T]\colon A_t>0\}$, with $\inf\emptyset=T$.
\end{assumption}
\noindent Assumption \ref{assum:Tarski_3} implies the following uniqueness result of the best response.

\begin{lemma}\label{L:UniqBestResponse}
Suppose that Assumptions \ref{assumption:1A}, \ref{assum:Tarski_1}.i)-ii) and \ref{assum:Tarski_3} hold. For every $L \in \mathcal V$, there exists a unique process $\hat L\in\mathcal V$ satisfying $\int_{0^-}^T(Y_t-\xi_t)d\hat{L}_t=0$ and $\int_{0^-}^TA_td\hat{L}_t=0$, where $(Y,Z,A)$ solution of \eqref{RBSDE} (with input process $L$) and $\xi_t=h(t,X_t,\mathbb E[\int_{0^-}^T\phi(t-s)dL_s])$. Moreover, such a process $\hat L$ is given by $\hat L_t=\mathbf 1_{t<\tau_{\min}}$, for $t\in[0,T]$.
\end{lemma}
\begin{proof}
Let $\hat L_t=\mathbf 1_{t<\tau_{\min}}$, for $t\in[0,T]$. We first verify that $\hat L$ satisfies the two constraints. To this end, notice that the measure $-d\hat L_t$ is the Dirac measure at $\tau_{\min}$, namely $-d\hat L_t=\delta_{\tau_{\min}}(dt)$. Then, regarding the first constraint, we have
\[
\int_{0^-}^T(Y_t-\xi_t)d\hat L_t = - (Y_{\tau_{\min}}-\xi_{\tau_{\min}}) = 0,
\]
where the second equality follows from the continuity of $Y$ and $\xi$, and also from the definition of $\tau_{\min}$. Concerning the second constraint, we have
\[
\int_{0^-}^TA_td\hat L_t = - A_{\tau_{\min}}.
\]
By Assumption \ref{assum:Tarski_3}, we know that $\tau_{\min}=\tau_{\max}$, which yields $A_{\tau_{\min}}=A_{\tau_{\max}}$. Since $A_{\tau_{\max}}=0$, the second constraint follows.\\
Finally, from the same proof as in Theorem \ref{Prop:L} we deduce that, given a process $\tilde L\in\mathcal V$ satisfying the two constraints $\int_{0^-}^T(Y_t-\xi_t)d\tilde{L}_t=0$ and $\int_{0^-}^TA_td\tilde{L}_t=0$, it holds that
\[
\tilde L_t = 1, \quad t<\tau_{\min}, \qquad\qquad \tilde L_t=0, \quad t\geq\tau_{\max}.
\]
Since by Assumption \ref{assum:Tarski_3} we know that $\tau_{\min}=\tau_{\max}$, the claim follows.
\end{proof}

\begin{remark}\label{R:Assumption1-i)}
    We report sufficient conditions for the validity of Assumption \ref{assum:Tarski_3}. In this remark we suppose that Assumptions \ref{assumption:1A} and \ref{assumption:1B} hold, as well as that $b=b(t,x)$, $\sigma=\sigma(t,x)$ and $h=h(t,x)$ are independent of their last argument, so in particular $X$ satisfies the stochastic differential equation
    \begin{equation}
    \label{eq:SDE without L_remark}dX_t=b(t,X_t)dt+\sigma(t,X_t)dW_t.
    \end{equation}
    Let $F(t,x)=f(t,x,\mathbb E[\bar f(t,X_t)L_t])$, for all $(t,x)\in[0,T]\times\mathbb R^d$. In the present Markovian setting, there exists a continuous function $v\colon[0,T]\times\mathbb R^d\rightarrow\mathbb R$ satisfying $Y_t=v(t,X_t)$, for $t\in[0,T]$ (see e.g. Lemma 8.4 in \cite{el1981aspects}). Then, we define the continuation region $\mathcal C=\{(t,x)\in[0,T]\times\mathbb R^d\colon v(t,x)>h(t,x)\}$ and the stopping region $\mathcal S=\{(t,x)\in[0,T]\times\mathbb R^d\colon v(t,x)=h(t,x)\}$. Notice that $\tau_{\min}$ can be equivalently written as $\tau_{\min}=\inf\{t\in[0,T]\colon (t,X_t)\in\mathcal S\}$. Now, consider the following set of assumptions.
    \begin{assumption}
    \label{assum:uniq_opt_stop}\quad
    \begin{enumerate}[\upshape a)]
        \item Let $X=\{X_t\}_{t\in[0,T]}$ be the solution to equation \eqref{eq:SDE without L_remark}; then, for every $t\in[0,T]$, we assume that $X_t$ has a distribution absolutely continuous with respect to the Lebesgue measure $dx$ on $\mathbb R^d$. For a sufficient condition, see Theorem 2.3.1 in \cite{nualart2006malliavin}.
        \item $h\in C^{1,2}([0,T]\times\mathbb R^d)$ with derivatives satisfying a polynomial growth condition.
        \item $(dt\otimes dx)(\partial\mathcal C)=0$.
        \item There exists a constant $\ell>0$ such that
        \begin{equation}\label{Lh}
        \mathcal Lh(t,x) + F(t,x) \leq - \ell, \qquad (t,x)\in\mathcal S,
        \end{equation}
        where
        \[
        \mathcal L h(t,x) = \partial_t h(t,x) + (b(t,x),\partial_x h(t,x)) + \frac{1}{2}\textup{tr}\big(\sigma(t,x)\sigma^{\scriptscriptstyle\intercal}(t,x)\partial_{xx}^2h(t,x)\big).
        \]
        \item $\tau_{\min}=\inf\{t\in[0,T]\colon (t,X_t)\in\textup{int}\,\mathcal S\}$ with $\inf \emptyset := \infty$, where $\textup{int}\,\mathcal S$ is the interior of the set $\mathcal S$.
    \end{enumerate}
    \end{assumption}
    \noindent We claim that under the above set of assumptions the equality $\tau_{\min}=\tau_{\max}$ holds. To this end, set $\xi_t=h(t,X_t)$, for $t\in[0,T]$. Since $h\in C^{1,2}([0,T]\times\mathbb R^d)$, by It\^o's formula, we have
    \[
    d\xi_t = \mathcal Lh(t,X_t)dt + (\partial_x h(t,X_t)\,\sigma(t,X_t),dW_t).
    \]
    Then
    \[
    d(Y_t-\xi_t) = -(F(t,X_t)+\mathcal Lh(t,X_t))dt+(Z_t-\partial_x h(t,X_t)\,\sigma(t,X_t),dW_t)-dA_t.
    \]
    On the other hand, applying Tanaka--Meyer's formula to $(Y-\xi)^+$, yields (notice that the term $\mathbf{1}_{Y_t>\xi_t}dA_t$ is zero as a consequence of the Skorokhod condition)
    \begin{align*}
    d(Y_t-\xi_t)^+ &= -\mathbf{1}_{Y_t>\xi_t}(F(t,X_t)+\mathcal Lh(t,X_t))dt \\
    &\quad +\mathbf{1}_{Y_t>\xi_t}(Z_t-\partial_x h(t,X_t)\,\sigma(t,X_t),dW_t)+\frac{1}{2}d\tilde L_t,
    \end{align*}
    where $\tilde L$ is the local time of the semimartingale $Y-\xi$ at zero. Since $Y-\xi \equiv (Y-\xi)^+$, by combining the two above equalities and by uniqueness of the semimartingale decomposition, we get
    \[
    dA_t+\frac{1}{2}d\tilde{L}_t = -\mathbf{1}_{Y_t=\xi_t}(F(t,X_t)+\mathcal Lh(t,X_t))dt.
    \]
    We notice that, under the above assumptions, the local time $\tilde L$ is indistinguishable from zero (see e.g. Theorem 6 in \cite{jacka1993local}). Hence, we find
    \begin{align}\label{Jacka}
    A_t=-\int_0^t\mathbf{1}_{Y_s=\xi_s}(F(s,X_s)+\mathcal Lh(s,X_s))ds, \qquad 0\leq t\leq T.
    \end{align}
    From the definition of $\tau_{\min}$, we obtain
    \[
    A_t=-\int_{\tau_{\min}}^{t\vee\tau_{\min}}\mathbf{1}_{Y_s=\xi_s}(F(s,X_s)+\mathcal Lh(s,X_s))ds, \qquad 0\leq t\leq T.
    \]
    By \eqref{Lh}, we get
    \[
    A_{\tau_{\max}}\geq\ell\int_{\tau_{\min}}^{\tau_{\max}}\mathbf{1}_{Y_s=\xi_s}ds.
    \]
    Since $A$ is continuous, we have $A_{\tau_{\max}}=0$. Hence, we conclude that
    \begin{equation}\label{tau_min_tau_max}
    \int_{\tau_{\min}}^{\tau_{\max}}\mathbf{1}_{Y_s=\xi_s}ds = 0, \qquad \text{almost surely}.
    \end{equation}
    We now deduce from the above equality that $\tau_{\max}=\tau_{\min}$ almost surely. Let $\omega\in\Omega$ be such that \eqref{tau_min_tau_max} holds and assume by contradiction that $\tau_{\max}(\omega)>\tau_{\min}(\omega)$. Choose $n\ge 1$ such that $\tau_{\min}(\omega) + \frac{1}{n} < \tau_{\max}(\omega)$. By assumption e), there exists $t_n(\omega)$ such that
    \[
    \tau_{\min}(\omega)\le t_n(\omega) < \tau_{\min}(\omega) + \frac{1}{n}
    \quad\text{and}\quad
    \bigg(t_n(\omega), X_{t_n(\omega)}(\omega)\bigg)\in \operatorname{int}\mathcal S.
    \]
    In particular, $t_n(\omega) < \tau_{\max}(\omega)$.
    Since $\operatorname{int}\mathcal S$ is open, there exists $r(\omega)>0$ such that
    \[
    \Big(\big(t_n(\omega)-r(\omega),\,t_n(\omega)+r(\omega)\big)\cap[0,T]\Big)\ \times\
    \Big\{x\in\mathbb{R}^d\colon |x-X_{t_n(\omega)}(\omega)|<r(\omega)\Big\}
    \subset \mathcal S.
    \]
    Because the path $t\mapsto X_t(\omega)$ is continuous, there exists $\delta(\omega)\in(0,r(\omega))$ such that
    for all $s\in[t_n(\omega),\,t_n(\omega)+\delta(\omega)]$,
    \[
    |X_s(\omega)-X_{t_n(\omega)}(\omega)|<r(\omega)
    \quad\text{and}\quad
    |s-t_n(\omega)|<r(\omega).
    \]
    Hence, for all $s\in[t_n(\omega),\,t_n(\omega)+\delta(\omega)]$ we have $(s,X_s(\omega))\in\mathcal S$, so that $Y_s(\omega) = v(s,X_s(\omega)) = h(s,X_s(\omega)) = \xi_s(\omega)$. Therefore
    \[
    \int_{\tau_{\min}(\omega)}^{\tau_{\max}(\omega)} \mathbf{1}_{Y_s(\omega)=\xi_s(\omega)}\,ds
    \ge
    \int_{t_n(\omega)}^{\min\{t_n(\omega)+\delta(\omega),\,\tau_{\max}(\omega)\}} 1\,ds
    =
    \min\{\delta(\omega),\,\tau_{\max}(\omega)-t_n(\omega)\}
    > 0,
    \]
    because $t_n(\omega)<\tau_{\max}(\omega)$.
    This contradicts \eqref{tau_min_tau_max}. Hence $\tau_{\max}(\omega)=\tau_{\min}(\omega)$ for every $\omega$ such that \eqref{tau_min_tau_max} holds. It follows that $\tau_{\max} = \tau_{\min}$ almost surely.
\end{remark}

We introduce the best reply set-valued map $\Gamma\colon\mathcal V\rightarrow 2^{\mathcal V}$. We write it as follows $\Gamma=\Gamma_2\circ\Gamma_1$, where $\Gamma_1\colon\mathcal V\rightarrow\mathbb S^2\times\mathbb S^2\times\mathbb H^2\times\mathbb A^2$ and $\Gamma_2\colon\Gamma_1(\mathcal V)\rightarrow 2^{\mathcal V}$. The map $\Gamma_1$ is defined as follows: given $L\in\mathcal V$, then $\Gamma_1(L):=(X,Y,Z,A)$, with $X \in \mathbb{S}^2$ solving equation \eqref{MFG SDE} (with input process $L$) and $(Y,Z,A) \in \mathbb{S}^2\times\mathbb H^2\times\mathbb A^2$ solving the reflected backward stochastic differential equation \eqref{RBSDE} (with input process $L$). On the other hand, given $(X,Y,Z,A)\in\Gamma_1(\mathcal V)$, $\hat L\in\Gamma_2((X,Y,Z,A))$ if $\hat L\in\mathcal V$ and the two constraints \eqref{hatL} are satisfied.  

We introduce now several order relations which will be used in this section. For  $X, X' \in \mathbb{S}^2$, we set
\[
X \leq_{\mathbb{S}^2} X' \quad \text{if and only if} \quad X^i_t \leq X'^{,i}_t, \,\,\, \text{for all } i \in \{1,\dots,d\} \,\, \text{for all } t \in [0,T] \text{ almost surely},
\]
that is the inequality is understood component-wise, since $X, X'$ are stochastic processes in $\mathbb{R}^d$.

We also define the following order on $\mathbb{A}^2$: for $A, A' \in \mathbb{A}^2$, we set
\[
dA \leq_{\mathbb{A}^2} dA' \quad \text{if and only if} \quad A_t - A_s \leq A'_t - A'_s \ \text{a.s. for all } 0 \leq s \leq t \leq T.
\]
Note that the sets $(\mathbb{S}^2, \leq_{\mathbb{S}^2})$ and $(\mathbb{A}^2, \leq_{\mathbb{A}^2})$ are partially ordered sets. To alleviate the notation, we define $S:=(X,Y,Z,A)$. 
We introduce the following orders on the set of solutions $\mathcal S$ of system \eqref{RBSDE}. For $S=(X,Y,Z,A) \in \mathcal{S}$ and $S'=(X',Y',Z',A') \in \mathcal{S}$, we say that $S \leq_{\mathcal{S}} S'$ (resp. $S \leq_{\Sigma} S'$) if $X \leq_{\mathbb{S}^2} X'$ and $Y \leq_{\mathbb{S}^2} Y'$ (resp. if $X \leq_{\mathbb{S}^2} X'$, $Y \leq_{\mathbb{S}^2} Y'$, and $dA \geq_{\mathbb{A}^2} dA'$). The set $(\mathcal{S}, \leq_{\mathcal{S}})$ (resp. $(\mathcal{S}, \leq_{\Sigma})$) is a partially ordered set.

Finally we introduce an order relation on $R$, the set of solutions of system \eqref{MFG SDE}-\eqref{MFRBSDE}. Given two solutions $(X,Y,Z,A,L)$ and $(\bar{X}, \bar{Y}, \bar{Z},\bar{A}, \bar{L})$, if $L \leq_{\mathcal{V}} \bar{L}$, set $(X,Y,Z,A, L) \leq_R ( \bar{X},\bar{Y}, \bar{Z}, \bar{A}, \bar{L})$. With such an order, the set of solutions of system \eqref{MFG SDE}-\eqref{MFRBSDE} is a partially ordered set.
\begin{definition}
\label{def:max_min_solution}
 $(\tilde{X},\tilde{Y},\tilde{Z},\tilde{A},\tilde{L}) \in \mathbb{S}^2\times\mathbb{S}^2\times\mathbb{H}^2\times\mathbb{A}^2\times\mathcal V$ is a maximal (resp. minimal) solution of the system \eqref{MFG SDE}-\eqref{MFRBSDE} if for any other solution $(X,Y,Z,A,L)$ we have $(X,Y,Z,A,L) \leq_R (\tilde{X},\tilde{Y},\tilde{Z},\tilde{A},\tilde{L})$ (resp. $(\tilde{X},\tilde{Y},\tilde{Z},\tilde{A},\tilde{L})\leq_R (X,Y,Z,A,L)$).
\end{definition}
\paragraph{Properties of the best response maps.}
Before proceeding with the  presentation of the main results of this section, we need a preliminary lemma.
\begin{lemma}[Comparison result on $X$]\label{L:X_le_X'}
    Suppose that Assumptions \ref{assumption:1A}, \ref{assum:Tarski_2}.i)-ii) hold. Let $L, L' \in \mathcal{V}$ and consider the corresponding solutions $X,X'$ to equation \eqref{MFG SDE}. If $L \leq_{\mathcal{V}} L'$ then $X \leq_{\mathbb{S}^2} X'$.
\end{lemma}
\begin{proof}
    We begin noting that, since $\sigma=\sigma(t)$, the process $X_t-X_t'$ is a finite variation (absolutely continuous) process. Applying the standard chain rule to $|(X_t-X'_t)^+|^2$ (here the positive part is applied component-wise to the vector $X_t-X_t'$), and taking the expectation, we find (we denote by $C_K$ a non-negative constant only depending on the Lipschitz constant $K$ of $b$ and $\bar b$)
    \[
    \begin{split}
    &\mathbb{E}[|(X_t-X'_t)^+|^2]= 2\mathbb{E}\bigg[\int_0^t(X_s-X'_s)^+\big(b(s,X_s,\mathbb{E}[\bar{b}(s,X_s)L_s])-b(s,X'_s,\mathbb{E}[\bar{b}(s,X'_s)L'_s])\big)ds\bigg]\\
    &= 2\mathbb{E}\bigg[\int_0^t(X_s-X'_s)^+\big(b(s,X_s,\mathbb{E}[\bar{b}(s,X_s)L_s])-b(s,X_s',\mathbb{E}[\bar{b}(s,X_s)L_s])\\
    &\,\,\,\,\,\,\,\,\,\,\,\,+b(s,X_s',\mathbb{E}[\bar{b}(s,X_s)L_s])-b(s,X_s',\mathbb{E}[\bar{b}(s,X_s)L'_s])\\
 &\,\,\,\,\,\,\,\,\,\,\,\,+b(s,X_s',\mathbb{E}[\bar{b}(s,X_s)L_s'])-b(s,X_s',\mathbb{E}[\bar{b}(s,X_s\vee X_s')L'_s])\\
    &\,\,\,\,\,\,\,\,\,\,\,\,+b(s,X_s',\mathbb{E}[\bar{b}(s,X_s\vee X_s')L'_s])-b(s,X'_s,\mathbb{E}[\bar{b}(s,X'_s)L'_s])\big)ds\bigg]\\
    &\le C_K\mathbb{E}\bigg[\int_0^t|(X_s-X'_s)^+|^2ds\bigg]+C_K\mathbb{E}\bigg[\int_0^t(X_s-X'_s)^+|\mathbb{E}[\bar{b}(s,X_s\vee X_s')L'_s]-\mathbb{E}[\bar{b}(s,X'_s)L'_s]|ds\bigg]\\
    &\le C_K\mathbb{E}\bigg[\int_0^t|(X_s-X'_s)^+|^2ds\bigg]+C_K\int_0^t\mathbb{E}[(X_s-X'_s)^+]\mathbb{E}[|X_s\vee X_s'-X'_s|]ds\\
    &\le C_K\mathbb{E}\bigg[\int_0^t|(X_s-X'_s)^+|^2ds\bigg],
    \end{split}
    \]
    where we used the Lipschitz continuity of $b$ and $\bar b$ in Assumption \ref{assumption:1A}, moreover in the first inequality we used that $b(s,X_s',\mathbb{E}[\bar{b}(s,X_s)L_s])\leq b(s,X_s',\mathbb{E}[\bar{b}(s,X_s)L'_s])$ and $b(s,X'_s,\mathbb{E}[\bar{b}(s, X_s)L'_s]) \leq b(s, X'_s,\mathbb{E}[\bar{b}(s,X_s\vee X'_s)L'_s])$, since $L \leq_{\mathcal{V}} L'$ by assumption and $X \leq_{\mathbb{S}^2} X \vee X'$. Thus, by Gronwall's lemma, we deduce that $(X_t-X'_t)^+=0$, for all $t \in [0,T]$ almost surely. This allows us to conclude that $X \leq_{\mathbb{S}^2} X'$.
\end{proof}

\noindent We start by establishing a monotonicity property of $\Gamma_1$.
\begin{proposition}[Monotonicity of $\Gamma_1$]
\label{P:increasing maps}
    Under Assumptions \ref{assumption:1A}, \ref{assum:Tarski_1}.i)-ii) and \ref{assum:Tarski_2}.i)-ii), the map $\Gamma_1$ is non-decreasing: for any $L, L' \in \mathcal{V}$ with $L \leq_{\mathcal{V}} L'$, it holds that $\Gamma_1(L) \leq_{\mathcal{S}} \Gamma_1(L')$. Suppose in addition that $b$, $\sigma$ are independent of $m$, and $h$ is independent of $w$. Then $\Gamma_1$ is non-decreasing in  the following sense: $\Gamma_1(L) \leq_{\Sigma} \Gamma_1(L')$.
\end{proposition}
\begin{proof}
    Let $L,L' \in \mathcal{V}$ such that $L \leq_{\mathcal{V}} L'$. We consider the corresponding solutions $X,X'$ to equation \eqref{MFG SDE} and $(Y,Z,A),(Y',Z',A')$ to system \eqref{RBSDE}. We have proved in the previous lemma that $X \leq_{\mathbb{S}^2} X'$.\\
    Since $\phi'$ is non-negative, using integration by parts formula, we get, for every $t\in[0,T]$,
    \begin{align*}
    \mathbb{E}\bigg[\int_{0^-}^T\phi(t-s)dL_s\bigg]&=\mathbb{E}\bigg[-\phi(t)+\int_0^TL_s\phi'(t-s)ds\bigg] \\
    &\leq \mathbb{E}\bigg[-\phi(t)+\int_0^TL_s'\phi'(t-s)ds\bigg]=\mathbb{E}\bigg[\int_{0^-}^T\phi(t-s)dL'_s\bigg].
    \end{align*}
Let $\xi_t:=h\bigg(t,X_t,\mathbb{E}\bigg[\int_{0^-}^T\phi(t-s)dL_s\bigg]\bigg)$ and $\xi_t':=h\bigg(t,X'_t,\mathbb{E}\bigg[\int_{0^-}^T\phi(t-s)dL'_s\bigg]\bigg)$, for $t\in[0,T]$. Then, by the previous computation and Assumption \ref{assum:Tarski_2}-ii), we have
    \[
    \begin{split}\xi_t-\xi_t'&=h\bigg(t,X_t,\mathbb{E}\bigg[\int_{0^-}^T\phi(t-s)dL_s\bigg]\bigg)-h\bigg(t,X_t,\mathbb{E}\bigg[\int_{0^-}^T\phi(t-s)dL'_s\bigg]\bigg)\\
    &\quad+h\bigg(t,X_t,\mathbb{E}\bigg[\int_{0^-}^T\phi(t-s)dL'_s\bigg]\bigg)-h\bigg(t,X'_t,\mathbb{E}\bigg[\int_{0^-}^T\phi(t-s)dL'_s\bigg]\bigg)\le 0.
    \end{split}
    \]
    Furthermore, define $f_t:=f(t,X_t,\mathbb{E}[\bar{f}(t,X_t)L_t])$ and $f_t':=f(t,X'_t,\mathbb{E}[\bar{f}(t,X'_t)L'_t])$, for $t \in  [0,T]$. Since $L \leq_{\mathcal{V}} L'$, Lemma \ref{L:X_le_X'} implies $X \leq_{\mathbb{S}^2} X'$. Consequently, given the non-negativity (Assumption \ref{assum:Tarski_2}-i)) and monotonicity (Assumption \ref{assum:Tarski_2}-ii)) of $\bar{f}$, the non-negativity of the processes $L$ and $L'$ and the fact that $L \leq_{\mathcal{V}} L'$, it follows that
    \[
    \bar{f}(t,X_t)L_t \leq \bar{f}(t,X'_t)L'_t \quad \forall t \in [0,T] \,\,\, \text{a.s}
    \]
    Thus by Assumption \ref{assum:Tarski_2}-ii), we can conclude that 
    $f_t \leq f'_t$ for all $t \in  [0,T]$ almost surely. Finally, applying the comparison principle from \cite{el1981aspects}, we can conclude that $Y \leq_{\mathbb{S}^2} Y'$, which together with $X \leq_{\mathbb{S}^2} X'$, yields $\Gamma_1(L)\leq_{\mathcal{S}} \Gamma_1(L')$.
\noindent Under the additional assumption that $b$ and $\sigma$ are independent of $m$, and that $h$ is independent of $w$, by using Theorem 4.2 in \cite{peng2005smallest}, we obtain $dA \geq_{\mathbb{A}^2} dA'$, which further implies that $\Gamma_1(L) \leq_{\Sigma} \Gamma_1(L')$.
\end{proof}
\noindent Using the same arguments as in Theorem \ref{Prop:L}, we have that, for all $S \in \mathcal{S}$, $\Gamma_2(S) \subset \{L \in \mathcal{V}: L^{\text{min},S} \leq_{\mathcal{V}} L \leq_{\mathcal{V}} L^{\text{max},S}\}$,
with $L^{\text{min},S}:=\mathbf{1}_{t <\tau_{\text{min},S}}$ and $L^{\text{max},S}:=\mathbf{1}_{t <\tau_{\text{max},S}}$,
where
$$\tau_{\text{min},S}:=\inf\{ t\in[0,T]: Y_t^{S}=\xi_t^{S}\}$$
and
$$\tau_{\text{max},S}:=\inf\{ t\in[0,T]: A_t^{S}>0\}.$$
Denote by 
$$\mathbb{V}(S):= \{L \in \mathcal{V}: L^{\text{min},S} \leq_{\mathcal{V}} L \leq_{\mathcal{V}} L^{\text{max},S}\}.$$
\begin{lemma}[Properties of $\Gamma_2(S)$]\label{gamma2}
Suppose that Assumptions \ref{assumption:1A}, \ref{assum:Tarski_1}.i)-ii) and \ref{assum:Tarski_2} hold. Then
\begin{enumerate}[i)]
\item For all $S \in \mathcal{S}$, define the map $\underline{R} (S):=\Tilde{L}^S$, with $\Tilde{L}^S_t:=\essinf\nolimits_{\Gamma_2(S)} L_t $. Then $\underline{R}(S) \in \Gamma_2(S)$.
\item The map $\underline{R}$ is increasing with respect to $S$ in the following sense: if $S^1 \leq_{\mathcal{S}} S^2$ then $\underline{R}(S^1)\leq_{\mathcal{V}} \underline{R}(S^2)$.
\end{enumerate}
Assume now that $b$, $\sigma$ are independent of $m$ and $h$ is independent of $w$. Then 
\begin{enumerate}[iii)]
\item For all $S \in \mathcal{S}$, define the map $\overline{R}(S) := \bar{L}^S$, with $\bar{L}^S_t := \esssup_{\Gamma_2(S)} L_t$. Then $\overline{R}(S) \in \Gamma_2(S)$.
\item[iv)] The map $\overline{R}$ is increasing with respect to $S$ in the following sense: if $S^1 \leq_{\Sigma} S^2$, then $\overline{R}(S^1) \leq_{\mathcal{V}} \overline{R}(S^2)$.
\end{enumerate}
\end{lemma}
\begin{proof}
i) Fix $S \in \mathcal{S}$.  Denote the process $L^S$ as follows: for all $t \in [0,T]$, $L_t^S:=\essinf_{\mathbb{V}(S)} L_t $. Then, since $\Gamma_2(S) \subset \mathbb{V}(S)$, we have, 
\begin{align}\label{eq1}
\Tilde{L}^S \geq_{\mathcal{V}} L^S.\,\ 
\end{align}
By the complete lattice property of the set $\mathcal{V}$, since $\mathbb{V}(S) \subset \mathcal{V}$ and by definition of $\mathbb{V}(S)$, we have
\begin{align}\label{eq2}
L^{\text{min},S} \leq _{\mathcal{V}} L^{S}.
\end{align}
Now, we have that 
$L^{\text{min},S} \in \Gamma_2(S)$ by the same arguments as in the proof of Lemma \ref{L:UniqBestResponse}. Therefore, 
\begin{align}\label{eq3}
L^{\text{min},S} \geq_{\mathcal{V}} \Tilde{L}^S.
\end{align}
From \eqref{eq1}, \eqref{eq2}, \eqref{eq3}, we finally deduce that, for all $t \in [0,T]$,
\begin{align}
L_t^{\text{min},S}=\Tilde{L}^S_t,\,\, \text{a.s.},
\end{align}
which leads to $\underline{R}(S) \in \Gamma_2(S)$.

ii) Let $S:=(X,Y,Z,A),  S':=(X',Y',Z',A') \in \Gamma_1(\mathcal V)$ such that $S$ is associated to $L$ and $S'$ to $L'$, such that $L \leq_{\mathcal{V}} L'$. By the previous point i),  we know that $\underline{R}(S)=L^{\text{min},S}$ and $\underline{R}(S')=L^{\text{min},S'}$.
Moreover, it holds that $ L^{\text{min},S}_t=\mathbf{1}_{t<\tau_{\text{min},S}}$ and $ L^{\text{min},S'}_t=\mathbf{1}_{t <\tau'_{\text{min}, S'}}$. We recall that $\tau_{\text{min},S}=\inf\{t \in [0,T]: Y_t=\xi_t\}$ and $\tau_{\text{min}, S'}=\inf\{t \in [0,T]: Y'_t=\xi'_t\}$, with $\xi$ and $\xi'$ as in the proof of Lemma \ref{P:increasing maps}.
  Fix $\omega \in \Omega.$ Define
    \[
    \Xi(\omega)=\{t \in [0,T]: Y_t(\omega)=\xi_t(\omega)\} \qquad \text{and} \qquad \Xi'(\omega)=\{t \in [0,T]: Y'_t(\omega)=\xi'_t(\omega)\}.
    \]
    Let $t \in \Xi'(\omega)$, then $Y'_t(\omega)=\xi'_t(\omega)$ and we also know that $Y_t(\omega) \le Y'_t(\omega)$. By Assumption \ref{assum:Tarski_2}, we know that $Y_t(\omega)-\xi_t(\omega) \leq Y'_t(\omega)-\xi'_t(\omega)=0$. Since $Y_t(\omega) \geq \xi_t(\omega)$, we conclude that $Y_t(\omega)=\xi_t(\omega)$. This implies that $t \in \Xi(\omega)$. Hence $\Xi'(\omega) \subseteq \Xi(\omega)$. Since $\tau_{\text{min},S}(\omega)=\inf \Xi(\omega)$ and $\tau_{\text{min},S'}(\omega)=\inf \Xi'(\omega)$, by the properties of the infimum we conclude that $\tau_{\text{min}}(\omega) \le \tau'_{\text{min}}(\omega)$, i.e. $L^{\text{min},S} \leq_{\mathcal{V}} L^{\text{min},S'}$.

iii) Recalling that $L^{\max, S}$ is optimal (see Theorem 2.43 in \cite{el2006aspects}), the proof follows the same steps as in item i) 

iv) Let $S:=(X,Y,Z,A),  S':=(X',Y',Z',A') \in \mathcal{S}$ such that $S$ is associated to $L$, $S'$ to $L'$, with $L \leq_{\mathcal{V}} L'$. 
Recall that, by item $iii)$ we have $\overline{R}(S)=L^{\text{max},S}$ and $\overline{R}(S')=L^{\text{max},S'}$.
By Proposition \ref{P:increasing maps}, we deduce that $A'_t \leq A_t$ a.s. for all $0 \leq t \leq T$, which implies that
$$  L^{\text{max},S} \leq_{\mathcal{V}}  L^{\text{max},S'}.$$ The conclusion follows.
\end{proof}

\begin{lemma}[Properties of $\Gamma$]
\label{L: properties of Gamma}
Suppose that Assumptions \ref{assumption:1A}, \ref{assum:Tarski_1}.i)-ii) and \ref{assum:Tarski_2} hold. Then
\begin{enumerate}[i)]
\item For all $L \in \mathcal{V}$, define the map 
$\underline{\mathcal{R}}(L) := \tilde{L}^L$, 
with $\tilde{L}^L_t := \essinf_{\Gamma(L)} \hat{L}_t$. 
Then $\underline{\mathcal{R}}(L) \in \Gamma(L)$.
\item The map $\underline{\mathcal{R}}$ is increasing with respect to $L$, i.e., if $L^1 \leq_{\mathcal{V}} L^2$ then $\underline{\mathcal{R}}(L^1) \leq_{\mathcal{V}} \underline{\mathcal{R}}(L^2)$.
\end{enumerate}
Assume, in addition, that $b$ and $\sigma$ are independent of $m$ and $h$ is independent of $w$. Then:
\begin{enumerate}[iii)]
\item For all $L \in \mathcal{V}$, define the map 
$\overline{\mathcal{R}}(L) := \bar{L}^L$, 
with $\bar{L}^L_t := \esssup_{\Gamma(L)} \hat{L}_t$. 
Then $\overline{\mathcal{R}}(L) \in \Gamma(L)$.
\item[iv)] The map $\overline{\mathcal{R}}$ is increasing with respect to $L$, i.e., if $L^1 \leq_{\mathcal{V}} L^2$ then $\overline{\mathcal{R}}(L^1) \leq_{\mathcal{V}} \overline{\mathcal{R}}(L^2)$.
\end{enumerate}
\end{lemma}
\begin{proof}
The proof follows directly from the proof of Lemma \ref{gamma2}, by using that 
$\underline{\mathcal{R}}(L)=\underline{R}(S^L)$ (resp. $\overline{\mathcal{R}}(L)=\overline{R}(S^L)$) and that $\Gamma_2(S^L)=\Gamma(L)$, with $S^L=\Gamma_1(L)$, and from Lemma \ref{P:increasing maps}, which gives that if  $L^1 \leq_{\mathcal{V}} L^2$ then $S^{L^1} \leq_{\mathcal{S}} S^{L^2}$.
\end{proof}

\subsection{Existence of Extremal Solutions to System~\eqref{MFG SDE}-\eqref{MFRBSDE}}
In this subsection, we first establish the existence of minimal and maximal solutions to system~\eqref{MFG SDE}-\eqref{MFRBSDE}. We then investigate properties of the solution set under suitable assumptions and construct learning algorithms whose convergence is subsequently proved.

\begin{theorem}[\textit{Existence of extremal solutions and properties of the of solutions.}]\label{T:Tarski}
    Under Assumptions \ref{assumption:1A}, \ref{assum:Tarski_1}.i)-ii) and \ref{assum:Tarski_2}, the following statement holds true.
    \begin{enumerate}[i)]
        \item
        The set of solutions to system \eqref{MFG SDE}-\eqref{MFRBSDE} is non-empty. In particular, there exists a minimal solution $(\tilde X^{\text{min}},\tilde Y^{\text{min}},\tilde Z^{\text{min}},\tilde A^{\text{min}},\tilde L^{\text{min}})$. If, in addition, we suppose that $b$ and $\sigma$ are independent of $m$, and that $h$ is independent of $w$, then there also exists a maximal solution $(\tilde X^{\text{max}},\tilde Y^{\text{max}},\tilde Z^{\text{max}},\tilde A^{\text{max}},\tilde L^{\text{max}})$.
        \item Under the additional Assumption \ref{assum:Tarski_3} we have that the set of solutions to system \eqref{MFG SDE}-\eqref{MFRBSDE} is a complete lattice.
    \end{enumerate}
   \end{theorem}
\begin{proof}
i). We know that $(\mathcal{V}, \leq_{\mathcal{V}})$ is a complete lattice (see (\ref{complete lattice cond})). By Lemma \ref{L: properties of Gamma}, for $L \in \mathcal{V}$ we have that $\underline{\mathcal{R}}(L) \in \Gamma(L)$ (resp. $\overline{\mathcal{R}}(L) \in \Gamma(L)$), $\underline{\mathcal{R}}$ (resp. $\overline{\mathcal{R}}$)  is increasing in the order $\leq_{\mathcal{V}}$.

To simplify notation, for a subset $\mathcal{A} \subseteq \mathcal{V}$, we define $\essinf_{L \in \mathcal{A}} L:= (\essinf_{L \in \mathcal{A}} L_t)_{t \in [0,T]}$, which is by convention  identified with its c\`adl\`ag representative.

By Tarski's theorem \cite{Tarski1955}, we get that the set of fixed points of $\underline{\mathcal{R}}$ (resp. $\overline{\mathcal{R}}$) is a non-empty, complete lattice. Since any such fixed point $L$ uniquely determines the 4-tuple $(X,Y,Z,A)$, we deduce that the set $R$ of solutions to system \eqref{MFG SDE}-\eqref{MFRBSDE} is non-empty.

We also get by Tarski's theorem that the set of fixed points of $\underline{\mathcal{R}}$ has an infimum which is attained, i.e. there exists $L^\star \in \{L:L= \underline{\mathcal{R}}(L)\}$ such that $L^\star=\essinf \{L:L= \underline{\mathcal{R}}(L)\}$. Moreover, the set of pre-fixed points of $\{L: \underline{\mathcal{R}}(L)\leq_{\mathcal{V}} L\}$ has an infimum. By Tarski's theorem these two coincide. 
 We have that the set of fixed points of $\underline{\mathcal{R}}$ is contained in $\mathcal{A}:=\{L \in \mathcal{V}: L \in \Gamma(L)\}$, thus $\essinf \mathcal{A}\leq_{\mathcal{V}} \essinf\{L:L=\underline{\mathcal{R}}(L)\}=L^\star$. Moreover, $\mathcal{A} \subseteq \{L: \underline{\mathcal{R}}(L)\leq_{\mathcal{V}} L\}$, thus $L^\star=\essinf\{L:\underline{\mathcal{R}}(L) \leq_{\mathcal{V}}L\}\leq_{\mathcal{V}}\essinf\mathcal{A}$. We can conclude that $\essinf\mathcal{A}=L^\star$. Hence we can conclude that $L^\star$ is the minimal fixed point and that the MKV-RFBSDE system admits a \textit{minimal} solution. Under the additional assumptions on $b,\sigma$ and $h$, the same arguments can be used to deduce the existence of a maximal solution.\\
ii) Under the uniqueness assumption on the best response, we apply the same argument as above to the map $\Gamma:\mathcal{V} \to \mathcal{V}$, and by Tarski we derive that $\mathcal{A}$ is a non-empty complete lattice.
Let pr a map which associates to each fixed point $L$ of $\Gamma$ a solution of the system \eqref{MFG SDE}-\eqref{MFRBSDE}, i.e. $\text{pr}: L \to (X,Y,Z,A,L)$. Since pr  is an order-preserving isomorphism, it preserves the lattice structure (see Lemma 2.27 and Theorem 2.31 \cite{davey2002introduction}). Hence, as the set of fixed points of $\Gamma$ is a complete lattice, it follows that the set of solutions of the MKV-RFBSDE \eqref{MFG SDE}-\eqref{MFRBSDE} is a complete lattice as well. Furthermore, $\text{pr}(\essinf L)=\inf \text{pr}(L)$ and $\text{pr}(\esssup L)=\sup \text{pr}(L)$ (see Lemma 2.27 \cite{davey2002introduction}), which implies that there exists a minimal (resp. maximal solution) to the MKV-RFBSDE \eqref{MFG SDE}-\eqref{MFRBSDE} given by $(\tilde X^{min},\tilde Y^{min},\tilde Z^{min},\tilde A^{min},\tilde L^{min})$ 
(resp. $(\tilde X^{max},\tilde Y^{max},\tilde Z^{max},\tilde A^{max},\tilde L^{max})$).
\end{proof}

\paragraph{Learning algorithms.} We define inductively two sequences of processes as follows:
\begin{enumerate}[1)]
    \item Set $\underline{L}^0=0$ and, for $n\ge1$, let $\underline{L}^n=\underline{\mathcal{R}}(\underline{L}^{n-1})$. Then, for $n\geq1$, define $(\underline{X}^n, \underline{Y}^n, \underline{Z}^n, \underline{A}^n)$ as the solution of the following system:
    \begin{equation}\label{Sequence_underline}
    \begin{cases}
    d\underline{X}^n_t=b(t,\underline{X}^n_t,\mathbb{E}[\bar{b}(t,\underline{X}^n_t)\underline{L}^{n-1}_t])dt+\sigma(t)dW_t\\
    -d\underline{Y}^n_t=f(t,\underline{X}^n_t, \mathbb{E}[\bar{f}(t,\underline{X}^n_t)\underline{L}^{n-1}_t])dt-\underline{Z}^n_tdW_t+d\underline{A}^n_t, \\
    \underline X_0^n=X_0, \quad \underline Y_T^n=h\bigg(T,\underline{X}^n_T,\mathbb{E}\bigg[\int_{0^-}^T\phi(T-s)d\underline{L}^{n-1}_s\bigg]\bigg), \\
    \underline{Y}^n_t \ge h\bigg(t,\underline{X}^n_t,\mathbb{E}\bigg[\int_{0^-}^T\phi(t-s)d\underline{L}^{n-1}_s\bigg]\bigg), \\
    \int_0^T\bigg(\underline{Y}^n_t-h\bigg(t,\underline{X}^n_t,\mathbb{E}\bigg[\int_{0^-}^T\phi(t-s)d\underline{L}^{n-1}_s\bigg]\bigg)\bigg)d\underline{A}^n_t=0.
    \end{cases}
    \end{equation}
    \item Set $\overline{L}^0=1_{[0,T)}$ and, for $n\ge1$, let $\overline{L}^n=\overline{\mathcal{R}}(\overline{L}^{n-1})$. Then, for $n\geq1$, define $(\overline{X}^n, \overline{Y}^n, \overline{Z}^n, \overline{A}^n)$ as the solution to the following system:
    \begin{equation}\label{Sequence_overline}
    \begin{cases}
    d\overline{X}^n_t=b(t,\overline{X}^n_t)dt+\sigma(t)dW_t\\
    -d\overline{Y}^n_t=f(t,\overline{X}^n_t, \mathbb{E}[\bar{f}(t,\overline{X}^n_t)\overline{L}^{n-1}_t])dt-\overline{Z}^n_tdW_t+d\overline{A}^n_t, \\
    \overline X_0^n=X_0, \quad \overline Y_T^n=h(T,\overline{X}^n_T), \\
    \overline{Y}^n_t \ge h(t,\overline{X}^n_t), \\
    \int_0^T\big(\overline{Y}^n_t-h(t,\overline{X}^n_t)\big)d\overline{A}^n_t=0.
    \end{cases}
    \end{equation}
\end{enumerate}
\begin{theorem}[\textit{Convergence of the learning algorithms}]\label{T:converging sequences}
    Under Assumptions \ref{assumption:1A}, \ref{assumption:1B}, \ref{assum:Tarski_1} and \ref{assum:Tarski_2} the following statements hold true.
    \begin{enumerate}
    \item[i)] The sequence $\{\underline{L}^n\}_n$ in \eqref{Sequence_underline} is non-decreasing, $\underline{L}^{n+1} \geq_{\mathcal{V}} \underline{L}^n$, and $\underline{L}^n$ converges to $\tilde{L}^{\text{min}}$ in $\mathbb{H}^2$.
    \item [ii)] The sequence $(\underline{X}^n, \underline{Y}^n, \underline{Z}^n, \underline{A}^n)$ is non-decreasing, it converges to $(\tilde{X}^{min}, \tilde{Y}^{min}, \tilde{Z}^{min}, \tilde{A}^{min})$ in $\mathbb{S}^2\times \mathbb{S}^2\times\mathbb{H}^2\times\mathbb{S}^2$.
    \end{enumerate}
    Assume, in addition, that $b, \sigma$ are independent of $m$ and $h$ is independent of $w$. Then
    \begin{enumerate}
        \item [iii)] The sequence $\{\overline{L}^n\}_n$ in \eqref{Sequence_overline} is non-increasing, $\overline{L}^{n+1} \leq_{\mathcal{V}} \overline{L}^n$, and $\overline{L}^n$ converges to  $\tilde{L}^{\text{max}}$ in $\mathbb{H}^2$.
    \item [iv)] The sequence $(\overline{X}^n, \overline{Y}^n, \overline{Z}^n, \overline{A}^n)$ is non-increasing, it converges to $(\tilde{X}^{max}, \tilde{Y}^{max}, \tilde{Z}^{max}, \tilde{A}^{max})$ in $\mathbb{S}^2\times \mathbb{S}^2\times\mathbb{H}^2\times\mathbb{S}^2$.
    \end{enumerate}
\end{theorem}
\begin{proof}
We only provide the proof of items i)-ii), as iii)-iv) follows by the same arguments. We split the proof of item i) into four steps.\\
\textsc{Step 1}. In this step, we show the monotonicity of the sequence $\{\underline{L}^n\}_n$ by induction. Since $\underline{L}^0=0$, we have, using the monotonicity of $\underline{\mathcal{R}}$,
\[
0=\underline{L}^0\leq_{\mathcal V}\underline{\mathcal{R}}(\underline{L}^0)=\underline{L}^1
\]
Now, suppose that $\underline{L}^{n-1}\leq_{\mathcal{V}}\underline{L}^{n}$. Applying the map $\overline{\mathcal{R}}$ to both sides, we get
\[
\underline{L}^{n}=\underline{\mathcal{R}}(\underline{L}^{n-1})\leq_{\mathcal{V}}\underline{\mathcal{R}}(\underline{L}^{n}) =\underline{L}^{n+1}.
\]
This shows that the sequence $\{\underline{L}^n\}_n$ is a non-decreasing sequence, bounded from above by $\mathbf 1_{[0,T)}$. We then define the process
\begin{equation}
\label{eq:limit L}
\tilde{L}=\sup_n \underline{L}^n = \lim_n\underline{L}^n.
\end{equation}
By Lebesgue's dominated convergence theorem we have that $\underline{L}^n$ converges strongly in $\mathbb H^2(\mathbb R)$ to $\tilde{L} \in \mathcal{V}$.

\vspace{2mm}

\noindent\textsc{Step 2}.
In this step we prove that $(\underline{X}^n,\underline{Y}^n, \underline{Z}^n, \underline{A}^n)$ converges to the unique solution of the following system:
\[
\begin{cases}
    d\tilde{X}_t=b(t,\tilde{X}_t,\mathbb{E}[\bar{b}(t,\tilde{X}_t)\tilde{L}_t])dt+\sigma(t)dW_t\\
    -d\tilde{Y}_t=f(t,\tilde{X}_t,\mathbb{E}[\bar{f}(t,\tilde{X}_t)\tilde{L}_t])dt-\tilde{Z}_tdW_t+d\tilde{A}_t, \\
    \tilde X_0=X_0, \quad \tilde Y_T=h\bigg(T,\tilde{X}_T,\mathbb{E}\bigg[\int_{0^-}^T\phi(T-s)d\tilde{L}_s\bigg]\bigg), \\
    \tilde{Y}_t \ge h\bigg(t,\tilde{X}_t,\mathbb{E}\bigg[\int_{0^-}^T\phi(t-s)d\tilde{L}_s\bigg]\bigg), \\
    \int_0^T\bigg(\tilde{Y}_t-h\bigg(t,\tilde{X}_t,\mathbb{E}\bigg[\int_{0^-}^T\phi(t-s)d\tilde{L}_s\bigg]\bigg)\bigg)d\tilde{A}_t=0,
\end{cases}
\]
with $\tilde{L}\in\Gamma_2\big((\tilde{X},\tilde{Y}, \tilde{Z}, \tilde{A})\big)$, i.e. $\tilde{L}\in\Gamma(\tilde{L})$.\\
We begin by showing the convergence of the sequence $\{\underline{X}^n\}_n$. By Proposition \ref{P:SDE}, we have, for some non-negative constant $C_2$,
\[
\mathbb{E}\bigg[\sup_{0 \le t \le T}|\tilde{X}_t-\underline{X}^n_t|^2\bigg] \le C_2\int_0^T\bigg|\mathbb{E}[\bar{b}(t,\tilde{X}_t)\tilde{L}_t]-\mathbb{E}[\bar{b}(t,\tilde{X}_t)\underline{L}^{n-1}_t]\bigg|^2dt.
\]
By \eqref{eq:limit L}, Assumption \ref{assumption:1A}, Proposition \ref{P:SDE}, and Lebesgue's dominated convergence theorem, we get
\[
\begin{split}
\int_0^T\bigg|\mathbb{E}[\bar{b}(t,\tilde{X}_t)\tilde{L}_t]-\mathbb{E}[\bar{b}(t,\tilde{X}_t)&\underline{L}^{n-1}_t]\bigg|^2dt \le \int_0^T\mathbb{E}[|\bar{b}(t,\tilde{X}_t)|^2]\mathbb{E}[|\tilde{L}_t-\underline{L}^{n-1}_t|^2]dt \underset{n\rightarrow\infty}{\longrightarrow} 0.
\end{split}
\]
Hence, we conclude that
\begin{equation}
\label{eq:limit X tarski}
\lim_{n}\mathbb{E}\bigg[\sup_{0 \le t \le T}|\tilde{X}_t-\underline{X}^n_t|^2\bigg]=0.
\end{equation}
Let us now investigate the convergence of $\{(\underline Y^n,\underline Z^n,\underline A^n)\}_n$. By Proposition \ref{P:Stability}, we have the following stability result for our system
\[
\begin{split}
    &\mathbb{E}\bigg[\sup_{0 \le t \le T}|\underline{Y}^n_t-\tilde{Y}_t|^2+\int_0^T|\underline{Z}^n_t-\tilde{Z}_t|^2dt+\sup_{0 \le t \le T}|\underline{A}^n_t-\tilde{A}_t|^2\bigg]\\
    & \le c\mathbb{E}\bigg[\int_0^T|f(t,\underline{X}^n_t,\mathbb{E}[\bar{f}(t,\underline{X}^n_t)\underline{L}^{n-1}_t])-f(t,\tilde{X}_t,\mathbb{E}[\bar{f}(t,\underline{X}^n_t)\underline{L}^{n-1}_t])|^2dt\bigg]\\
    &\,\,\,\,\,\,\,\,\,+c\int_0^T|\mathbb{E}[\bar{f}(t,\underline{X}^n_t)\underline{L}^{n-1}_t]-\mathbb{E}[\bar{f}(t,\tilde{X}_t)\tilde{L}_t]|^2dt\\
    &\,\,\,\,\,\,\,\,\,+c\,\mathbb{E}\bigg[\bigg|h\bigg(T,\underline{X}^n_T,\mathbb{E}\bigg[\int_{0^-}^T\phi(T-s)d\underline{L}^{n-1}_s\bigg]\bigg)-h\bigg(T, \tilde{X}_T,\mathbb{E}\bigg[\int_{0^-}^T\phi(T-s)d\underline{L}^{n-1}_s\bigg]\bigg)\bigg|^2\bigg]\\
    &\,\,\,\,\,\,\,\,\,+c\sqrt{\mathbb{E}\bigg[\sup_{0 \le t \le T}\bigg|h\bigg(t,\underline{X}^n_t,\mathbb{E}\bigg[\int_{0^-}^T\phi(t-s)d\underline{L}^{n-1}_s\bigg]\bigg)-h\bigg(t, \tilde{X}_t,\mathbb{E}\bigg[\int_{0^-}^T\phi(t-s)d\underline{L}^{n-1}_s\bigg]\bigg)\bigg|^2\bigg]}\\
    &\,\,\,\,\,\,\,\,\,+c\bigg|\mathbb{E}\bigg[\int_{0^-}^T\phi(T-s)d\underline{L}^{n-1}_s\bigg]-\mathbb{E}\bigg[\int_{0^-}^T\phi(T-s)d\tilde{L}_s\bigg]\bigg|^2 \\
    &\,\,\,\,\,\,\,\,\,+c\sup_{0 \le t \le T}\bigg|\mathbb{E}\bigg[\int_{0^-}^T\phi(t-s)d\underline{L}^{n-1}_s\bigg]-\mathbb{E}\bigg[\int_{0^-}^T\phi(t-s)d\tilde{L}_s\bigg]\bigg|.
\end{split}
\]
We now show that the terms above converge to zero as $n\rightarrow\infty$. We begin with the first term. By the Lipschitz continuity of $f$, we get
\[
\begin{split}
    \mathbb{E}\bigg[\int_0^T|f(t,\underline{X}^n_t,\mathbb{E}[\bar{f}(t,\underline{X}^n_t)\underline{L}^{n-1}_t])-f(t,\tilde{X}_t,\mathbb{E}[\bar{f}(t,\underline{X}^n_t)\underline{L}^{n-1}_t])|^2dt\bigg]\le K^2\mathbb{E}\bigg[\sup_{0 \le t \le T}|\underline{X}^n_t-\tilde{X}_t|^2\bigg].
\end{split}
\]
Then, by limit (\ref{eq:limit X tarski}) we see that such a term converges to zero as $n\rightarrow\infty$.\\
Regarding the second term, we have
\[
\begin{split}
    &\int_0^T|\mathbb{E}[\bar{f}(t,\underline{X}^n_t)\underline{L}^{n-1}_t]-\mathbb{E}[\bar{f}(t,\tilde{X}_t)\tilde{L}_t]|^2dt \\
    &\le2 \int_0^T\bigg(|\mathbb{E}[\bar{f}(t,\underline{X}^n_t)\underline{L}^{n-1}_t]-\mathbb{E}[\bar{f}(t,\tilde{X}_t)\underline{L}^{n-1}_t]|^2+|\mathbb{E}[\bar{f}(t,\tilde{X}_t)\underline{L}^{n-1}_t]-\mathbb{E}[\bar{f}(t,\tilde{X}_t)\tilde{L}_t]|^2\bigg)dt\\
    &\leq2K^2\int_0^T\bigg(|\mathbb{E}[|\underline{X}^n_t-\tilde{X}_t|\underline{L}^{n-1}_t]|^2+(\mathbb{E}[|\bar{f}(t,\tilde{X}_t)||\underline{L}^{n-1}_t-\tilde{L}_t|])^2\bigg)dt,
\end{split}
\]
where the last inequality follows from the Lipschitz property of $\bar f$. Then, by \eqref{eq:limit L} and \eqref{eq:limit X tarski} we see that the second term converges to zero as $n\rightarrow\infty$.\\
Concerning the third and fourth terms, they can be treated analogously. For this reason, we only report the proof of the convergence for the fourth term. Using the continuity of $h$, by Lebesgue's dominated convergence theorem, we get
\[
\begin{split}
\sqrt{\mathbb{E}\bigg[\sup_{0 \le t \le T}\bigg|h\bigg(t,\underline{X}^n_t,\mathbb{E}\bigg[\int_{0^-}^T\phi(t-s)d\underline{L}^{n-1}_s\bigg]\bigg)-h\bigg(t, \tilde{X}_t,\mathbb{E}\bigg[\int_{0^-}^T\phi(t-s)d\underline{L}^{n-1}_s\bigg]\bigg)\bigg|^2\bigg]} \to 0
\end{split}
\]
then the claim follows from (\ref{eq:limit X tarski}). Finally, the last two terms above vanish thanks to Lemma \ref{L:UnifConv_phi}.\\

\noindent\textsc{Step 3} It remains to prove that $\tilde{L}\in\Gamma_2\big((\tilde{X},\tilde{Y}, \tilde{Z}, \tilde{A})\big)=\Gamma(\tilde{L})$, namely that the following conditions hold:
\[
\int_{0^-}^T \bigg(\tilde Y_t-h\bigg(t,\tilde{X}_t,\mathbb{E}\bigg[\int_{0^-}^T\phi(t-s)d\tilde{L}_s\bigg]\bigg)\bigg)d\tilde L_t=0 \qquad \text{ and } \qquad \int_{0^-}^T \tilde A_td\tilde L_t=0.
\]
By construction of the sequence $\underline{L}^n=\Gamma_2\big((\underline{X}^n,\underline{Y}^n, \underline{Z}^n, \underline{A}^n,\underline L^{n-1})\big)$, for $n\geq1$, namely 
\[
\int_{0^-}^T\bigg(\underline{Y}^n_t-h\bigg(t,\underline{X}^n_t,\mathbb{E}\bigg[\int_{0^-}^T\phi(t-s)d\underline{L}^{n-1}_s\bigg]\bigg)\bigg)d\underline{L}^n_t=0 \qquad \text{ and } \qquad \int_{0^-}^T\underline{A}^n_td\overline{L}^n_t=0.
\]
We have already proved in \textsc{Step 1} that $\underline{L}^n$ converges strongly to $\tilde{L}$. Since we are in a Hilbert space, this implies that $\underline{L}^n$ converges weakly to $\tilde{L}$. Moreover, by \textsc{Step 2} we have that $\underline{X}^n \to \tilde{X}$ in $\mathbb{S}^2$ and $(\underline{Y}^n, \underline{Z}^n, \underline{A}^n) \to (\tilde{Y}, \tilde{Z}, \tilde{A})$ in $\mathbb{S}^2\times\mathbb{H}^2\times\mathbb{S}^2$. Now, we have
\[
\begin{split}
&\bigg|\int_{0^-}^T\!\!\!\bigg(\underline{Y}^n_t-h\bigg(t,\underline{X}^n_t,\mathbb{E}\bigg[\int_{0^-}^T\!\phi(t-s)d\underline{L}^{n-1}_s\bigg]\bigg)\!\bigg)d\underline{L}^{n}_t-\int_{0^-}^T\!\!\!\bigg(\tilde{Y}_t-h\bigg(\!t,\tilde{X}_t,\mathbb{E}\bigg[\int_{0^-}^T\!\phi(t-s)d\tilde{L}_s\bigg]\bigg)\!\bigg)d\tilde{L}_t\bigg|\\
&=\bigg|\int_{0^-}^T\!\!\!(\underline{Y}^n_t-\tilde{Y}_t)d\underline{L}^n_t+\int_{0^-}^T\!\!\!\bigg(h\bigg(\!t,\tilde{X}_t,\mathbb{E}\bigg[\int_{0^-}^T\!\phi(t-s)d\tilde{L}_s\bigg]\bigg)-h\bigg(\!t,\underline{X}^n_t,\mathbb{E}\bigg[\int_{0^-}^T\!\phi(t-s)d\underline{L}^{n-1}_s\bigg]\bigg)\!\bigg)d\underline{L}^n_t\\
&\quad+\int_{0^-}^T\bigg(\tilde{Y}_t-h\bigg(t,\tilde{X}_t,\mathbb{E}\bigg[\int_{0^-}^T\phi(t-s)d\tilde{L}_s\bigg]\bigg)\bigg)d(\underline{L}^n_t-\tilde{L}_t)\bigg|\\
&\le\int_{0^-}^T|\underline{Y}^n_t-\tilde{Y}_t|d|\underline{L}^n|_t+K\int_{0^-}^T|\underline{X}^n_t-\tilde{X}_t|d|\underline{L}^n|_t \\
&\quad+K\int_{0^-}^T\bigg|\mathbb{E}\bigg[\int_{0^-}^T\phi(t-s)d\tilde{L}_s\bigg]-\mathbb{E}\bigg[\int_{0^-}^T\phi(t-s)d\underline{L}^{n-1}_s\bigg]\bigg|d|\underline{L}^n|_t\\
&\quad+\bigg|\int_{0^-}^T\bigg(\tilde{Y}_t-h\bigg(t,\tilde{X}_t,\mathbb{E}\bigg[\int_{0^-}^T\phi(t-s)d\tilde{L}_s\bigg]\bigg)\bigg)d(\tilde{L}_t-\underline{L}^n_t)\bigg|\\
&\le \sup_{0 \le t \le T}|\underline{Y}^n_t-\tilde{Y}_t|+K\sup_{0 \le t \le T}|\underline{X}^n_t-\tilde{X}_t| \\
&\quad+K\sup_{0 \le t \le T}\bigg|\mathbb{E}\bigg[\int_{0^-}^T\phi(t-s)d\tilde{L}_s\bigg]-\mathbb{E}\bigg[\int_{0^-}^T\phi(t-s)d\underline{L}^{n-1}_s\bigg]\bigg|\\
&\quad+\bigg|\int_{0^-}^T\bigg(\tilde{Y}_t-h\bigg(t,\tilde{X}_t,\mathbb{E}\bigg[\int_{0^-}^T\phi(t-s)d\tilde{L}_s\bigg]\bigg)\bigg)d(\underline{L}^n_t-\tilde{L}_t)\bigg|.
\end{split}
\]
Recalling that $\underline{Y}^n \to \tilde{Y}$ in $\mathbb{S}^2$ (resp. $\underline{X}^n \to \tilde{X}$ in $\mathbb{S}^2$), there exists a subsequence such that $\sup_t|\underline{Y}^n_t-\tilde{Y}_t|$ (resp. $\sup_t|\underline{X}^n_t-\tilde{X}_t|$) converges almost surely to zero. The third term above vanishes thanks to Lemma \ref{L:UnifConv_phi}. For the last term, it is sufficient to replicate the same technique used in the proof of Proposition \ref{P:Gamma}. In conclusion, we find that
\[
\int_{0^-}^T\bigg(\underline{Y}^n_t-h\bigg(t,\underline{X}^n_t,\mathbb{E}\bigg[\int_{0^-}^T\phi(t-s)d\underline{L}^{n-1}_s\bigg]\bigg)\bigg)d\underline{L}^n_t \to \int_{0^-}^T\bigg(\tilde{Y}_t-h\bigg(t,\tilde{X}_t,\mathbb{E}\bigg[\int_{0^-}^T\phi(t-s)d\tilde{L}_s\bigg]\bigg)\bigg)d\tilde{L}_t.
\]
Then, by uniqueness of the limit,
\[
\int_{0^-}^T\bigg(\tilde{Y}_t-h\bigg(t,\tilde{X}_t,\mathbb{E}\bigg[\int_{0^-}^T\phi(t-s)d\tilde{L}_s\bigg]\bigg)\bigg)d\tilde{L}_t=0.
\]
Similarly for $\int_{0^-}^T\underline{A}^n_td\underline{L}^n_t \to \int_{0^-}^T\tilde{A}_td\tilde{L}_t$, concluding that $\tilde{L}$ satisfies the two claimed constraints, i.e. $\tilde{L}\in\Gamma(\tilde{L})$.
\medskip
\\
\textsc{Step 4}. From \textsc{Step 2} and \textsc{Step 3},  we know that $(\tilde{X},\tilde{Y}, \tilde{Z}, \tilde{A}, \tilde{L})$ is a solution of system \eqref{MFG SDE}-\eqref{MFRBSDE}. We now prove that $(\tilde{X},\tilde{Y}, \tilde{Z}, \tilde{A}, \tilde{L})$ is the minimal solution of system \eqref{MFG SDE}-\eqref{MFRBSDE} (Definition \ref{def:max_min_solution}). Let $(X,Y,Z,A,L)$ be another solution, then we have $L\in \Gamma(L)$. By definition of $\underline{L}^0$, we have $ \underline{L}^0\leq_{\mathcal{V}} L$. By monotonicity of $\overline{\mathcal{R}}$, we obtain $\underline{L}^1 = \underline{\mathcal{R}}(\underline{L}^0)\leq_{\mathcal{V}}\underline{\mathcal{R}}(L)\leq_{\mathcal V} L$. Therefore, iterating the map $\underline{\mathcal{R}}$, we have $\underline{L}^n \leq_{\mathcal{V}}L$, for any $n \in \mathbb{N}$. Taking the limit as in (\ref{eq:limit L}), we conclude that $\tilde L \leq_{\mathcal{V}} L$. Moreover, since $\tilde L \leq_{\mathcal{V}} L$, by monotonicity of $\Gamma_1$ (Proposition \ref{P:increasing maps}), we have $\tilde X \leq_{\mathbb{S}^2} X$ and $\tilde Y \leq_{\mathbb{S}^2} Y$.\\
Thus, $(\tilde{X},\tilde{Y}, \tilde{Z}, \tilde{A}, \tilde{L})$ is the minimal solution of system \eqref{MFG SDE}-\eqref{MFRBSDE}. By uniqueness of the minimal solution (Definition \ref{def:max_min_solution}), we conclude that $\tilde{L}=\tilde{L}^{min}$.
\end{proof}

\section{Connection with mean field games of optimal stopping in mixed strategies}
\label{S:Connection}
In the present Section, we investigate the relation between the MKV-RFBSDE system \eqref{MFG SDE}-\eqref{MFRBSDE} and a mean field game of optimal stopping in mixed strategies, which is articulated in the following way.

For $L, \hat L \in \mathcal V$, let $\bar{\Lambda}\colon\mathcal V^2\rightarrow \mathbb R$ be given by (recall that $g(x,w)=h(T,x,w)$):
\begin{align}\label{Reward}
\bar{\Lambda}(L,\hat L)&:=\mathbb{E}\biggl[\int_0^T f(s,X_s, \mathbb{E}[\bar{f}(s,X_s)L_s])\hat L_s ds \\
&\quad-\int_{0^-}^T h\bigg(s, X_s,\mathbb{E}\bigg[\int_{0^-}^T\phi(s-r)dL_r\bigg]\bigg)d\hat L_s\biggr], \notag
\end{align}
with $X$ solving equation \eqref{MFG SDE} (with input process $L$).

\begin{definition}[\textit{OS-MFG equilibria in randomized strategies}]\label{defn:OSequilibrium}
The process $\hat L\in\mathcal V$ is a mean field optimal stopping equilibrium if it holds that
\[
\bar{\Lambda}(\hat L,\hat L') \le \bar{\Lambda}(\hat L,\hat L), \qquad \text{for all }\hat L'\in\mathcal V,
\]
where $\bar{\Lambda}$ is defined in \eqref{Reward}.
\end{definition}
\begin{theorem}[\textit{Equivalence between solutions to system \eqref{MFG SDE}-\eqref{MFRBSDE} and OS-MFG equilibria in randomized strategies}]\label{T:BSDE_MFG}

Suppose Assumptions \ref{assumption:1A}, \ref{assumption:1B}, \ref{assumption: 1C} (resp. \ref{assumption:1A}, \ref{assum:Tarski_1}.i)-ii) and \ref{assum:Tarski_2}) hold. Let $(X,Y,Z,A,L)\in\mathbb{S}^2\times\mathbb S^2\times\mathbb H^2\times\mathbb A^2\times\mathcal V$ be a solution of \eqref{MFG SDE}-\eqref{MFRBSDE}. Then $L$ is a mean field optimal stopping equilibrium (Definition \ref{defn:OSequilibrium}).\\
Conversely, let $\hat{L}\in\mathcal V$ be a mean field optimal stopping equilibrium. Then, $(\hat X,\hat Y,\hat Z,\hat A,\hat{L})$ with $\hat X\in\mathbb S^2$ solution to equation \eqref{MFG SDE} (with input process $\hat L$) and $(\hat Y,\hat Z,\hat A)\in\mathbb S^2\times\mathbb H^2\times\mathbb A^2$ solution to equation \eqref{RBSDE} (with input process $\hat L$) is a solution of system \eqref{MFG SDE}-\eqref{MFRBSDE}, namely the following two constraints are satisfied:
\[
\int_{0^-}^T \bigg(\hat Y_t-h\bigg(t,\hat X_t,\mathbb{E}\bigg[\int_{0^-}^T\phi(t-s)d\hat L_s\bigg]\bigg)\bigg)d\hat L_t=0 \qquad \text{ and } \qquad \int_{0^-}^T \hat A_td\hat L_t=0.
\]
\end{theorem}
\begin{proof}

\vspace{2mm}

\noindent Let $(X,Y,Z,A,L)\in\mathbb{S}^2\times\mathbb S^2\times\mathbb H^2\times\mathbb A^2\times\mathcal V$ be a solution of \eqref{MFG SDE}-\eqref{MFRBSDE}. To alleviate notation, set $\xi_t:=h(t, X_t,\mathbb{E}[\int_{0^-}^T\phi(t-s)dL_s])$, for $t\in[0,T]$. Since $L$ satisfies the Skorokhod conditions $\int_{0^-}^T(Y_t-\xi_t)dL_t=0$ and $\int_{0^-}^T A_t dL_t=0$, by Theorem \ref{equiv},  we deduce that
\[
\mathbb E[Y_0] = \bar{\Lambda}(L,L) = \sup_{\tilde L\in\mathcal V}\bar{\Lambda}(L,\tilde L).
\]
\vspace{2mm}
Conversely, let $\hat{L} \in \mathcal{V}$ be a mean field optimal stopping equilibrium, and let $(\hat X,\hat Y,\hat Z,\hat A)$ solve \eqref{MFG SDE} and \eqref{RBSDE} (with input process $\hat L$). Set $\hat\xi_t:=h(t, \hat X_t,\mathbb{E}[\int_{0^-}^T\phi(t-s)d\hat L_s])$, for $t\in[0,T]$. By Theorem \ref{equalvalues} at $t=0$, we have
\[
\hat Y_0\geq\mathbb{E}\bigg[\int_0^T f(t,\hat X_t, \mathbb{E}[\bar{f}(t,\hat X_t)\hat L_t])\tilde L_tds-\int_{0^-}^T\hat \xi_t d\tilde L_t\bigg|\mathcal{F}_0\bigg], \qquad \mathbb P\text{-a.s., for all }\tilde L\in\mathcal V.
\]
Taking the expectation, we find
\[
\mathbb E[\hat Y_0] \geq \bar{\Lambda}(\hat L,\tilde L), \qquad \forall\,\tilde L\in\mathcal V.
\]
From the arbitrariness of $\tilde L$, this shows that $\mathbb E[\hat Y_0]\geq\sup_{\tilde L\in\mathcal V}\bar{\Lambda}(\hat L,\tilde L)$. On the other hand, recalling that \(\hat{L}\) is a mean field optimal stopping equilibrium, we know that $\sup_{\tilde L\in\mathcal V}\bar\Lambda(\hat L,\tilde L)=\bar\Lambda(\hat L,\hat L)$. Consequently, we deduce that $\mathbb E[\hat Y_0]\geq\bar{\Lambda}(\hat L,\hat L)$. To establish the reverse inequality, let $L^\star_t=\mathbf 1_{t < \tau^\star}$, where $\tau^\star=\inf\{t:\hat Y_t=\hat \xi_t\}$. By Theorem \ref{equiv}, $\mathbb E[\hat Y_0]=\bar \Lambda(\hat L, L^\star)$ and since $\hat L$ is an equilibrium, $\mathbb E[\hat Y_0]=\bar \Lambda(\hat L, L^\star) \leq \bar \Lambda(\hat L, \hat L)$. Combining these two inequalities, we get $\mathbb E[\hat Y_0]=\bar \Lambda(\hat L, \hat L)$. Then, by the same arguments as in Theorem \ref{equiv} at time $t=0$,  it follows that $\int_{0^-}^T (\hat{Y}_t-\hat{\xi}_t)\,d\hat{L}_t = 0$ and $\int_{0^-}^T \hat{A}_t\,d\hat{L}_t = 0$. Therefore, $(\hat X,\hat Y,\hat Z,\hat A, \hat{L})$  is a solution of \eqref{MFG SDE}-\eqref{MFRBSDE} in $\mathbb{S}^2\times\mathbb S^2\times\mathbb H^2\times\mathbb A^2\times\mathcal V$.
\end{proof}
\paragraph{Existence and properties of OS-MFG equilibria in randomized strategies.}
Using the results established above for solutions to system~\eqref{MFG SDE}-\eqref{MFRBSDE}, together with the equivalence result in Theorem \ref{T:BSDE_MFG}, we derive below the following results on OS-MFG equilibria
in randomized strategies.
\begin{proposition}[\textit{Existence of OS-MFG equilibria in randomized strategies}]
    Under Assumptions \ref{assumption:1A}, \ref{assumption:1B} and \ref{assumption: 1C}, the set of mean field optimal stopping equilibria is non-empty.
\end{proposition}
\begin{proof}
    By Theorem \ref{thm:main}, we know that system \eqref{MFG SDE}-\eqref{MFRBSDE} has a solution $(X,Y,Z,A,L) \in \mathbb{S}^2\times\mathbb{S}^2\times\mathbb{H}^2\times\mathbb{A}^2\times\mathcal{V}$. Moreover, by Theorem \ref{T:BSDE_MFG} we have a one-to-one relation between the solutions of system \eqref{MFG SDE}-\eqref{MFRBSDE} and mean field optimal stopping equilibria. Thus, the set of mean field optimal stopping equilibria is non-empty.
\end{proof}
\begin{remark}
\label{remark: relation MKRBSDE - MFG equilibria}
    By Theorem \ref{T:BSDE_MFG} we show that there is a one-to-one relation between the solution to system \eqref{MFG SDE}-\eqref{MFRBSDE} and the set of mean field optimal stopping equilibrium (Definition \ref{defn:OSequilibrium}). We can define a projection map $\Pi:\mathbb{S}^2\times\mathbb{S}^2\times\mathbb{H}^2\times\mathbb{A}^2\times\mathcal{V} \to \mathcal{V}$ which is a bijection between the sets
    \[
    R=\bigg\{\text{Solution to system } \eqref{MFG SDE}-\eqref{MFRBSDE}\bigg\} \xrightarrow{\Pi} \bigg\{\text{mean field optimal stopping equilibria}\bigg\}=\mathcal{M}.
    \] 
\end{remark}
\begin{remark}
\label{remark:monotonicity p}
Let $(X,Y,Z,A,L)$ and $(X',Y',Z',A',L')$ be two solutions to the system \eqref{MFG SDE}-\eqref{MFRBSDE} such that $(X,Y,Z,A,L) \leq_R (X',Y',Z',A',L')$.
Then 
\[
L=\Pi(X,Y,Z,A,L) \leq_{\mathcal{V}} \Pi(X',Y',Z',A',L')=L'.
\]
Then the map $\Pi$ is order-preserving isomorphism.
\end{remark}
\begin{proposition}[\textit{Properties of the set of OS-MFG equilibria in randomized strategies}]
    Under Assumptions \ref{assumption:1A}, \ref{assum:Tarski_1}.i)-ii), \ref{assum:Tarski_2} and \ref{assum:Tarski_3} the set of mean field optimal stopping equilibria $\mathcal{M}$ is a non-empty complete lattice (compatible with the order relation $\leq_{\mathcal V}$). In particular, under Assumptions \ref{assumption:1A}, \ref{assum:Tarski_1}.i)-ii) and \ref{assum:Tarski_2},  there exist a minimal equilibrium $\tilde{L}^{min}$ and a maximal equilibrium $\tilde{L}^{max}$.
\end{proposition}
\begin{proof}
By Theorem \ref{T:Tarski}.ii), the set of solution to the system \eqref{MFG SDE}-\eqref{MFRBSDE} $R$ is a non-empty complete lattice. According to Remark \ref{remark:monotonicity p}, the projection map $\Pi$ is an order-preserving isomorphism. Since an order-preserving isomorphism preserves the lattice structure (see \cite{davey2002introduction}), it follows that $\mathcal{M}$ is also a non-empty complete lattice. Furthermore, the extremal elements of $\mathcal{M}$ are uniquely determined by the projection of the extremal solutions in $R$. Specifically, if $\Sigma^{min}:=(\tilde{X}^{min}, \tilde{Y}^{min}, \tilde{Z}^{min}, \tilde{A}^{min}, \tilde{L}^{min})$ and $\Sigma^{max}:=(\tilde{X}^{max}, \tilde{Y}^{max}, \tilde{Z}^{max}, \tilde{A}^{max}, \tilde{L}^{max})$ are the minimal and maximal solution in $R$, then the minimal and maximal mean field optimal stopping equilibria are given by
\[
\tilde{L}^{min}=\Pi(\Sigma^{min}) \quad \text{and} \quad \tilde{L}^{max}=\Pi(\Sigma^{max}).
\]
\end{proof}

\section{Approximate Nash equilibria for the $N$-player game}
\label{S:Approximate}
In this Section,  we show that, starting from a solution of the mean field system \eqref{MFG SDE}-\eqref{MFRBSDE}, we construct an approximate equilibrium for the associated \(N\)-player optimal-stopping game.
We suppose here that, on the same probability space $(\Omega, \mathcal{F}, \mathbb{P})$, there exists a sequence $\{(W^i_t)_{t \in [0,T]}\}_{i \ge 1}$ of independent $m$-dimensional Brownian motions. We also assume that there exists a sequence $\{X_0^i\}_{i \geq 1}$ of independent and identically distributed $\mathbb R^d$-valued random variables, independent of $\{(W^i_t)_{t \in [0,T]}\}_{i \ge 1}$ and having the same distribution as $X_0$, so in particular $\mathbb E[|X_0^i|^q]<\infty$, for every $i$. For every $N\in\mathbb N$, we denote by $\mathbb F^N=\{\mathcal F_t^N\}_{t\in[0,T]}$ the filtration generated by $\{X_0^i\}_{i=1,\ldots,N}$ and by the Brownian motions $\{(W^i_t)_{t \in [0,T]}\}_{i=1,\ldots,N}$, augmented with the $\mathbb P$-null sets of $\mathcal F$. We also denote by $\mathcal V_{\mathbb F^N}$ the space $\mathcal V$ defined with respect to the filtration $\mathbb F^N$.

We now introduce the $N$-player game. Let $\boldsymbol{L}=(L^1,\ldots,L^N)\in\mathcal V_{\mathbb F^N}^N$ be the strategy profile for the $N$ players. Then, for each $1 \le i \le N$, we define the reward functional of the $i$-th player as follows (recall that $g(x,w)=h(T,x,w)$):
\begin{align}\label{payoff N-player}
\Lambda^{N,i}(L^1, \dots, L^N)&=\mathbb{E}\biggl[\int_0^Tf\biggl(t,X^i_t, \frac{1}{N}\sum_{j=1}^N\bar{f}(t,X^j_t)L^j_t\biggr)L^i_tdt\\
&\quad-\int_{0^-}^Th\biggl(t,X^i_t,\frac{1}{N}\sum_{j=1}^N\int_{0^{-}}^T\phi(t-s)dL^j_s\biggr)dL^i_t\biggr], \notag
\end{align}
where $(X^1_t,\dots, X^N_t)_{0 \le t \le T}$ satisfies the system of $Nd$ stochastic differential equations
\begin{equation}\label{N-player SDE_0}
dX^i_t=b\biggl(t,X^i_t, \frac{1}{N}\sum_{j=1}^N\bar{b}(t,X^j_t)L_t^j\biggr)dt+\sigma\biggl(t,X^i_t, \frac{1}{N}\sum_{j=1}^N\bar{\sigma}(t,X^j_t)L_t^j\biggr)dW^i_t,
\end{equation}
with $t \in [0,T]$ and initial condition $X^i_0$. Under Assumption \ref{assumption:1A}, the $Nd$-dimensional system \eqref{N-player SDE_0} is well-posed since the maps $b$ and $\sigma$ are assumed to be globally Lipschitz. Moreover the processes $L^i$ are bounded and the maps $\bar{b},\bar{\sigma}$ satisfy linear growth in the variable $x$ uniformly in $t \in [0,T]$. Thus, the system has a unique strong solution. Moreover, for $2 \leq p \leq q$, for every $N \geq 1$ and for every $i \in \{1, \dots N\}$ estimate \eqref{EstimateSDE} holds.
Given the unique solution of the forward component consider
\begin{equation}
\label{eq: backward component N player}
    \begin{cases}
        -dY^i_t=f\bigg(t,X^i_t,\frac{1}{N}\sum_{j=1}^N\bar{f}(t,X^j_t)L^j_t\bigg)dt+dA^i_t-\sum_{j=1}^N(Z^{i,j}_t,dW^j_t),\\
        Y^i_t \geq \xi^i_t,\\
        \int_0^T(Y^i_t-\xi^i_t)dA^i_t=0, \qquad Y^i_T=\xi^i_T.
    \end{cases}
\end{equation}
where $\xi^i_t=h\bigg(t,X^i_t,\frac{1}{N}\sum_{j=1}^N\int_{0^-}^T\phi(t-s)dL^j_s\bigg)$. Under Assumption \ref{assumption:1B} (resp. \ref{assum:Tarski_1}.i)-ii)), the backward component \eqref{eq: backward component N player} is well-posed since the obstacle $\xi^i$ is well defined and satisfies the standard square integrability conditions. Moreover, since the processes $L^i$ are bounded, $\bar{f}$ has linear growth and $f$ is globally Lipschitz (resp. has polynomial growth), the driver satisfies the usual integrability conditions. Thus, the solution of \eqref{eq: backward component N player} exists and it is unique.\\
We also introduce the reward functional of the limiting problem:
\begin{equation}\label{Lambda}
\Lambda(L)=\mathbb{E}\bigg[\int_0^Tf(t,X_t,\mathbb{E}[\bar{f}(t,X_t)L_t])L_tdt-\int_{0^-}^Th\bigg(t,X_t,\mathbb{E}\bigg[\int_{0^-}^T\phi(t-s)dL_s\bigg]\bigg)dL_t\bigg],
\end{equation}
where $X$ solves equation \eqref{MFG SDE} with input process $L\in\mathcal V$.\\
Firstly we prove the following technical result, which will be useful in the proof of Theorem \ref{T:Approximate} and shows that an independent initial enlargement does not change the optimal stopping value.
\begin{lemma}[\textit{Enlargement of the filtration}]\label{L:IndepEnlarg}
Fix $L \in \mathcal V$. Let $\mathcal H \subset \mathcal F$ be a sigma-field independent of $\mathcal F_T$, and define
\[
\mathcal G_t := \mathcal F_t \vee \mathcal H, \qquad 0 \le t \le T.
\]
Let $\mathcal V_{\mathbb G}$ be the set of equivalence classes of $\mathbb G$-adapted processes $\hat L$ admitting a representative, still denoted $\hat L$, such that $\hat L$ is $[0,1]$-valued, non-increasing, c\`adl\`ag, $\hat L_{0^-}=1$, and $\hat L_T=0$. Then
\[
\sup_{\hat L \in \mathcal V_{\mathbb G}} \bar{\Lambda}(L,\hat L)
=
\sup_{\hat L \in \mathcal V} \bar{\Lambda}(L,\hat L).
\]
\end{lemma}
\begin{proof}
Fix $L \in \mathcal V$. Since $\mathcal V_{\mathbb G}\supset\mathcal V$, we only have to prove the inequality $\sup_{\hat L \in \mathcal V_{\mathbb G}} \bar{\Lambda}(L,\hat L)
\leq
\sup_{\hat L \in \mathcal V} \bar{\Lambda}(L,\hat L)$.\\
Let $X$ be the solution of equation \eqref{MFG SDE} with input process $L$. Let also $(Y,Z,A)$ be the solution of the reflected backward stochastic differential equation \eqref{RBSDE} with input process $L$. To alleviate notation, set $\xi_t:=h(t, X_t,\mathbb{E}[\int_{0^-}^T\phi(t-s)dL_s])$, for $t\in[0,T]$. Let us begin by proving that $\sup_{\hat L \in \mathcal V_{\mathbb G}} \bar{\Lambda}(L,\hat L)\leq\mathbb E[Y_0]$. To this end, define
\[
R_t := \int_0^t f\big(s, X_s, \mathbb E[\bar f(s,X_s)L_s]\big) ds + h\left(t, X_t, \mathbb E\left[\int_{0^-}^T \phi(t-u)\, dL_u\right]\right), \qquad 0 \le t \le T.
\]
Define
\[
S_t := Y_t + \int_0^t f\big(s, X_s, \mathbb E[\bar f(s,X_s)L_s]\big) ds.
\]
Notice that $S_0=Y_0$, so that we have to prove that $\sup_{\hat L \in \mathcal V_{\mathbb G}} \bar{\Lambda}(L,\hat L)\leq\mathbb E[S_0]$. By \eqref{Y_t=esssup_tau}, we have
\[
S_t
=
\operatorname*{ess\,sup}_{\tau \in \mathcal T([t,T])}
\mathbb E\left[
R_\tau
\,\middle|\, \mathcal F_t
\right].
\]
In particular, it holds that $S_t \ge R_t$, for all $t\in[0,T]$. From the reflected backward stochastic differential equation \eqref{RBSDE} satisfied by $(Y,Z,A)$,
\[
S_t = S_0 + \int_0^t Z_s\, dW_s - A_t,
\]
hence $S$ is a continuous $\mathbb F$-supermartingale. Now, notice that, since $S_t$ is $\mathcal F_t$-measurable, for every $0 \le s \le t$, it holds that
\[
\mathbb E[S_t \mid \mathcal G_s]
=
\mathbb E[S_t \mid \mathcal F_s]
\le S_s.
\]
So $S$ is also a $\mathbb G$-supermartingale. Given $\hat L \in \mathcal V_{\mathbb G}$, for $r \in [0,1]$, define
\[
\tau_r := \inf\{ t \in [0,T] : \hat L_t \le r\}.
\]
Because $\hat L$ is $\mathbb G$-adapted, c\`adl\`ag, and non-increasing, $\tau_r$ is a $\mathbb G$-stopping time. Moreover, for every integrable Borel function $\psi\colon [0,T] \to \mathbb R$, it holds that
\[
\int_0^1 \psi(\tau_r)\, dr
=
-\int_{0^-}^T \psi(t)\, d\hat L_t.
\]
Applying this identity to $\psi(t) = h(t, X_t, \mathbb E[\int_{0^-}^T \phi(t-u)\, dL_u])$, and using also that $\hat L_t=\int_0^1 \mathbf 1_{t < \tau_r}dr$, we get
\begin{align*}
&\int_0^T f\big(t, X_t, \mathbb E[\bar f(t,X_t)L_t]\big) \hat L_t\, dt
-
\int_{0^-}^T h\left(t, X_t, \mathbb E\left[\int_{0^-}^T \phi(t-u)\, dL_u\right]\right) d\hat L_t \\
&=
\int_0^1
\left(
\int_0^{\tau_r} f\big(t, X_t, \mathbb E[\bar f(t,X_t)L_t]\big) dt + h\left(\tau_r, X_{\tau_r}, \mathbb E\left[\int_{0^-}^T \phi(\tau_r-u)\, dL_u\right]\right)
\right) dr \\
&=
\int_0^1 R_{\tau_r}\, dr.
\end{align*}
Taking the expectation, we find
\[
\bar{\Lambda}(L,\hat L)
=
\int_0^1 \mathbb E[R_{\tau_r}]\, dr.
\]
Since $R_t \le S_t$, for all $t$, and $S$ is a $\mathbb G$-supermartingale, for every $r \in [0,1]$,
\[
\mathbb E[R_{\tau_r}]
\le
\mathbb E[S_{\tau_r}]
\le
\mathbb E[S_0].
\]
Therefore, $\bar{\Lambda}(L,\hat L) \le \mathbb E[S_0]$. Since $\hat L \in \mathcal V_{\mathbb G}$ is arbitrary, we conclude that $\sup_{\hat L \in \mathcal V_{\mathbb G}}\bar{\Lambda}(L,\hat L) \le \mathbb E[S_0]$.\\
It remains to prove that $\sup_{\hat L \in \mathcal V}\bar{\Lambda}(L,\hat L)\ge \mathbb E[S_0]=\mathbb E[Y_0]$ (recall that $S_0=Y_0$). Define $\tau = \inf\{t \in [0,T]\colon Y_t = \xi_t\}$ and $L_t^* = \mathbf 1_{t<\tau}$, for $0 \le t \le T$. By the proof of Lemma \ref{L:Non-Empty}, $L^*$ satisfies the two Skorokhod conditions, then by Theorem \ref{equiv}, $L^*$ is optimal. 
Moreover, using \eqref{E[Y_t]} at time $t=0$, we obtain
\[
\mathbb E[Y_0]=\bar{\Lambda}(L,L^*).
\]
Since \(L^*\in\mathcal V\), we obtain the claimed inequality
\[
\mathbb E[S_0]=\mathbb E[Y_0]=\bar{\Lambda}(L,L^*)
\le
\sup_{\hat L \in \mathcal V}\bar{\Lambda}(L,\hat L).
\]
\end{proof}
In this section, for any $N \in \mathbb{N}$, for any set of strategies $\boldsymbol{L}=(L^1, \dots, L^N)\in\mathcal V_{\mathbb F^N}^N$, and for any $k \in \{1, \dots, N\}$, we denote $\boldsymbol{L}^{-k}:=(L^1, \dots, L^{k-1}, L^{k+1}, \dots, L^N) \in \mathcal{V}_{\mathbb F^N}^{N-1}$. Moreover, for any process $L \in \mathcal{V}_{\mathbb F^N}$, we denote $L \otimes\boldsymbol{L}^{-k}:=(L^1, \dots, L^{k-1},L, L^{k+1}, \dots, L^N) \in \mathcal{V}_{\mathbb F^N}^{N}$. We recall that a set of strategies $\hat{\boldsymbol{L}}=(\hat{L}^1, \dots, \hat{L}^N) \in \mathcal{V}_{\mathbb F^N}^N$ is said to be a Nash equilibrium for the $N$-player game if
\[
    \forall i \in \{1,\dots, N\} , \,\,\, \forall L \in \mathcal{V}_{\mathbb F^N}, \qquad \Lambda^{N,i}(\hat{\boldsymbol{L}})\ge\Lambda^{N,i}(L\otimes \hat{\boldsymbol{L}}^{-i}).
\]
We want to construct an approximate Nash equilibrium for the $N$-player game.
\begin{definition}{($\varepsilon$-Nash equilibrium for the $N$-player game).}
Let $\varepsilon > 0$. We say $\hat{\boldsymbol{L}} \in \mathcal{V}_{\mathbb F^N}^N$ is an $\varepsilon$-Nash equilibrium for the $N$-player game \eqref{payoff N-player}-\eqref{N-player SDE_0} if
\[
    \forall i \in \{1,\dots, N\} , \,\,\, \forall L \in \mathcal{V}_{\mathbb F^N}, \qquad \Lambda^{N,i}(\hat{\boldsymbol{L}})\ge \Lambda^{N,i}(L \otimes \hat{\boldsymbol{L}}^{-i})-\varepsilon.
\]
\end{definition}
\noindent Suppose that Assumptions \ref{assumption:1A}, \ref{assumption:1B} and \ref{assumption: 1C} (or \ref{assumption:1A}, \ref{assum:Tarski_1}.i)-ii) and \ref{assum:Tarski_2}) hold. By Proposition \ref{P:SDE} and Theorem \ref{thm:main} (resp. Theorem \ref{T:Tarski} i)) we know that there exists a quintuple $(\hat X,\hat Y,\hat Z,\hat A,\hat L)$ solution to system \eqref{MFG SDE}-\eqref{MFRBSDE}, respectively. Recall that the reference filtration in Theorem \ref{thm:main} (resp. Theorem \ref{T:Tarski} i)) is generated by $X_0$ and by the Brownian motion $\{W_t\}_{t\in[0,T]}$, augmented with the $\mathbb P$-null sets of $\mathcal F$. Then, by Doob's measurability theorem it follows that there exists a measurable map $\hat\psi\colon[0,T]\times\mathbb R^d\times \mathcal C([0,T];\mathbb R^m)\rightarrow[0,1]$ such that the processes $\{\hat L_t\}_{t\in[0,T]}$ and $\{\hat\psi(t,X_0,W_{\cdot\wedge t})\}_{t\in[0,T]}$ are indistinguishable. Then, we denote
\begin{equation}\label{hatL^i}
\hat L^i_t=\hat{\psi}(t,X_0^i,W^i_{\cdot\wedge t}), \qquad 0 \le t \le T, \,\,\, i =1,\ldots,N.
\end{equation}
Then the processes $\hat L^i$ are independent and identically distributed copies of $\hat L$, moreover $\hat L^i \in \mathcal V_{\mathbb F^N}$.

For each integer $N$, we consider the solution $(\hat X^1_t,\dots,\hat X^N_t)_{0 \le t \le T}$ of the system of $N$ stochastic differential equations
\begin{equation}
\label{N-player SDE}
d\hat X^i_t=b\biggl(t,\hat X^i_t, \frac{1}{N}\sum_{j=1}^N\bar{b}(t,\hat X^j_t)\hat{L}_t^j\biggr)dt+\sigma\biggl(t,\hat X^i_t, \frac{1}{N}\sum_{j=1}^N\bar{\sigma}(t,\hat X^j_t)\hat{L}_t^j\biggr)dW^i_t,
\end{equation}
with $t \in [0,T]$ and $\hat X^i_0=X_0^i$. The processes $\hat X^1,\ldots,\hat X^N$ give the dynamics of the states of the $N$ players, when they use the set of strategies $(\hat L^1,\ldots,\hat L^N)$.

We prove now two lemmas which establish two type of estimates between the $i$-th component of the $Nd$-system and the $i$-th independent copy of the McKean--Vlasov SDE.
\begin{lemma}[\textit{Estimates I}]
\label{lemma: order barX-X}
Suppose that Assumptions \ref{assumption:1A}, \ref{assumption:1B} and \ref{assumption: 1C} (or \ref{assumption:1A}, \ref{assum:Tarski_1}.i)-ii) and \ref{assum:Tarski_2}) hold. For every $i \in \{1, \dots, N\}$, let $\hat X^i$ be the solution of \eqref{N-player SDE} and let $\bar{X}^i$ be the solution of the following McKean--Vlasov stochastic differential equation:
\begin{equation}
\label{eq:iid SDE}
\bar X_t^i=X_0^i+\int_0^tb\big(s,\bar X_s^i,\mathbb E\big[\bar b(s,\bar X_s^i)\hat L_s^i\big]\big)ds+\int_0^t\sigma\big(s,\bar X_s^i,\mathbb E\big[\bar\sigma(s,\bar X_s^i)\hat L_s^i\big]\big)dW_s^i,
\end{equation}
for all $0\leq t\leq T$. Then, there exists a constant $C\geq0$, independent of $(\hat L^1,\ldots,\hat L^N)$, $N$, $i$, such that, for every $i=1,\ldots,N$,
\[
\mathbb{E}\bigg[\sup_{0 \le t \le T}|\bar{X}^i_t-\hat X^i_t|^2\bigg]\le \frac{C}{N}\big(1+\mathbb E\big[|X_0|^{2}\big]\big).
\]
\end{lemma}
\begin{remark}
    Note that the pair $(\bar X^i, \hat L^i)$ are independent and identically distributed with the same law as $(\hat X, \hat L)$.
\end{remark}
\begin{proof}
Fix $i \in \{1, \dots, N\}$. Applying It\^o's formula to $|\bar{X}^i_t-\hat X^i_t|^2$, and using Jensen and Burkholder--Davis--Gundy inequalities, we find, for some constant $C\geq0$ (in the sequel we denote by $C$ a non-negative constant, independent of $(\hat L^1,\ldots,\hat L^N)$, $N$, $i$, which may change from line to line)
    \[
    \begin{split}
    \mathbb{E}\bigg[\sup_{0\leq s\leq t}|\bar{X}^i_s-\hat X^i_s|^2\bigg]&\le C\int_0^t\mathbb E\bigg[\bigg|b(s,\bar{X}^i_s,\mathbb{E}[\bar{b}(s,\bar{X}^i_s)\hat {L}^i_s])-b\bigg(s,\hat X^i_s,\frac{1}{N}\sum_{j=1}^N\bar{b}(s,\hat X^j_s)\hat{L}^j_s\bigg)\bigg|^2\bigg]ds\\
    &\quad+C\int_0^t\mathbb E\bigg[\bigg|\sigma(s,\bar{X}^i_s,\mathbb{E}[\bar{\sigma}(s,\bar{X}^i_s)\hat{L}^i_s])-\sigma\bigg(s,\hat{X}^i_s,\frac{1}{N}\sum_{j=1}^N\bar{\sigma}(s,\hat{X}^j_s)\hat{L}^j_s\bigg)\bigg|^2\bigg]ds.
    \end{split}
    \]
    By the Lipschitz property of $b$ and $\sigma$ in Assumption \ref{assumption:1A}, we find
    \[
    \begin{split}
    \mathbb{E}\bigg[\sup_{0\leq s\leq t}|\bar{X}^i_s-\hat X^i_s|^2\bigg]&\le  C\int_0^t\mathbb{E}[|\bar{X}^i_s-\hat{X}^i_s|^2]ds+C\int_0^t\mathbb{E}\bigg[\bigg|\mathbb{E}[\bar{b}(s,\bar{X}^i_s)\hat{L}^i_s]-\frac{1}{N}\sum_{j=1}^N\bar{b}(s,\hat{X}^j_s)\hat{L}^j_s\bigg|^2\bigg]ds\\
    &\quad +C\int_0^t\mathbb{E}\bigg[\bigg|\mathbb{E}[\bar{\sigma}(s,\bar{X}^i_s)\hat{L}^i_s]-\frac{1}{N}\sum_{j=1}^N\bar{\sigma}(s,\hat{X}^j_s)\hat{L}^j_s\bigg|^2\bigg]ds.
    \end{split}
    \]
    In particular, we have
    \[
    \begin{split}
    &\mathbb{E}\bigg[\bigg|\mathbb{E}[\bar{b}(s,\bar{X}^i_s)\hat{L}^i_s]-\frac{1}{N}\sum_{j=1}^N\bar{b}(s,\hat{X}^j_s)\hat{L}^j_s\bigg|^2\bigg] \\
    &\le2\mathbb{E}\bigg[\bigg|\mathbb{E}[\bar{b}(s,\bar{X}^i_s)\hat{L}^i_s]-\frac{1}{N}\sum_{j=1}^N\bar{b}(s,\bar{X}^j_s)\hat{L}^j_s\bigg|^2\bigg] +2\mathbb{E}\bigg[\bigg|\frac{1}{N}\sum_{j=1}^N\bar{b}(s,\bar{X}^j_s)\hat{L}^j_s-\frac{1}{N}\sum_{j=1}^N\bar{b}(s,\hat{X}^j_s)\hat{L}^j_s\bigg|^2\bigg] \\
    &\le2\frac{\mathbb E[|\bar{b}(s,\bar{X}^i_s)\hat{L}^i_s-\mathbb E[\bar{b}(s,\bar{X}^i_s)\hat{L}^i_s]|^2]}{N}+\frac{C}{N}\sum_{j=1}^N\mathbb{E}\bigg[\bigg|\bar{X}^j_s-\hat{X}^j_s\bigg|^2\bigg],
   \end{split}
    \]
    where in the last inequality we used that $\bar{b}(s,\bar{X}^1_s)\hat{L}^1_s,\ldots,\bar{b}(s,\bar{X}^N_s)\hat{L}^N_s$ are independent and identically distributed. By the standard estimate as in \eqref{EstimateSDE} and the linear growth of $\bar b$, we see that
    \[
    \mathbb E[|\bar{b}(s,\bar{X}^i_s)\hat{L}^i_s-\mathbb E[\bar{b}(s,\bar{X}^i_s)\hat{L}^i_s]|^2] \ \leq \ C\big(1+\mathbb E\big[|X_0|^2\big]\big).
    \]
    Thus
    \[
    \mathbb{E}\bigg[\bigg|\mathbb{E}[\bar{b}(s,\bar{X}^i_s)\hat{L}^i_s]-\frac{1}{N}\sum_{j=1}^N\bar{b}(s,\hat{X}^j_s)\hat{L}^j_s\bigg|^2\bigg] \leq \frac{C}{N}\big(1+\mathbb E\big[|X_0|^2\big]\big) + \frac{C}{N}\sum_{j=1}^N\mathbb{E}\bigg[\bigg|\bar{X}^j_s-\hat{X}^j_s\bigg|^2\bigg].
    \]
    An analogous estimate holds for the term involving $\sigma$. Hence
    \[
    \mathbb{E}\bigg[\sup_{0\leq s\leq t}|\bar{X}^i_s-\hat{X}^i_s|^2\bigg]\le  C\int_0^t\mathbb{E}[|\bar{X}^i_s-\hat{X}^i_s|^2]ds+\frac{C}{N}\big(1+\mathbb E\big[|X_0|^2\big]\big)+\frac{C}{N}\sum_{j=1}^N\int_0^t\mathbb{E}\bigg[\bigg|\bar{X}^j_s-\hat{X}^j_s\bigg|^2\bigg]ds.
    \]
    Let $v_i(t):=\mathbb{E}[\sup_{0\leq s\leq t}|\bar{X}^i_s-\hat{X}^i_s|^2]$. Then, from the latter inequality we obtain
    \[
    v_i(t)\leq C\int_0^t v_i(s) ds + \frac{C}{N}\big(1+\mathbb E\big[|X_0|^{2}\big]\big) + \frac{C}{N}\sum_{j=1}^N\int_0^t v_j(s)ds.
    \]
    Summing over $i$, and multiplying for $1/N$, we get
    \[
    \frac{1}{N}\sum_{i=1}^N v_i(t)\leq \frac{C}{N}\int_0^t \sum_{i=1}^N v_i(s) ds + \frac{C}{N}\big(1+\mathbb E\big[|X_0|^{2}\big]\big) + \frac{C}{N}\sum_{j=1}^N\int_0^t v_j(s)ds.
    \]
    Applying Gronwall's inequality to the function $\frac{1}{N}\sum_{i=1}^N v_i$, we obtain
    \[
    \frac{1}{N}\sum_{i=1}^N\mathbb{E}\bigg[\sup_{0 \le t \le T}\big|\bar{X}^i_t-\hat{X}^i_t\big|^2\bigg]\le \frac{C}{N}\big(1+\mathbb E\big[|X_0|^{2}\big]\big).
    \]
    Since $\mathbb{E}[\sup_{0 \le t \le T}|\bar{X}^1_t-\hat{X}^1_t|^2]=\cdots=\mathbb{E}[\sup_{0 \le t \le T}|\bar{X}^N_t-\hat{X}^N_t|^2]$, we obtain $\mathbb{E}[\sup_{0 \le t \le T}|\bar{X}^i_t-\hat{X}^i_t|^2]\le \frac{C}{N}\big(1+\mathbb E\big[|X_0|^{2}\big]\big)$, for all $i=1,\dots,N$.
\end{proof}

\noindent For the purpose of comparison, when one player $k$ chooses to deviate using a generic strategy $L\in\mathcal V_{\mathbb F^N}$, the dynamics of the state $U^i$ of player $i \in \{1, \dots, N\}$ are given by
\begin{align}
\label{eq:deviation player}
dU^i_t&=b\bigg(t,U^i_t,\frac{1}{N}\bar{b}(t,U^k_t)L_t+\frac{1}{N}\sum_{j=1, j \ne k}^N\bar{b}(t,U^j_t)\hat {L}^j_t\bigg)dt \\
&\quad +\sigma\bigg(t,U^i_t,\frac{1}{N}\bar{\sigma}(t, U^k_t)L_t+\frac{1}{N}\sum_{j=1, j \ne k}^N\bar{\sigma}(t,U^j_t)\hat {L}^j_t\bigg)dW^i_t, \notag
\end{align}
with $t \in [0,T]$ and $U^i_0=X_0^i$.

\begin{lemma}[\textit{Estimate II}]
\label{lemma:barX-U}
    Suppose that Assumptions \ref{assumption:1A},  \ref{assumption:1B} and \ref{assumption: 1C} (or \ref{assumption:1A}, \ref{assum:Tarski_1}.i)-ii) and \ref{assum:Tarski_2}) hold. For $i \in \{1, \dots, N\}$, let $\bar{X}^i$ and $U^i$ be the solutions of (\ref{eq:iid SDE}) and (\ref{eq:deviation player}).
    Then, there exists a constant $C\geq0$, independent of $(\hat L^1,\ldots,\hat L^N)$, $L$, $N$, $i$, such that, for every $i=1,\ldots,N$,
    \[
    \mathbb{E}\bigg[\sup_{0 \le t \le T}|\bar{X}^i_t-U^i_t|^2\bigg]\le \frac{C}{N}\big(1+\mathbb E\big[|X_0|^{2}\big]\big), \qquad \text{for all }i=1,\ldots,N.
    \]
\end{lemma}
\begin{proof}
The proof can be done proceeding along the same lines as in the proof of Lemma \ref{lemma: order barX-X}.
\end{proof}

\noindent In order to construct an $\varepsilon$-Nash equilibrium using the equilibrium strategy of the limiting mean field game, we impose an additional set of assumptions.
\begin{assumption}\label{ass:N_Players}
There exists a constant $K\geq0$ such that
\begin{align*}
|f(t,x,m)-f(t,x',m')|&\le K\big(1+|x|+|x'|\big)\big(|x-x'|+|m-m'|\big), \\
|h(t,x,w)-h(t,x',w')|&\le K\big(1+|x|+|x'|\big)\big(|x-x'|+|w-w'|\big),
\end{align*}    
for all $t \in [0,T]$, $(x,m,w),(x',m',w')\in \mathbb{R}^d\times \mathbb{R}^k\times \mathbb{R}$.
\end{assumption}
We now prove the main Theorem of this Section.
\begin{theorem}[\textit{Approximate Nash equilibria for the $N$-player game}]\label{T:Approximate}
Suppose that Assumptions \ref{assumption:1A}, \ref{assumption:1B}, \ref{assumption: 1C} and  \ref{ass:N_Players} (or \ref{assumption:1A}, \ref{assum:Tarski_1}.i)-ii), \ref{assum:Tarski_2} and \ref{ass:N_Players}) hold. Let $\hat{L}^i\in\mathcal V_{\mathbb F^N}$ be given by \eqref{hatL^i} and let $\hat{\boldsymbol{L}}=(\hat{L}^1, \dots, \hat{L}^N)$. Then for any $\varepsilon>0$, there exists an integer $N_{\varepsilon}$ such that for all $N \ge N_{\varepsilon}$, $\hat{\boldsymbol{L}}$ is an $\varepsilon$-Nash equilibrium for the $N$-player game \eqref{payoff N-player}-\eqref{N-player SDE_0}. That is, for every player $i \in \{1, \dots, N\}$ and every $L \in \mathcal V_{\mathbb F^N}$,
    \begin{equation}
    \label{eq: epsilon Nash}
    \Lambda^{N,i}(\hat{L}^1, \dots, \hat{L}^{i-1}, \hat{L}^i, \hat{L}^{i+1},\dots, \hat{L}^N)\ge \Lambda^{N,i}(\hat{L}^1, \dots, \hat{L}^{i-1}, L, \hat{L}^{i+1}, \dots, \hat{L}^N)-\varepsilon.
    \end{equation}
\end{theorem}
\begin{proof}
By symmetry of the game, we need to prove (\ref{eq: epsilon Nash}) only for $i=1$. For each Brownian motion $W^i$, we consider the following forward-backward system:
    \begin{equation}
        \begin{cases}
        d\bar{X}^i_t=b\big(t,\bar{X}^i_t, \mathbb{E}[\bar{b}(t,\bar{X}^i_t)\hat L_t^i]\big)dt+\sigma\big(t,\bar{X}^i_t, \mathbb{E}[\bar{\sigma}(t,\bar{X}^i_t)\hat L_t^i]\big)dW^i_t,\\
        -d\bar{Y}^i_t=f(t,\bar{X}^i_t,\mathbb{E}[\bar{f}(t,\bar{X}^i_t)\hat L_t^i])dt+d\bar{A}^i_t-\bar{Z}^i_tdW^i_t,\\
        \bar X_0^i=X_0^i, \quad \bar Y_T^i=h\big(T,\bar{X}^i_T,\mathbb{E}[\int_{0^-}^T\phi(T-s)d\hat L_s^i]\big), \\
        \bar{Y}^i_t \ge h\big(t,\bar{X}^i_t,\mathbb{E}[\int_{0^-}^T\phi(t-s)d\hat L_s^i]\big), \quad 0 \le t \le T,\\
        \int_0^T\big(\bar{Y}^i_t-h\big(t,\bar{X}^i_t, \mathbb{E}[\int_{0^-}^T\phi(t-s)d\hat L_s^i]\big)\big)d\bar{A}^i_t=0.
        \end{cases}
    \end{equation}
    We observe that $(\bar{X}^1, \hat{L}^1),\ldots,(\bar{X}^N, \hat{L}^N)$ are independent and identically distributed.
    We denote by $\Lambda$ the optimal payoff of the limiting problem, that is $\Lambda:=\Lambda(\hat L)$, see \eqref{Lambda} (notice that in $\Lambda(\hat L)$ appears $\hat X$ in place $X$, where $\hat X$ is the solution of \eqref{MFG SDE} with $\hat L$ in place of $L$). Since $(\bar{X}^1, \hat{L}^1)$ has the same distribution as $(\hat X,\hat{L})$, we find
    \[ \Lambda=\mathbb{E}\bigg[\int_0^Tf(t,\bar{X}^1_t,\mathbb{E}[\bar{f}(t,\bar{X}^1_t)\hat{L}^1_t])\hat{L}^1_tdt-\int_{0^-}^Th\bigg(t,\bar{X}^1_t,\mathbb{E}\bigg[\int_{0^-}^T\phi(t-s)d\hat{L}^1_s\bigg]\bigg)d\hat{L}^1_t\bigg].
    \]
     We will prove the following two properties.
     \begin{enumerate}[1)]
    \item $\lim_{N \to \infty} \Lambda^{N,1}(\hat{\boldsymbol{L}})= \Lambda$.
    \item It holds that $\lim_{N \to \infty} \sup_{L \in \mathcal{V}_{\mathbb F^N}}|\Lambda^{N,1}(L \otimes \hat{\boldsymbol{L}}^{-1})-\bar\Lambda(\hat L^1,L)|=0$,
    with $\bar\Lambda(\cdot,\cdot)$ defined as in \eqref{Reward}.
    \end{enumerate}
    \textsc{Step 1.} By Lemma \ref{lemma: order barX-X}, we get $\mathbb{E}[\sup_{0 \le t \le T}|\bar{X}^1_t-\hat X^1_t|^2] \longrightarrow 0$ as $N \to \infty$. Using the locally Lipschitz continuity of the coefficients $f$ and $h$, together with the Cauchy-Schwarz inequality, we get (in the sequel we denote by $C$ a non-negative constant, independent of $(\hat L^1,\ldots,\hat L^N)$, $N$, which may change from line to line)
    \begin{align}
    \label{eq:step1}
        &|\Lambda-\Lambda^{N,1}(\hat{\boldsymbol{L}})|\le\mathbb E\bigg[\int_0^T\bigg|\bigg(f(t,\bar{X}^1_t, \mathbb{E}[\bar{f}(t,\bar{X}^1_t)\hat{L}^1_t])-f\bigg(t,\hat X^1_t,\frac{1}{N}\sum_{j=1}^N\bar{f}(t,\hat X^j_t)\hat{L}^j_t\bigg)\bigg)\hat{L}^1_t\bigg|dt\notag\\
        &\quad+\int_{0^-}^T\bigg|\bigg(h\bigg(t,\bar{X}^1_t, \mathbb{E}\bigg[\int_{0^-}^T\phi(t-s)d\hat{L}^1_s\bigg]\bigg)-h\bigg(t,\hat X^1_t,\frac{1}{N}\sum_{j=1}^N\int_{0^-}^T\phi(t-s)d\hat{L}^j_s\bigg)\bigg)\bigg|d|\hat{L}^1|_t\bigg]\notag\\
        &\le C\int_0^T\mathbb{E}\bigg[\bigg(1+|\bar{X}^1_t|^2+|\hat X^1_t|^2\bigg)\bigg]^{1/2}\mathbb{E}\bigg[|\bar{X}^1_t-\hat X^1_t|^2+\bigg|\mathbb{E}[\bar{f}(t,\bar{X}^1_t)\hat{L}^1_t]-\frac{1}{N}\sum_{j=1}^N\bar{f}(t,\hat X^j_t)\hat{L}^j_t\bigg|^2\bigg]^{1/2}dt\notag\\
        &\quad+C\,\mathbb E\bigg[\int_{0^-}^T\bigg(1+|\bar{X}^1_t|^2+|\hat X^1_t|^2\bigg)d|\hat{L}^1|_t\bigg]^{1/2}\mathbb E\bigg[\int_{0^-}^T\bigg(|\bar{X}^1_t-\hat X^1_t|^2+\bigg|\mathbb{E}\bigg[\int_{0^-}^T\phi(t-s)d\hat{L}^1_s\bigg]\notag\\
        &\quad-\frac{1}{N}\sum_{j=1}^N\int_{0^-}^T\phi(t-s)d\hat{L}^j_s\bigg|^2\bigg)d|\hat{L}^1|_t\bigg]^{1/2}.
    \end{align}
    Proceeding along the same lines as in the proof of Lemma \ref{lemma: order barX-X}, we can show that
    \[
    \int_0^T\mathbb{E}\bigg[|\bar{X}^1_t-\hat X^1_t|^2+\bigg|\mathbb{E}[\bar{f}(t,\bar{X}^1_t)\hat{L}^1_t]-\frac{1}{N}\sum_{j=1}^N\bar{f}(t,\hat X^j_t)\hat{L}^j_t\bigg|^2\bigg]^{1/2}dt \longrightarrow 0
    \]
    and
    \[
    \mathbb E\bigg[\int_{0^-}^T\bigg(|\bar{X}^1_t-\hat X^1_t|^2+\bigg|\mathbb{E}\bigg[\int_{0^-}^T\phi(t-s)d\hat{L}^1_s\bigg]-\frac{1}{N}\sum_{j=1}^N\int_{0^-}^T\phi(t-s)d\hat{L}^j_s\bigg|^2\bigg)d|\hat{L}^1|_t\bigg] \longrightarrow 0.
    \]
    This concludes the proof of property 1).
    \medskip
    \\
    \textsc{Step 2.} We now suppose that player 1 deviates by choosing a generic strategy $L \in \mathcal{V}_{\mathbb F^N}$, while each other players $i=2,\ldots,N$ adopt $\hat{L}^i$. Let 
    \begin{align*}
    dU^i_t&=b\bigg(t,U^i_t,\frac{1}{N}\bar{b}(t,U^1_t)L_t+\frac{1}{N}\sum_{j=2}^N\bar{b}(t,U^j_t)\hat{L}^j_t\bigg)dt \\
    &\quad+\sigma\bigg(t,U^i_t,\frac{1}{N}\bar{\sigma}(t,U^1_t)L_t+\frac{1}{N}\sum_{j=2}^N\bar{\sigma}(t,U^j_t)\hat{L}^j_t\bigg)dW^i_t,
    \end{align*}
    with $U^i_0=X_0^i$. By Lemma \ref{lemma:barX-U} we have that $\mathbb{E}[\sup_{0 \le t \le T}|\bar{X}^i_t-U^i_t|^2]\le \frac{C}{N}\big(1+\mathbb E\big[|X_0|^{2}\big]\big)$.\\
    Thus, proceeding along the same lines as in (\ref{eq:step1}), we get
    \[
    \lim_{N \to \infty} \sup_{L \in \mathcal{V}_{\mathbb F^N}}\big|\Lambda^{N,1}(L \otimes \hat{\boldsymbol{L}}^{-1})-\bar\Lambda(\hat L^1,L)\big| = 0,
    \]
    with $\bar\Lambda(\cdot,\cdot)$ defined as in \eqref{Reward}.

    \vspace{2mm}
    
    \noindent Given $\varepsilon>0$, thanks to Step 2, there exists $N_\varepsilon\in\mathbb N$ sufficiently large such that, for any $L \in \mathcal{V}_{\mathbb F^N}$,
    \[
    \Lambda^{N,1}(L\otimes\hat{\boldsymbol{L}}^{-1}) \le \bar\Lambda(\hat L^1,L) + \frac{\varepsilon}{2} \le \Lambda+\frac{\varepsilon}{2}, \qquad \text{for all } N > N_\varepsilon,
    \]
    where the last inequality follows from the optimality of $\hat{L}^1$ and also from Lemma \ref{L:IndepEnlarg} with $\mathcal H$ equal to $\sigma(X_0^j,(W_s^j)_{0\leq s\leq T},j\neq1)$. By Step 1, we have, possibly enlarging $N_\varepsilon$,
    \[
    \Lambda^{N,1}(\hat{\boldsymbol{L}})\ge \Lambda-\frac{\varepsilon}{2}, \qquad \text{for all } N > N_\varepsilon.
    \]
    Combining these two inequalities, we get that, for all $L \in \mathcal{V}_{\mathbb F^N}$,
    \[
    \Lambda^{N,1}(\hat{\boldsymbol{L}})\ge \Lambda-\frac{\varepsilon}{2}\ge \Lambda^{N,1}(L\otimes\hat{\boldsymbol{L}}^{-1})-\varepsilon, \qquad \text{for all } N > N_\varepsilon.
    \]
    \end{proof}

\section{Relation between MKV-RFBSDEs and the PDE\\ approach}
\label{S:PDE}
In this Section, we show a rigorous relation between the system \eqref{MFG SDE}-\eqref{MFRBSDE} and the system of partial differential equations introduced in \cite{bertucci2018optimal}.\\
We suppose that $b=b(t,x)$, $\sigma=\sigma(t,x)$ and $h=h(t,x)$ are independent of their last argument and $h \in W^{1,2}([0,T]\times \mathbb R^d)$ \footnote{We denote with $W^{1,2}([0,T] \times \mathbb R^d)$ the Sobolev space of functions in $L^2([0,T]\times \mathbb R^d)$ with first-order weak derivative in time and second-order weak derivative in space in $L^2([0,T]\times \mathbb R^d)$.}. We also suppose that $k=d$ and $\bar f(t,x)=x$, for every $(t,x)\in[0,T]\times\mathbb R^d$. Since $g(x,w)=h(T,x,w)$, the terminal condition is simply $g(x)=h(T,x)$. Let $(X,Y,Z,A,L) \in \mathbb{S}^2\times\mathbb{S}^2\times\mathbb{H}^2\times\mathbb{A}^2\times\mathcal{V}$ be a solution to system \eqref{MFG SDE}-\eqref{MFRBSDE}. For every $t\in(0,T]$, define the finite measure $m_t$ on $\mathbb R^d$ by
\[
m_t(B) \coloneq \mathbb E\big[\mathbf 1_B(X_t)L_t\big], \qquad \text{for every Borel subset $B$ of }\mathbb R^d, 
\]
and $m_0:=\mu_0=\text{Law}(X_0)$. Set
\[
\bar m(t) \coloneq \mathbb{E}[X_tL_t]=\int_{\mathbb{R}^d}x\,m_t(dx), \qquad t \in [0,T].
\]
In particular, $m_t$ is a sub-probability measure on $\mathbb R^d$.\\
For every $t\in[0,T]$, we denote by $\mathbb{S}_t^2$, $\mathbb{H}_t^2$, $\mathbb{A}_t^2$ the obvious counterparts of $\mathbb{S}^2$, $\mathbb{H}^2$, $\mathbb{A}^2$ on the interval $[t,T]$. For every $(t,x)\in [0,T]\times\mathbb R^d$, let $X^{t,x}$ be the unique $\mathbb{R}^d$-valued solution of the stochastic differential equation
\[
X^{t,x}_s = x+\int_t^sb(r,X^{t,x}_r)dr+\int_t^s\sigma(r,X^{t,x}_r)dW_r, \qquad t\leq s\leq T.
\]
We denote by $\mathcal L$ the second-order differential operator
\[
(\mathcal L u)(t,x) = \sum_{i=1}^db_i(t,x)\frac{\partial u}{\partial x_i}(t,x)+\frac{1}{2}\sum_{i,j=1}^d a_{i,j}(t,x)\frac{\partial^2 u}{\partial x_i\partial x_j}(t,x),
\]
for $u\in \mathcal{C}^{1,2}([0,T]\times\mathbb{R}^d)$, where we set $a=\sigma\sigma^\top$. Its formal adjoint is
\[
\mathcal{L}^* \mu \coloneq -\sum_{i=1}^d \partial_{x_i}\big(b_i \mu\big) + \frac{1}{2}\sum_{i,j=1}^d\partial^2_{x_i x_j}\big(a_{i,j}\mu\big).
\]
For every $(t,x)\in [0,T]\times\mathbb R^d$, we consider the reflected backward stochastic differential equation on $[t,T]$:
\begin{equation}\label{MFRBSDE-PDE}
\begin{cases}
\vspace{2mm}Y^{t,x}_s=h(T,X^{t,x}_T)+\int_s^T f\big(r,X^{t,x}_r,\bar m(r)\big)dr + A^{t,x}_T-A^{t,x}_s - \int_s^T(Z^{t,x}_r,dW_r), \quad t\leq s\leq T, \\
\vspace{2mm}Y^{t,x}_s \geq h(s, X^{t,x}_s), \quad t\leq s\leq T, \\
\vspace{2mm}\int_{t}^T (Y^{t,x}_s-h(s,X^{t,x}_s)) dA^{t,x}_s=0.
\end{cases}
\end{equation}
Notice that under Assumptions \ref{assumption:1A} and \ref{assumption:1B}, it follows that there exists $(Y^{t,x}, Z^{t,x}, A^{t,x}) \in \mathbb{S}_t^2 \times \mathbb{H}_t^2 \times \mathbb{A}_t^2$ solution of \eqref{MFRBSDE-PDE}. We then define
\[
v(t,x) \coloneq Y^{t,x}_t, \qquad (t,x)\in [0,T]\times \mathbb{R}^d.
\]

\begin{theorem}
Assumptions \ref{assumption:1A}, \ref{assumption:1B} and \ref{assumption: 1C} hold. Suppose also that Assumptions \ref{assum:uniq_opt_stop}.a, \ref{assum:uniq_opt_stop}.b and \ref{assum:uniq_opt_stop}.c hold. \\
Then the pair $(v,m)$ has the following properties:
    \begin{itemize}
        \item [i)] The function $v$ is a continuous viscosity solution of $\min (-\partial_tv-\mathcal{L}v-f(t,x,\bar m(t)), v-h)=0$; 
        \item [ii)] $\partial_tm-\mathcal{L}^\star m \le 0$, in $\mathcal{D}'((0,T)\times \mathbb{R}^d)$, and $m(0)=\mu_0$, where $\mu_0$ is the law of $X_0$;
        \item [iii)] $v=h$, at $t=T$;
        \item [iv)] $\partial_tm-\mathcal{L}^\star m=0$, in $\mathcal{D}'(\{v >h\})$;
        \item [v)] $\int_{\{v=h\}}(f(t,x,\bar m(t))+(\partial_t+ \mathcal{L})h(t,x))m_t(dx)dt=0$.
    \end{itemize}
\end{theorem}
\begin{proof}
Item $iii)$ is immediate from the terminal condition in \eqref{MFRBSDE-PDE}, item $i)$ follows from Lemma 8.4 and Theorem 8.5 in \cite{el1981aspects}. Moreover, by the classical theory of Markovian reflected BSDEs (see, e.g., \cite{el1981aspects}), it holds that
\[
Y_s^{t,x}=v(s,X_s^{t,x}), \qquad t\leq s\leq T.
\]
Since $b=b(t,x)$, $\sigma=\sigma(t,x)$, then $X$ corresponds to $X^{0,X_0}$. From this identification, we also obtain
\[
Y_t=v(t,X_t), \qquad 0\leq t\leq T.
\]
We now show item $ii)$. Let $u \in \mathcal{C}^{1,2}_b([0,T]\times \mathbb{R}^d)$ be non-negative. Applying It\^o's formula to $u(t,X_t)L_t$, we find
\[
d(u(t,X_t)L_t)=u(t,X_t)dL_t+L_{t^-}\big((\partial_tu+\mathcal{L}u)dt+(\nabla u\sigma,dW_t)\big).
\]
Integrating from 0 to $T$ and taking the expectation, we get
\begin{align}\label{eqq}
\mathbb{E}[u(T,X_T)L_T]-\mathbb{E}[u(0,X_0)L_{0^-}]=\mathbb{E}\bigg[\int_{0^-}^Tu(t,X_t)dL_t\bigg]+\mathbb{E}\bigg[\int_0^T(\partial_tu+\mathcal{L}_tu)(t,X_t)L_tdt\bigg].
\end{align}
Recalling that $L_T=0$, $L_{0^-}=1$, $-dL$ is a non-negative measure, and $u$ is non-negative, we obtain
\[
0 \le \mathbb{E}[u(0,X_0)]+\mathbb{E}\bigg[\int_0^T(\partial_tu+\mathcal{L}u)(t,X_t)L_tdt\bigg].
\]
This can be rewritten as
\[
0 \le \int_{\mathbb{R}^d}u(0,x)\mu_0(dx)+\int_0^T\int_{\mathbb{R}^d}(\partial_tu+\mathcal{L}u)(t,x)m_t(dx)dt.
\]
This shows that $\partial_tm-\mathcal{L}^\star m \le 0$, in $\mathcal{D}^{\prime}((0,T) \times \mathbb{R}^d)$.\\
Let us now show item $iv)$. Let $u \in \mathcal{C}^{1,2}_b([0,T]\times\mathbb{R}^d)$ with support in $\{v>h\}$, so that $\{u\neq 0\}$ and $\{v=h\}$ are disjoint sets. By the Skorokhod condition $\int_{0^-}^T(Y_t-h(t,X_t))dL_t=0$ it follows that the topological support of the random measure $-dL_t$ is contained almost surely in $\{t\colon Y_t=h(t,X_t)\}$. Then, we get that, almost surely, $\{t:u(t,X_t) \neq 0\}$ and the topological support of $-dL_t$ are disjoint sets. Hence, by the same computation  as in \eqref{eqq}, we obtain
\[
0=\mathbb{E}[u(0,X_0)]+\mathbb{E}\bigg[\int_0^T(\partial_tu+\mathcal{L}u)(t,X_t)L_tdt\bigg].
\]
This can be written as
\[
0=\int_{\mathbb{R}^d}u(0,x)\mu_0(dx)+\int_0^T\int_{\mathbb{R}^d}(\partial_tu+\mathcal{L}u)(t,x)m_t(dx)dt.
\]
We therefore deduce $\partial_tm-\mathcal{L}^\star m =0$ in $\mathcal{D}'(\{v >h\})$.\\
Finally, we prove item $v)$. To alleviate the notation, we denote $h(t,X_t)$ simply by $h_t$, for $t\in[0,T]$. Applying It\^o's formula to $(Y-h)L$, we get
\begin{equation}
\label{eq:eq1}
\begin{split}
    d\big((Y_t-h_t)L_t\big)&=L_{t^{-}} d(Y_t-h_t)+(Y_t-h_t)dL_t\\
    &=-\big(f(t,X_t, \mathbb{E}[X_tL_t])+\partial_th(t,X_t)+\mathcal{L}h(t,X_t)\big)L_tdt\\
    &\quad+L_{t^{-}}\big(Z_t-\nabla h(t,X_t)\sigma(t,X_t),dW_t\big)-L_{t^{-}}dA_t+(Y_t-h_t)dL_t.
\end{split}
\end{equation}
Now, applying Tanaka-Meyer's formula to $(Y_t-h_t)^+L_t$, we obtain 
\begin{equation}
\begin{split}
\label{eq:eq2}
    d\big((Y_t-h_t)^+L_t\big)&=L_{t^{-}} d(Y_t-h_t)^++(Y_t-h_t)^+dL_t\\
    &=-\mathbf{1}_{Y_t>h_t}\big(f(t,X_t, \mathbb{E}[X_tL_t])+\partial_th(t,X_t)+\mathcal{L}h(t,X_t)\big)L_tdt\\
    &\quad+\mathbf{1}_{Y_t>h_t}L_{t^{-}}\big(Z_t-\nabla h(t,X_t)\sigma(t,X_t),dW_t\big)\\
    &\quad-\mathbf{1}_{Y_t > h_t}L_{t^{-}}dA_t+\frac{1}{2}L_{t^{-}}d\tilde{L}_t^0(Y_t-h_t)+(Y_t-h_t)^+dL_t.
\end{split}
\end{equation}
Observe that
\[
\int_0^T\mathbf{1}_{Y_t=h_t}L_{t^-}dA_t=\int_0^TL_{t^-}dA_t-\int_0^T\mathbf{1}_{Y_t>h_t}L_{t^-}dA_t=0
\]
The first term above vanishes, in fact, using the Skorokhod condition $\int_{0^-}^TA_tdL_t=0$, together with $L_T=0$, $A_0=0$, then from the integration by parts formula we get $0=0+\int_0^TL_{t^-}dA_t$. The second term also vanishes as a consequence of the fact that, from the Skorokhod condition $\int_0^T(Y_t-h_t)dA_t=0$ it follows that the topological support of the random measure $dA_t$ is contained almost surely in $\{t\colon Y_t = h(t,X_t)\}$. Since $Y_t-h_t \equiv (Y_t-h_t)^+$, by combining the two above equalities \eqref{eq:eq1} and \eqref{eq:eq2}, we get
\begin{equation}
\begin{split}
    \mathbb{E}\bigg[\int_0^T \mathbf{1}_{Y_t=h_t}\big(f(t,X_t, \mathbb{E}[X_tL_t])+\partial_th(t,X_t)
    &+\mathcal{L}h(t,X_t)\big)L_tdt\\
    &+\frac{1}{2}\int_0^TL_{t^-}d\tilde{L}^0_t(Y_t-h_t)\bigg]=0.
\end{split}
\end{equation}
By Theorem 6 in \cite{jacka1993local}, the local time $\tilde{L}^0(Y-h)$ is indistinguishable from zero, so the second integral in the above expectation is zero.
Thus, it follows that
\[
\mathbb{E}\bigg[\int_0^T\mathbf{1}_{Y_t=h_t}\big(f(t,X_t,\mathbb{E}[X_tL_t])+\partial_th(t,X_t)+\mathcal{L}h(t,X_t)\big)L_tdt\bigg]=0.
\]
\end{proof}

\small
\bibliographystyle{plain}
\bibliography{references.bib}
\end{document}